\documentclass[a4paper, reqno]{amsart}

\usepackage{amsthm,amsfonts,amssymb,amsmath,amsxtra}
\usepackage[all]{xy}
\SelectTips{cm}{}
\usepackage{tikz}
\usepackage{verbatim}

\usepackage[margin=1.25in]{geometry}
\usepackage{mathrsfs}

\RequirePackage{xspace}
\RequirePackage{etoolbox}
\RequirePackage{varwidth}
\RequirePackage{enumitem}
\RequirePackage{tensor}
\RequirePackage{mathtools}
\RequirePackage{longtable}
\RequirePackage{multirow}
\RequirePackage{makecell}
\RequirePackage{bigints}
\RequirePackage{xcolor}

\usepackage[backref=page]{hyperref} %

\newcommand{\sH}{\ensuremath{\mathscr{H}}\xspace}

\newcommand{\sV}{\ensuremath{\mathscr{V}}\xspace}

\newcommand{\fka}{\ensuremath{\mathfrak{a}}\xspace}
\newcommand{\fkb}{\ensuremath{\mathfrak{b}}\xspace}

\newcommand{\fkd}{\ensuremath{\mathfrak{d}}\xspace}

\newcommand{\fkg}{\ensuremath{\mathfrak{g}}\xspace}
\newcommand{\fkh}{\ensuremath{\mathfrak{h}}\xspace}

\newcommand{\fkp}{\ensuremath{\mathfrak{p}}\xspace}
\newcommand{\fkq}{\ensuremath{\mathfrak{q}}\xspace}

\newcommand{\BA}{\ensuremath{\mathbb{A}}\xspace}

\newcommand{\BC}{\ensuremath{\mathbb{C}}\xspace}

\newcommand{\BF}{\ensuremath{\mathbb{F}}\xspace}
\newcommand{\BG}{\ensuremath{\mathbb{G}}\xspace}

\newcommand{\BO}{\ensuremath{\mathbb{O}}\xspace}
\newcommand{\BP}{\ensuremath{\mathbb{P}}\xspace}
\newcommand{\BQ}{\ensuremath{\mathbb{Q}}\xspace}
\newcommand{\BR}{\ensuremath{\mathbb{R}}\xspace}

\newcommand{\BX}{\ensuremath{\mathbb{X}}\xspace}

\newcommand{\BZ}{\ensuremath{\mathbb{Z}}\xspace}

\newcommand{\bA}{\ensuremath{\mathbf{A}}\xspace}

\newcommand{\bm}{\ensuremath{\mathbf{m}}\xspace}

\newcommand{\CA}{\ensuremath{\mathcal{A}}\xspace}

\newcommand{\CD}{\ensuremath{\mathcal{D}}\xspace}

\newcommand{\CF}{\ensuremath{\mathcal{F}}\xspace}

\newcommand{\CH}{\ensuremath{\mathcal{H}}\xspace}
\newcommand{\CI}{\ensuremath{\mathcal{I}}\xspace}

\newcommand{\CL}{\ensuremath{\mathcal{L}}\xspace}
\newcommand{\CM}{\ensuremath{\mathcal{M}}\xspace}
\newcommand{\CN}{\ensuremath{\mathcal{N}}\xspace}
\newcommand{\CO}{\ensuremath{\mathcal{O}}\xspace}
\newcommand{\CP}{\ensuremath{\mathcal{P}}\xspace}

\newcommand{\CR}{\ensuremath{\mathcal{R}}\xspace}

\newcommand{\CV}{\ensuremath{\mathcal{V}}\xspace}

\newcommand{\CX}{\ensuremath{\mathcal{X}}\xspace}

\newcommand{\CZ}{\ensuremath{\mathcal{Z}}\xspace}

\newcommand{\uF}{\underline{F}}

\newcommand{\nat}{{\natural}}

\newcommand{\ad}{\mathrm{ad}}

\newcommand{\AT}{\mathrm{AT}}
\DeclareMathOperator{\Aut}{Aut}

\newcommand{\Ch}{{\mathrm{Ch}}}
\DeclareMathOperator{\charac}{char}

\newcommand{\cl}{{\mathrm{cl}}}

\newcommand{\corr}{\mathrm{corr}}
\newcommand{\crys}{\mathrm{crys}}

\newcommand{\del}{\operatorname{\partial Orb}}
\newcommand{\delJ}{\partial J}
\DeclareMathOperator{\diag}{diag}
\DeclareMathOperator{\disc}{disc}

\DeclareMathOperator{\End}{End}

\newcommand{\Fil}{\ensuremath{\mathrm{Fil}}\xspace}
\newcommand{\fin}{\mathrm{fin}}

\DeclareMathOperator{\Gal}{Gal}

\newcommand{\GL}{\mathrm{GL}}

\newcommand{\GS}{\mathrm{GS}}

\newcommand{\GU}{\mathrm{GU}}

\newcommand{\HG}{H\! G}

\DeclareMathOperator{\Hom}{Hom}

\newcommand{\id}{\ensuremath{\mathrm{id}}\xspace}
\let\Im\relax

\DeclareMathOperator{\Im}{Im}
\renewcommand{\i}{^{-1}}
\newcommand{\Ind}{{\mathrm{Ind}}}
\DeclareMathOperator{\Int}{\ensuremath{\mathrm{Int}}\xspace}
\newcommand{\inv}{{\mathrm{inv}}}
\DeclareMathOperator{\Isom}{Isom}

\DeclareMathOperator{\Lie}{Lie}

\DeclareMathOperator{\Nm}{Nm}

\DeclareMathOperator{\Orb}{Orb}
\DeclareMathOperator{\ord}{ord}

\DeclareMathOperator{\Ros}{Ros}

\DeclareMathOperator{\Res}{Res}
\DeclareMathOperator{\reschar}{reschar}
\newcommand{\rs}{\ensuremath{\mathrm{rs}}\xspace}

\newcommand{\Sh}{\mathrm{Sh}}
\newcommand{\sig}{{\mathrm{sig}}}

\newcommand{\SL}{{\mathrm{SL}}}

\DeclareMathOperator{\Spec}{Spec}
\DeclareMathOperator{\Spf}{Spf}

\newcommand{\SO}{{\mathrm{SO}}}

\newcommand{\spl}{\mathrm{spl}}
\newcommand{\ssm}{\smallsetminus}

\DeclareMathOperator{\supp}{supp}
\newcommand{\surj}{\twoheadrightarrow}

\DeclareMathOperator{\sgn}{sgn}

\newcommand{\temp}{\mathrm{temp}}
\DeclareMathOperator{\tr}{tr}

\newcommand{\U}{\mathrm{U}}

\DeclareMathOperator{\uHom}{\underline{Hom}}
\newcommand{\un}{\mathrm{un}}

\DeclareMathOperator{\vol}{vol}

\newcommand{\wtHG}{\ensuremath{{\widetilde{{H\! G}}}}\xspace}

\newcommand{\wt}{\widetilde}
\newcommand{\wh}{\widehat}

\newcommand{\ov}{\overline}
\newcommand{\incl}{\hookrightarrow}
\newcommand{\lra}{\longrightarrow}

\newcommand{\bs}{\backslash}
\newcommand{\la}{\langle}
\newcommand{\ra}{\rangle}
\newcommand{\lv}{\lvert}
\newcommand{\rv}{\rvert}

\newtheorem{theorem}{Theorem}
\newtheorem{proposition}[theorem]{Proposition}
\newtheorem{lemma}[theorem]{Lemma}

\newtheorem{conjecture}[theorem]{Conjecture}
\newtheorem{`conjecture'}[theorem]{``Conjecture''}

\theoremstyle{definition}
\newtheorem{definition}[theorem]{Definition}

\newtheorem{remark}[theorem]{Remark}

\newenvironment{altenumerate}
   {\begin{list}
      {(\theenumi) }
      {\usecounter{enumi}
       \setlength{\labelwidth}{0pt}
       \setlength{\labelsep}{0pt}
       \setlength{\leftmargin}{0pt}
       \setlength{\itemsep}{\the\smallskipamount}
       \renewcommand{\theenumi}{\roman{enumi}}
      }}
   {\end{list}}

\newenvironment{altitemize}
   {\begin{list}
      {$\bullet$}
      {\setlength{\labelwidth}{0pt}
	   \setlength{\itemindent}{5pt}
       \setlength{\labelsep}{5pt}
       \setlength{\leftmargin}{0pt}
       \setlength{\itemsep}{\the\smallskipamount}
      }}
   {\end{list}}

\numberwithin{equation}{section}
\numberwithin{theorem}{section}

\newcommand{\aform}{\ensuremath{\langle\text{~,~}\rangle}\xspace}
\newcommand{\sform}{\ensuremath{(\text{~,~})}\xspace}

\newcounter{variant}

\setitemize[0]{leftmargin=*,itemsep=\the\smallskipamount}
\setenumerate[0]{leftmargin=*,itemsep=\the\smallskipamount}

\renewcommand{\to}{%
   \ifbool{@display}{\longrightarrow}{\rightarrow}%
   }
\let\shortmapsto\mapsto
\renewcommand{\mapsto}{%
   \ifbool{@display}{\longmapsto}{\shortmapsto}%
   }
\newcommand{\hooklongrightarrow}{\mathrel{\mkern 0.5mu\lhook\mkern -3.5mu\relbar\mkern -3mu \rightarrow }}
\newcommand{\inj}{%
   \ifbool{@display}{\hooklongrightarrow}{\hookrightarrow}
   }
\newcommand{\isoarrow}{%
   \ifbool{@display}{\overset{\sim}{\longrightarrow}}{\xrightarrow\sim}%
   }
\newlength{\olen}
\newlength{\ulen}
\newlength{\xlen}
\newcommand{\xra}[2][]{%
   \ifbool{@display}%
      {\settowidth{\olen}{$\overset{#2}{\longrightarrow}$}%
       \settowidth{\ulen}{$\underset{#1}{\longrightarrow}$}%
       \settowidth{\xlen}{$\xrightarrow[#1]{#2}$}%
       \ifdimgreater{\olen}{\xlen}%
          {\underset{#1}{\overset{#2}{\longrightarrow}}}%
          {\ifdimgreater{\ulen}{\xlen}%
             {\underset{#1}{\overset{#2}{\longrightarrow}}}
             {\xrightarrow[#1]{#2}}}}%
      {\xrightarrow[#1]{#2}}
   }
\makeatother
\newcommand{\xyra}[2][]{%
   \settowidth{\xlen}{$\xrightarrow[#1]{#2}$}%
   \ifbool{@display}%
      {\settowidth{\olen}{$\overset{#2}{\longrightarrow}$}%
       \settowidth{\ulen}{$\underset{#1}{\longrightarrow}$}%
       \ifdimgreater{\olen}{\xlen}%
          {\mathrel{\xymatrix@M=.12ex@C=3.2ex{\ar[r]^-{#2}_-{#1} &}}}%
          {\ifdimgreater{\ulen}{\xlen}%
             {\mathrel{\xymatrix@M=.12ex@C=3.2ex{\ar[r]^-{#2}_-{#1} &}}}
             {\mathrel{\xymatrix@M=.12ex@C=\the\xlen{\ar[r]^-{#2}_-{#1} &}}}}}%
      {\mathrel{\xymatrix@M=.12ex@C=\the\xlen{\ar[r]^-{#2}_-{#1} &}}}%
   }
\makeatletter
\newcommand{\xla}[2][]{%
   \ifbool{@display}%
      {\settowidth{\olen}{$\overset{#2}{\longleftarrow}$}%
       \settowidth{\ulen}{$\underset{#1}{\longleftarrow}$}%
       \settowidth{\xlen}{$\xleftarrow[#1]{#2}$}%
       \ifdimgreater{\olen}{\xlen}%
          {\underset{#1}{\overset{#2}{\longleftarrow}}}%
          {\ifdimgreater{\ulen}{\xlen}%
             {\underset{#1}{\overset{#2}{\longleftarrow}}}
             {\xleftarrow[#1]{#2}}}}%
      {\xleftarrow[#1]{#2}}
   }
\renewcommand{\lra}{%
   \ifbool{@display}{\longleftrightarrow}{\leftrightarrow}%
   }
\newcommand{\undertilde}{\raisebox{0.4ex}{\smash[t]{$\scriptstyle\sim$}}}

\begin{document}

\title{Arithmetic diagonal cycles on unitary Shimura varieties}
\author{M. Rapoport}
\address{Mathematisches Institut der Universit\"at Bonn, Endenicher Allee 60, 53115 Bonn, Germany, and University of Maryland, Department of Mathematics, College Park, MD 20742, USA}
\email{rapoport@math.uni-bonn.de}
\author{B. Smithling}
\address{University of Maryland, Department of Mathematics, College Park, MD 20742, USA}
\email{bds@umd.edu}
\author{W. Zhang}
\address{Massachusetts Institute of Technology, Department of Mathematics, 77 Massachusetts Avenue, Cambridge, MA 02139, USA}
\email{weizhang@mit.edu}

\date{\today}

\begin{abstract}
We define variants of PEL type of the Shimura varieties that appear in the context of the Arithmetic Gan--Gross--Prasad conjecture. We formulate for them a version  of the AGGP conjecture.  We also construct (global and semi-global) integral models of these Shimura varieties and formulate for them  conjectures on arithmetic intersection numbers. We prove some of these conjectures in low dimension. 
\end{abstract}

\maketitle

\tableofcontents

\section{Introduction}\label{s:intro}

The theorem of Gross and Zagier \cite{GZ} relates the N\'eron--Tate heights of Heegner points on modular curves to special values of derivatives of certain $L$-functions. Ever since the appearance of \cite{GZ}, the problem of generalizing this fundamental result to higher dimension has attracted considerable attention. The generalization that is most relevant to the present paper is the \emph{Arithmetic Gan--Gross--Prasad conjecture} (AGGP conjecture) \cite[\S 27]{GGP}. This conjectural generalization concerns Shimura varieties attached to orthogonal groups of signature $(2, n-2)$, and to unitary groups of signature $(1, n-1)$ (note that modular curves are closely related to Shimura varieties associated to orthogonal groups of signature $(2, 1)$ and to unitary groups of signature $(1, 1)$). In  \cite[\S 27]{GGP}, algebraic cycles of codimension one on such Shimura varieties are defined by exploiting embeddings of Shimura varieties attached to orthogonal groups of signature $(2, n-3)$, resp.\ to unitary groups of signature $(1, n-2)$. By taking the graphs of these embeddings, one obtains cycles in codimension just above half the (odd) dimension of the ambient variety. 

For any algebraic variety $X$  smooth and proper of odd dimension over a number field, Beilinson and Bloch have defined a height pairing on the \emph{rational} Chow group $\Ch(X)_{\BQ,0}$ of cohomologically trivial cycles of codimension just above half the dimension. Their definition makes use of some widely open unsolved conjectures on algebraic cycles and the existence of regular proper integral models of $X$. By suitably replacing in the case at hand the graph cycle by a cohomologically trivial avatar, one obtains a linear form on $\Ch(X)_{\BQ,0}$, where now $X$ is the product of the two Shimura varieties in question. The AGGP conjecture relates a special value of the derivative of an $L$-function to the non-triviality of the restriction of this linear form to a Hecke eigenspace in $\Ch(X)_{\BQ, 0}$. It is stated in a very succinct way in \cite{GGP}, for orthogonal groups and for unitary groups. In the present paper, we restrict ourselves to \emph{unitary} groups, and one of our aims is to give in this case more details on (a variant of) this conjecture. Our version here is also an improvement of the version of the conjecture in \cite{Z12,Z12c}. One new feature of our version is that we use the \emph{standard sign conjecture} for unitary Shimura varieties (a theorem of Morel--Suh \cite[Th.\ 1.2]{MS}, cf.\ Section \ref{ss:HKproj}) to construct ``Hecke--K\"unneth" projectors that project the total cohomology of our Shimura variety to the odd-degree part.

As indicated above, the AGGP conjecture is based on conjectures of Beilinson and Bloch which seem out of reach at  present. As a consequence, the conjecture in \cite{GGP} has not been proved in a single case of higher dimension.\footnote{However, we point out that \cite{YZZ2} proves certain variants of this conjecture in a higher-dimensional case for \emph{orthogonal groups} of type $\SO(3)\times\SO(4)$.} A variant of the AGGP conjecture, inspired by the \emph{relative trace formula} of Jacquet--Rallis \cite{JR}, has been proposed by the third author \cite{Z12}. More precisely, this variant relates the Beilinson--Bloch height pairing with distributions that appear in the relative trace formula. This variant leads to local conjectures (on intersection numbers on Rapoport--Zink spaces), namely the \emph{Arithmetic Fundamental Lemma} conjecture and the \emph{Arithmetic Transfer} conjecture, cf.\ \cite{Z12, RSZ2}---and these have been proved in various cases \cite{Z12, RSZ1, RSZ2, M-AFL, M-Th, RTZ}. The second aim of the present paper is to formulate a  global conjecture  whose proof  in various cases is a realistic goal. In the present paper, basing ourselves on our local papers \cite{Z12, RSZ2},  we prove this conjecture for unitary groups of size $n\leq 3$. 

To formulate this conjecture, we define variants of the Shimura varieties appearing in \cite{GGP} and \cite{Z12} which are of PEL type, i.e.\ are related to moduli problems of abelian varieties with polarizations, endomorphisms, and level structures. In fact, we even define integral models of these Shimura varieties, in a global version and a semi-global version. The construction of such models is the third aim of the present paper. Once these models are defined, we replace the Beilinson--Bloch pairing on the cohomologically trivial Chow group by the Gillet--Soul\'e pairing on the arithmetic Chow group of the (global or semi-global) integral model. 

Now that we have formulated the three main goals of this paper, let us be more specific. 

Let $F$ be a CM number field, with maximal totally real subfield $F_0$. We fix a CM type $\Phi$ of $F$ and a distinguished element $\varphi_0\in\Phi$. Let $n\geq 2$ and let $r \colon \Hom(F,\BC) \to \{0,1,n-1,n\}$, $\varphi \mapsto r_\varphi$, be the function defined by
\begin{equation*}
\begin{aligned}
   r_\varphi :=
	\begin{cases}
		1,  &  \varphi = \varphi_0;\\
		0,  &  \varphi \in \Phi \ssm \{\varphi_0\};\\
		n-r_{\ov\varphi},  &  \varphi \notin \Phi.
	\end{cases}
\end{aligned}
\end{equation*}
Associated to these data, there is the field $E\subset \ov\BQ$ which is the composite of the reflex field of $r$ and the reflex field of $\Phi$. Then $E$  contains $F$ via $\varphi_0$. We denote by $Z^\BQ$ the torus 
$$
   Z^\BQ := \bigl\{\, z\in \Res_{F/\BQ}(\BG_m) \bigm| \Nm_{F/F_0}(z)\in\BG_m\,\bigr\} .
$$
We also fix an $F/F_0$-hermitian vector space $W$ of dimension $n$ with signature 
$$
\sig(W_\varphi)=(r_\varphi, r_{\ov\varphi}), \quad \varphi\in\Phi .
$$
Let $G$ be the unitary group of $W$, considered as an algebraic group over $\BQ$.\footnote{This notation differs from the main body of the paper, where $G$ denotes the unitary group of $W$ over $F_0$.}
Associated to $(G, r)$ is the Shimura variety of \cite{GGP}. In the present paper, we instead consider the Shimura variety associated to $\wt G := Z^\BQ\times G$. We are able to formulate a PEL moduli problem $M_{K_{\wt G}}(\wt G)$ of abelian varieties with additional structure (endomorphisms and polarization) which defines a model over $E$ of the Shimura variety
\begin{equation}
\Sh_{K_{\wt G}}(\wt G)= M_{K_{\wt G}}(\wt G)\otimes_E\BC .
\end{equation}
(In fact, we demand that $K_{\wt G}=K_{Z^\BQ}\times K_G$, where $K_{Z^\BQ}$ is the unique maximal compact subgroup of $Z^\BQ(\BA_f)$ and where $K_G$ is an open compact subgroup of $G(\BA_f)$.) The group differs from the group of unitary similitudes $\GU(W)$ by a central isogeny. The Shimura variety corresponding to the latter group is considered by Kottwitz \cite{Ko}, and he formulates a PEL moduli problem over the reflex field of $r$ which \emph{almost} defines a model for it---but not quite, because of the possible failure of the Hasse principle for $\GU(W)$. This Shimura variety is also considered by Harris--Taylor \cite{HT}. In the setup of \cite{HT}, we have $E=F$ and both their  Shimura variety and ours are defined over $F$; however, ours offers a number of technical advantages over theirs.\footnote{Kottwitz \cite{Ko} does not need any assumptions on the signature of $W$; neither do we, cf.\ \cite{RSZ4}.} The definition of our moduli problem is based on a sign invariant $\inv^r_v(A_0, A)\in \{\pm1\}$ for every non-archimedean place $v$ of $F_0$ which is  non-split in $F$. Here $(A_0, \iota_0, \lambda_0)$ is a polarized abelian variety of dimension $d=[F_0:\BQ]$ with complex multiplication of CM type $\ov\Phi$ of $F$ and $(A, \iota, \lambda)$ is a polarized abelian variety of dimension $nd$ with complex multiplication of generalized CM type $r$ of $F$. This sign invariant is similar to the one in \cite{KR-new, KRZ}, but much simpler. This simplicity is another reflection of the advantage of our Shimura varieties over those considered by Kottwitz \cite{Ko}.

This sign invariant also allows us to define \emph{global integral models} of $M_{K_{\wt G}}(\wt G)$ over $\Spec O_E$ (at least when $F/F_0$ is not everywhere unramified) and \emph{semi-global integral models}  over $\Spec O_{E, (\nu)}$, where $\nu$ is a fixed non-archimedean place of $E$, of residue characteristic $p$. These integral models generalize those in \cite{BHKRY} when $F_0=\BQ$ and when $K_G$ is the stabilizer of a self-dual lattice in $W$. Here we allow $K_G$ to be the stabilizer of more general vertex lattices, in the sense of \cite{KR-U1}.   To achieve flatness, we sometimes have to impose conditions on the Lie algebras of the abelian varieties in play that are known in a similar context from our earlier local papers \cite{RSZ1, RSZ2} (the \emph{Pappas wedge condition}, the \emph{spin condition} and its refinement, the \emph{Eisenstein conditions}). However, in contrast to Kottwitz, we do not need any unramifiedness conditions. 

Once the model $ M_{K_{\wt G}}(\wt G)$ and its global or semi-global model are defined, we can also create a restriction situation in analogy with \cite{GGP}. Namely, fixing a \emph{totally negative} vector $u\in W$ (satisfying additional integrality conditions for the global, resp.\ semi-global integral situation), we define $W^\flat$ to be the orthogonal complement of $u$. Then $W^\flat$ satisfies the same conditions as $W$, with $n$ replaced by $n-1$. We obtain a finite unramified morphism
\begin{equation}
   M_{K_{\wt H}}(\wt H)\to  M_{K_{\wt G}}(\wt G) ,
\end{equation}
resp.\ their  global, resp.\ semi-global integral versions. Here $\wt H=Z^\BQ\times H$, where $H=\U(W^\flat)$, considered as an algebraic group over $\BQ$. Using the graph of the above morphism, we obtain an element in the \emph{rational Chow group}, 
$$
z_{K_{\wt\HG}}\in \Ch^{n-1}\bigl(M_{K_{\wt \HG}}(\wt \HG)\bigr)_\BQ . 
$$
Here $M_{K_{\wt \HG}}(\wt \HG)$ is the model defined as above for the Shimura variety for the group $Z^\BQ\times H\times G$. Using the Hecke--K\"unneth projector, we construct a cohomologically trivial variant $z_{ K_{\wt\HG}, 0}\in \Ch^{n-1}(M_{K_{\wt \HG}}(\wt \HG))_{0,\BC}$ of this element, which, via the Beilinson--Bloch pairing,  in turn defines a linear form
$$
\ell_{K_{\wt \HG}}\colon \Ch^{n-1}\bigl(M_{K_{\wt \HG}}(\wt \HG)\bigr)_{\BC, 0}\to \BC .
$$
Our variant of the AGGP conjecture is expressed in terms of this linear form.

We similarly define, under certain hypotheses, elements in the rational Chow group, resp.\ rational arithmetic Chow group, of the (global, or semi-global) integral model $\CM_{K_{\wt \HG}}(\wt \HG)$. Let us consider the Gillet--Soul\'e intersection product pairing on the rational arithmetic Chow group of $\CM_{K_{\wt \HG}}(\wt \HG)$. We have a conjecture on the value of  the intersection  product of $z_{K_{\wt\HG}}$ and its image under a Hecke correspondence. Let us state the semi-global version, since this is the one for which we can produce concrete evidence. We also have a global version.

\begin{conjecture}[Semi-global conjecture]\label{conj semiglob-intro} Fix a non-split place $v_0$ of $F_0$ over the place $p\leq \infty$ of $\BQ$. 
Let $f=\otimes_{\ell} f_\ell\in \sH_{{K_\wtHG}}^p$ ($\sH_{{K_\wtHG}}$ if $p$ is archimedean) be a completely decomposed element of the finite Hecke algebra of $\wt\HG$,  and  let $f'=\otimes_{v}f'_v\in \sH(G'(\BA_{F_0}))$ be a Gaussian test function in the Hecke algebra of $G'=\Res_{F/F_0}(\GL_{n-1}\times\GL_n)$ such that  $\otimes_{v < \infty}f'_v$ is a smooth transfer of $f$.
Assume that for some place $\lambda$ prime to $p$, the function $f$ has regular support at $\lambda$ in the sense of Definition \ref{regularsupport} and that $f'$ has regular support at $\lambda$ in the sense of Definition \ref{regsuppG'}. 
\begin{altenumerate}
\item\label{intro i} Assume that $v_0$ is non-archimedean of hyperspecial type, cf.\ Section \ref{subsec hyper}, and that $f_{v_0}'=\mathbf{1}_{G'(O_{F_0,v_0})}$. Then
$$
\Int_{v_0}(f)=-\delJ_{v_0}(f') . 
$$
\item\label{intro ii}  Assume that $v_0$ is archimedean, or non-archimedean of AT  type, cf.\ Section \ref{subsec AT}. Then 
$$
\Int_{v_0}(f)=-\delJ_{v_0}(f')-J(f'_{\corr}[v_0]),
$$where $f'_{\corr}[v_0]=\otimes_{v}f'_{\corr, v}$, with $f'_{\corr, v}=f'_v$ for $v\neq v_0$,  is a correction function. Furthermore,   $f'$ may be  chosen such that $f'_{\corr}[v_0]$ is zero.
\end{altenumerate}
\end{conjecture}
We refer to the body of the text for an explanation of the terms used (cf.\ Conjecture \ref{conj semiglob}). Here it suffices to remark that $\Int_{v_0}(f)$ is the weighted sum of the local contributions of all places $\nu$ of $E$ over $v_0$  to the Gillet--Soul\'e intersection product of the diagonal subscheme $ \CM_{K_{\wt H}}(\wt H)$ of $ \CM_{K_{\wt \HG}}(\wt \HG)$ and its translate by the Hecke correspondence $R(f)$ associated to $f$.  By $J$, resp.\ $\delJ_{v_0}$, we denote the distributions (cf.\ \eqref{J(0)} and \eqref{delJ}) that arise in the twisted trace formula approach to the AGGP conjecture. The regularity assumption on $f$ guarantees that for non-archimedean $v_0$ the intersection of the two cycles has support in characteristic $p$, and the regularity assumption on $f'$ guarantees that the distributions $J$ and $\delJ_{v_0}$  localize. 

In the global context, when the Hecke correspondence $R(f)$ satisfies a suitable regularity assumption, the intersection product localizes, i.e., is a finite  sum of contributions, one from each place $\nu$ of $E$. We group together the local contributions from all places $\nu$ which induce a given place $v_0$ of $F_0$. From this point of view, Conjecture \ref{conj semiglob-intro}\eqref{intro i} predicts the contribution of the good places, and Conjecture \ref{conj semiglob-intro}\eqref{intro ii} the contributions from the archimedean places and certain bad places.

The conjecture is accessible in certain cases. We prove the following theorem.
\begin{theorem}\label{mainthmintro}
 Let $v_0$ be a non-archimedean place of $F_0$ that is non-split in $F$. Then Conjecture \ref{conj semiglob-intro}  above holds true for $n\leq 3$.
 \end{theorem}
The main input is our work in the local case (intersection product on Rapoport--Zink spaces): the proof of the AFL conjecture in the hyperspecial case for $n\leq 3$ by one of us \cite{Z12} and the proof of the AT conjecture for $n\leq 3$ in \cite{RSZ1, RSZ2}. The passage from the local statement to a global statement is modeled on the similar passage in \cite{KR-U2} (which also inspired  the similar passage in \cite{Z12}). In fact, this similarity is not only formal. Indeed, the definition of \emph{Kudla--Rapoport divisors} uses in an essential way that the Shimura variety for $\GU(W)$ is replaced by the Shimura variety for $\wt G$ (in fact, in \cite{KR-U2}, one uses $Z^\BQ\times \GU(W)$; as remarked in \cite{BHKRY}, the definition can be realized on the Shimura variety for $\wt G$). In this way, there is a direct connection between the intersection problem occurring in Conjecture \ref{conj semiglob-intro} and the intersection problem in \cite{KR-U2}.

 We also have a theorem for  split places (hence non-archimedean), cf.\ Proposition \ref{p:split I=J}.
 \begin{theorem}\label{mainthmintro split}
 Let $v_0$ be a place of $F_0$ that is split in $F$. Let $f$ and $f'$ be as in Conjecture \ref{conj integralnontrivlev}. Then 
$$
\Int_{v_0}(f)=\delJ_{v_0}(f')=0.
$$
\end{theorem}
The significance of this theorem is that in the global context, again under a regularity assumption on the Hecke correspondence $R(f)$, the contribution of the places $v_0$ which split in $F$ is trivial. 
\medskip

Let us now put the results of this paper in perspective. B. Gross asked for integral models of the Shimura varieties for unitary groups. Here we answer a modified version of his question, by constructing  integral models of the Shimura varieties for $\wt G$.  However, we have to pay a price by having to replace the field $F$, over which Gross's Shimura varieties have a model,  by the field $E$, over which our Shimura varieties have a model; and $E$  may be strictly larger than $F$. This also causes us  to modify our adaptation of the AGGP conjecture. It would be interesting to understand whether the Kisin--Pappas construction of integral models of Shimura varieties of abelian type \cite{KP}  yields a solution of Gross's question which can be used to give a variant of the AGGP conjecture which avoids having to replace $F$ by a bigger field. 

Our current knowledge of AT conjectures forces on us to be very specific when imposing level structures in our moduli problems. It seems realistic to hope that more cases of AT conjectures than in \cite{RSZ1, RSZ2} can be formulated, and this would allow more flexibility for the level structures. We hope to return to this point. 

How realistic is it to hope that the conjectures on the arithmetic intersection pairing can be proved, in cases that go beyond those treated in this paper? The stumbling block seems to be that in higher dimension it is difficult to avoid degenerate intersections---and degenerate intersections seem to be a challenge to currently available techniques, already in the local situation\footnote{Since the submission of the paper, the last author  has made progress in the good reduction case \href{arXiv:1909.02697}{arXiv:1909.02697}.} (intersection on Rapoport--Zink spaces). It might be fruitful to search for intersection problems derived from those considered here which avoid these apparently very difficult problems, in the spirit of B.~Howard's papers concerning the Kudla--Rapoport divisor intersection problem, cf.\ \cite{Ho-kr, Ho-kr2}.

As is apparent, automorphic $L$-functions do not appear explicitly  in the statements above, contrary to what happens in the AGGP conjecture. $L$-functions  are involved implicitly because the distributions $J$ and $\delJ_{v_0}$ are related to them (cf.\ Section \ref{s:Lfcts}). However, more analytic work is involved to make this relation more explicit. One of us (W.Z.) hopes to return to this point and explain this issue in more detail.

We also mention the article \cite{Z-ICM} of the third author, which discusses some of the broader context on the relation of special values of $L$-functions and their derivatives to period integrals and intersections of cycles, into which this paper falls.
\medskip

We finally give an overview of the layout of the paper. In Section \ref{s:gpthsetup}, we introduce the groups in play and define the concept of matching in this context. In Section \ref{Shimura varieties}, we introduce the various Shimura varieties mentioned above, and the corresponding moduli interpretation over $E$, and the relation between the moduli variety for $W^\flat$ and for $W$. In Section \ref{section semi-global}, we define the semi-global integral models of these moduli schemes and the morphisms between them. We do this in several contexts: for \emph{hyperspecial level}, for \emph{split level}, for \emph{Drinfeld level}, and for \emph{AT parahoric level}.  These levels reflect the possibility of applying the AFL conjecture, resp.\ the AT conjecture, resp.\ the vanishing theorem Theorem \ref{mainthmintro split} in the split level. The AT parahoric level is complicated by the fact that the morphisms between the semi-global integral models are in some cases not defined in the naive way. In Section \ref{section global} we give the global integral models, first without level structure and then with Drinfeld level structure. Section \ref{section ggp} is devoted to giving our version of the AGGP conjecture. Section \ref{s:Lfcts} is preparatory for the last section. Here we explain the distributions arising in the context of the relative trace formula and their relation to $L$-functions. In Section \ref{s:conjaip} everything comes finally together. First, we formulate our conjecture on the arithmetic intersection numbers. We do this in the global case without level structure, in the global case with Drinfeld level structure, and in the semi-global case.  Second, we  give the proofs of the semi-global versions in cases of small dimension, cf.\ Theorem \ref{mainthmintro}.  There are two appendices. In Appendix A we define the sign invariant that is used in the formulation of the moduli schemes. In Appendix B we check that in the case of \emph{banal signature} the relevant local models are trivial in a precise sense.

\subsection*{Acknowledgments} We are grateful to 
G.~Chenevier, B.~Gross, B.~Howard, S.~Kudla, Y.~Liu, P.~Scholze and Y.~Tian for helpful remarks. We also thank the referees for their work. We acknowledge the hospitality of the ESI (Vienna) and the MFO (Oberwolfach), where part of this work was carried out. M.R. is supported by a grant from  the Deutsche Forschungsgemeinschaft through the grant SFB/TR 45. B.S. is supported by Simons Foundation grant \#359425 and NSA grant H98230-16-1-0024. W.Z. is supported by NSF grant DMS $\#$1601144, $\#$1901642 and a Simons Fellowship.

\subsection*{Notation}
Except in Section \ref{s:gpthsetup}, $F$ denotes a CM number field and $F_0$ denotes its (maximal) totally real subfield of index $2$ (in Section \ref{s:gpthsetup}, $F/F_0$ can be any quadratic extension of number fields).  We denote by $a\mapsto \ov a$ the nontrivial automorphism of $F/F_0$.  We fix a presentation $F = F_0(\sqrt \Delta)$ for some totally negative element $\Delta \in F_0$, and we let $\Phi$ be the CM type for $F$ determined by $\sqrt \Delta$,
\begin{equation}\label{Phi}
   \Phi := \bigl\{\, \varphi\colon F \rightarrow \BC \bigm| \varphi(\sqrt\Delta) \in \BR_{>0} \cdot \sqrt{-1} \,\bigr\}.
\end{equation}
Note that by weak approximation, every CM type for $F$ arises in this way for some $F/F_0$-traceless element $\sqrt \Delta \in F^\times$.

We use the symbols $v$ and $v_0$ to denote places of $F_0$, and $w$ and $w_0$ to denote places of $F$.  We write $F_{0,v}$ for the $v$-adic completion of $F_0$, and we set $F_v := F \otimes_{F_0} F_{0,v}$; thus $F_v$ is isomorphic to $F_{0,v} \times F_{0,v}$ or to a quadratic field extension of $F_{0,v}$ according as $v$ is split or non-split in $F$.  We often identify the CM type $\Phi$ (or more precisely, the restrictions of its elements to $F_0$) with the archimedean places of $F_0$. When $v$ is a finite place, we write $\fkp_v$ for the maximal ideal in $O_{F_0}$ at $v$, we write $\varpi_v$ for a uniformizer in $F_{0,v}$, and we write $\pi_v$ for a uniformizer in $F_v$ (when $v$ splits in $F$ this means an ordered pair of uniformizers on the right-hand side of the isomorphism $F_v \cong F_{0,v} \times F_{0,v}$).  We write $O_{F_0,(v)}$ for the localization of $O_{F_0}$ at the maximal ideal $\fkp_v$, and $O_{F_0,v} \subset F_{0,v}$ for its $\fkp_v$-adic completion. We use analogous notation for other fields in place of $F_0$ and other finite places in place of $v$. In particular, we will often consider the $v$-adic completion $O_{F,v} = O_F \otimes_{O_{F_0}} O_{F_0,v}$ of $O_F$.

We write $\BA$, $\BA_{F_0}$, and $\BA_F$ for the adele rings of $\BQ$, $F_0$, and $F$, respectively.  We systematically use a subscript $f$ for the ring of finite adeles, and a superscript $p$ for the adeles away from the prime number $p$.

We take all hermitian forms to be linear in the first variable and conjugate-linear in the second, and we assume that they are non-degenerate unless we say otherwise.  For $k$ any field, $A$ an \'etale $k$-algebra of degree $2$, and $W$ a finite free $A$-module equipped with an $A/k$-hermitian form, we write $\det W \in k^\times/ \Nm_{A/k}A^\times$ for the class of $\det J$, where $J$ is any hermitian matrix (relative to the choice of an $A$-basis for $W$) representing the form.  Note that this value group is trivial when $A \simeq k \times k$.  We also write $-W$ for the same $A$-module as $W$, but with the hermitian form multiplied by $-1$.  Of course $W$ and $-W$ have the same unitary groups.  When $W$ is an $F/F_0$-hermitian space of dimension $n$ and $v$ is a place of $F_0$, we write $W_v$ for the induced $F_v/F_{0,v}$-hermitian space $W \otimes_{F_0} F_{0,v}$, and we define
\begin{equation}\label{Hasse def}
   \inv_v(W_v) := (-1)^{n(n-1)/2} \det W_v \in F_{0,v}^\times / \Nm F_v^\times.
\end{equation}
We say that $W_v$ is \emph{split} at a finite place $v$ if $\inv_v(W_v) = 1$; under our normalization, the antidiagonal unit matrix always defines a split hermitian form.  When $v$ is an archimedean place, the form on $W_v$ is isometric to $\diag(1^{(r)},(-1)^{(s)})$ for some $r+s=n$, and we write $\sig_v(W_v) := (r,s)$ (the \emph{signature}).  In the local setting, isometry classes of $n$-dimensional $F_v/F_{0,v}$-hermitian spaces are classified by $\inv_v$ when $v$ is a finite place, and by $\sig_v$ when $v$ is an archimedean place. By the Hasse principle, two global hermitian spaces are isometric if and only if they are isometric at every place $v$, i.e.\ they have the same invariants at each finite place and the same signatures at each archimedean place.  Given a global space $W$ as above, the product formula for the norm residue symbol for the extension $F/F_0$ gives
\begin{equation}\label{herm prod fmla}
   \prod_v \inv_v(W_v) = 1,
\end{equation}
where $v$ ranges through the places of $F_0$, and where we identify $F_{0,v}^\times / \Nm F_v^\times \subset \{\pm 1\}$.  Conversely, Landherr's theorem asserts that a collection $(W_v)_v$ of $F_v/F_{0,v}$-hermitian spaces arises as the set of local completions of a (unique, by the Hasse principle) global $F/F_0$-hermitian space exactly when $\inv_v(W_v) = 1$ for all but finitely many $v$ and the product formula \eqref{herm prod fmla} holds.  Given an embedding $\varphi\colon F \to \BC$, we write $W_\varphi := W \otimes_{F,\varphi}\BC$ for the induced hermitian space over $\BC$.

For any abelian scheme $A$ over a base scheme $S$, we denote by $A^\vee$ the dual abelian scheme.  When $S$ is locally noetherian and the prime number $p$ is invertible on it, we write $T_p(A)$ for the $p$-adic Tate module of $A$ (regarded as a smooth $\BZ_p$-sheaf on $S$) and $V_p(A) := T_p(A) \otimes \BQ$ for the rational $p$-adic Tate module (regarded as a smooth $\BQ_p$-sheaf on $S$). When $S$ is a $\BZ_{(p)}$-scheme, we similarly write $\wh V^p(A)$ for the rational prime-to-$p$ Tate module of $A$.  When $S$ is a scheme in characteristic zero, we write $\wh T(A)$ and $\wh V(A)$ for the respective full Tate and full rational Tate modules of $A$.

We use a superscript $\circ$ to denote the operation $-\otimes_\BZ \BQ$ on groups of homomorphisms of abelian schemes, so that for example $\Hom^\circ(A,A') := \Hom(A,A')\otimes_\BZ \BQ$.

Given modules $M$ and $N$ over a ring $R$, we write $M \subset^r N$ to indicate that $M$ is an $R$-submodule of $N$ of finite colength $r$.  Typically $R$ will be $O_{F,v}$ for $v$ a finite place of $F_0$.  When $\Lambda$ is an $O_F$-lattice in an $F/F_0$-hermitian space, we denote the dual lattice with respect to the hermitian form by $\Lambda^*$. We use the same notation when $\Lambda$ is an $O_{F,v}$-lattice in an $F_v/F_{0,v}$-hermitian space, and we call $\Lambda$ a \emph{vertex lattice of type $r$} if $\Lambda \subset^r \Lambda^* \subset \pi_v^{-1} \Lambda$.  Note that this terminology differs slightly from e.g.~\cite{KR-U1,RTW}.\footnote{ More precisely, in the case that $F_v/F_{0,v}$ is the unramified extension $\BQ_{p^2}/\BQ_p$, we have defined what \cite{KR-U1} calls a \emph{vertex lattice of level $0$ and type $n-r$}, where $n$ denotes the dimension of the hermitian space.  In the case that $F_v$ is a ramified extension of $F_{0,v} = \BQ_p$, our $\Lambda^*$ is what \cite{RTW} calls a \emph{vertex lattice of type $r$}, except that \cite{RTW} should also require the containment $\Lambda^* \subset \pi_v\i \Lambda$.}  A \emph{vertex lattice} is a vertex lattice of type $r$ for some $r$.  Let us single out the following special cases.  A \emph{self-dual} lattice is, of course, a vertex lattice of type $0$.  An \emph{almost self-dual} lattice is a vertex lattice of type $1$.  At the other extreme, a vertex lattice $\Lambda$ is \emph{$\pi_v$-modular} if $\Lambda^* = \pi_v^{-1} \Lambda$, and \emph{almost $\pi_v$-modular} if $\Lambda \subset \Lambda^* \subset^1 \pi_v^{-1}\Lambda$.

Given a discretely valued field $L$, we denote by $\breve L$ the completion of a maximal unramified extension of $L$.

We write $1_n$ for the $n \times n$ identity matrix.  We use a subscript $S$ to denote base change to a scheme (or other object) $S$, and when $S = \Spec A$, we often use a subscript $A$ instead.

\section{Group-theoretic setup}\label{s:gpthsetup}
In this section we introduce the groups and linear-algebraic objects which will be in play throughout the paper.  Let $F/F_0$ be a quadratic extension of number fields.

\subsection{Similitude groups and variants}
We begin by introducing algebraic groups over $F_0$,
\begin{align*}
   G' &:= \Res_{F/F_0}(\GL_{n-1} \times \GL_n),\\
	H_1' &:= \Res_{F/F_0} \GL_{n-1},\\
	H_2' &:= \GL_{n-1}\times\GL_n,\\
	H_{1,2}' &:= H_1' \times H_2'.
\end{align*}
Next let $W$ be a non-degenerate $F/F_0$-hermitian space of dimension $n \geq 2$. We fix a non-isotropic vector $u \in W$, which we call the \emph{special vector}. We denote by $W^\flat$ the orthogonal complement of $u$ in $W$.  We define another four algebraic groups over $F_0$,
\begin{align*}
	G &:= \U(W),\\
   H &:= \U(W^\flat),\\
	G_W&:= H \times G,\\
	H_W&:= H \times H.
\end{align*}
We further define the following algebraic groups over $\BQ$.  We systematically use the symbol $c$ to denote the similitude factor of a point on a unitary similitude group.
{\allowdisplaybreaks
\begin{align*}
   Z^\BQ &:= \bigl\{\, z \in \Res_{F/\BQ} \BG_m \bigm| \Nm_{F/F_0}(z) \in \BG_m \,\bigr\},\\
	H^\BQ &:= \bigl\{\, h \in \Res_{F_0/\BQ} \GU(W^\flat) \bigm| c(h)\in \BG_m \,\bigr\},\\
	G^\BQ &:= \bigl\{\, g \in \Res_{F_0/\BQ} \GU(W) \bigm| c(g)\in \BG_m \,\bigr\},\\
	\wt H &:= Z^\BQ \times_{\BG_m} H^\BQ = \bigl\{\, (z,h) \in Z^\BQ \times H^\BQ \bigm| \Nm_{F/F_0}(z) = c(h) \,\bigr\},\\
	\wt G &:= Z^\BQ \times_{\BG_m} G^\BQ = \bigl\{\, (z,g) \in Z^\BQ \times G^\BQ \bigm| \Nm_{F/F_0}(z) = c(g) \,\bigr\},\\
	\wtHG &:= \wt H \times_{Z^\BQ} \wt G = Z^\BQ \times_{\BG_m} H^\BQ \times_{\BG_m} G^\BQ \\
	      &\phantom{:}= \bigl\{\, (z,h,g) \in Z^\BQ \times H^\BQ \times G^\BQ \bigm| \Nm_{F/F_0}(z) = c(h) = c(g) \,\bigr\}.
\end{align*}
}%
Note that $Z^\BQ$ is naturally a central subgroup of $H^\BQ$ and $G^\BQ$, and these inclusions give rise to product decompositions
\begin{equation}\label{proddec}
   \begin{gathered}
   \begin{gathered}
	\xymatrix@R=0ex{
      \wt H \ar[r]^-\sim  &  Z^\BQ \times \Res_{F_0/\BQ} H\\
		(z, h) \ar@{|->}[r]  &  (z, z^{-1}h)
	}
   \end{gathered},
	\qquad
   \begin{gathered}
	\xymatrix@R=0ex{
	   \wt G \ar[r]^-\sim  &  Z^\BQ \times \Res_{F_0/\BQ} G\\
		(z, g) \ar@{|->}[r]  &  (z, z^{-1}g)
	}
	\end{gathered},
	\\
	\begin{gathered}
	\xymatrix@R=0ex{
      \wtHG \ar[r]^-\sim &  Z^\BQ \times \Res_{F_0/\BQ} (H\times G)\\
	   (z, h, g) \ar@{|->}[r]  &  (z, z^{-1}h, z^{-1}g)
	}
	\end{gathered}
   \end{gathered}
\end{equation}

We also record that the decomposition $W = W^\flat \oplus Fu$ gives rise to natural closed embeddings of algebraic groups,
\begin{equation}\label{embedoverQ}
   \begin{gathered}
   \xymatrix@R=0ex{
	   \wt H \ar@{^{(}->}[r]  &  \wt G \\
		(z,h) \ar@{|->}[r]  &  \bigl(z,\diag(h,z)\bigl)
	}
   \end{gathered}
	\quad\text{and}\quad
	\begin{gathered}
	\xymatrix@R=0ex{
	   \wt H \ar@{^{(}->}[r]  &  \wtHG \\
		(z,h) \ar@{|->}[r]  &  \bigl(z, h, \diag(h,z)\bigl)
	}
	\end{gathered}.
\end{equation}
Thus, in terms of the product decompositions in \eqref{proddec}, the embeddings $\wt H \inj \wt G$ and $\wt H \inj \wtHG$ in \eqref{embedoverQ} are obtained by applying the functor $Z^\BQ \times -$ to the embeddings $\Res_{F_0/\BQ} H \inj \Res_{F_0/\BQ} G$ and $\Res_{F_0/\BQ} H \inj \Res_{F_0/\BQ} (H \times G)$, respectively.

\subsection{Orbit matching}\label{orbit matching ss}
The following lemma is obvious. 
\begin{lemma}\label{matchintilde}
The natural projections in (\ref{proddec}) induce isomorphisms
\[
   \wt G / \wt H \isoarrow \Res_{F_0/\BQ} G \big/ \Res_{F_0/\BQ} H
\]
and
\begin{flalign*}
	\phantom{\qed} & &
   \wt H \bs \wtHG / \wt H \isoarrow \Res_{F_0/\BQ} H \big\bs \Res_{F_0/\BQ} G_W \big/ \Res_{F_0/\BQ} H .
	& & \qed
\end{flalign*}
\end{lemma}
There is a natural injection of orbit spaces of \emph{regular semisimple} elements, 
\[
   H(F_{0, v})\backslash G_W(F_{0, v})_\rs/H(F_{0, v})
	   \inj H'_1(F_{0,v})\backslash G'(F_{0, v})_\rs/H'_2(F_{0, v})
\]
for any place $v$ of $F_0$ (cf.\  \cite[\S 2]{RSZ1} for  the case when $v$ is a non-archimedean place not split in $F$; the case of  archimedean places is completely analogous).  
If $v$ is split in $F$, we define matching  as in \cite[\S2]{Z14}. Briefly speaking, we identify $H(F_{0,v})$ with $\GL_{n-1}(F_{0,v})$  and $G(F_{0,v})$ with $\GL_{n}(F_{0,v})$. This gives a natural way of matching regular semisimple elements, a  \emph{homogeneous version} of the \emph{matching relation} of \cite{Z12}.

 Using  Lemma \ref{matchintilde}, we obtain an injection for every prime number $p$,  
\begin{align}\label{matchwtHG}
   \wt H(\BQ_p) \bs \wtHG(\BQ_p)_\rs / \wt H(\BQ_p) \inj \prod_{v\mid p} H'_1(F_{0, v})\backslash G'(F_{0, v})_\rs/ H'_2(F_{0, v}) .
\end{align}

\section{The Shimura varieties}\label{Shimura varieties}

For the rest of the paper we take $F$ to be a CM field over $\BQ$ and $F_0$ to be its totally real subfield of index $2$.  We recall from the Introduction that we fix a totally imaginary element $\sqrt\Delta \in F^\times$, and we denote by $\Phi$ the induced CM type for $F$ given in (\ref{Phi}).

\subsection{The Shimura data}\label{Sh data}
In this subsection we define Shimura data for some of the groups introduced in Section \ref{s:gpthsetup}.  We assume that the hermitian space $W$ has the following signatures at the archimedean places of $F_0$: for a distinguished element $\varphi_0 \in \Phi$, the signature of $W_{\varphi_0}$ is $(1,n-1)$, and for all other $\varphi \in \Phi$ the signature of $W_\varphi$ is $(0,n)$. We also assume that the special vector $u$ is totally negative, i.e.~that $(u, u)_\varphi<0$ for all $\varphi$. 

We first define Shimura data $(G^\BQ, \{h_{G^\BQ}\})$ and $(H^\BQ, \{h_{ H^\BQ}\})$; comp.~\cite[\S1.1]{PR}. Using the canonical inclusions $G^\BQ_\BR\subset\prod_{\varphi\in\Phi} \GU(W_\varphi)$ and $H^\BQ_\BR\subset\prod_{\varphi\in\Phi} \GU(W^\flat_\varphi)$, it suffices to define the components $h_{G^\BQ, \varphi}$ of $h_{G^\BQ}$ and $h_{H^\BQ, \varphi}$ of $h_{H^\BQ}$. We define matrices 
\[
   J_\varphi :=
	\begin{cases}
		\diag\bigl(1, (-1)^{(n-1)}\bigr),  &  \varphi=\varphi_0;\\ 
      \diag( -1, -1, \dotsc, -1),  &  \varphi\in\Phi\ssm\{\varphi_0\},
   \end{cases}
\]
and we define the matrix $J_\varphi^\flat$ by removing a $-1$ from $J_\varphi$. We may then choose bases $W_\varphi\simeq \BC^n$ and $W^\flat_\varphi\simeq \BC^{n-1}$ such that $u \otimes 1 \in W_\varphi$ identifies with a multiple of $(0^{(n-1)},1)$, such that the inclusion $W_\varphi^\flat \subset W_\varphi$ is compatible with the inclusion $\BC^{n-1} \subset \BC^{n-1} \oplus \BC = \BC^n$, and such that the hermitian forms on $W_\varphi$ and $W_\varphi^\flat$ have respective normal forms
\[
   (x, y)_\varphi = \tensor*[^t]{x}{} J_\varphi \ov y
	\quad\text{and}\quad
	(x^\flat, y^\flat)_\varphi = \tensor*[^t]{x}{^\flat} J^\flat_\varphi \ov y^\flat.
\]
We then define the component maps
\[
   h_{G^\BQ,\varphi}\colon \BC^\times\to \GU(W_\varphi)(\BR)
	\quad\text{and}\quad
	h_{H^\BQ,\varphi}\colon \BC^\times\to \GU(W_\varphi^\flat)(\BR)
\]
to be induced by the respective $\BR$-algebra homomorphisms
\[
   \begin{gathered}
   \xymatrix@R=0ex{
      \BC \ar[r]  &  \End(W_\varphi)\\
		\sqrt{-1} \ar@{|->}[r]  &  \sqrt{-1} J_\varphi
	}
	\end{gathered}
	\quad\text{and}\quad
   \begin{gathered}
   \xymatrix@R=0ex{
      \BC \ar[r]  &  \End(W_\varphi^\flat)\\
		\sqrt{-1} \ar@{|->}[r]  &  \sqrt{-1} J_\varphi^\flat
	}
	\end{gathered}.
\]
By definition of $\Phi$, the form $x,y \mapsto \tr_{\BC/\BR} \varphi(\sqrt\Delta)\i(h_{G^\BQ,\varphi}(\sqrt{-1}) x, y)_\varphi$ is symmetric and positive definite on $W_\varphi$ for each $\varphi \in \Phi$, and similarly for $W^\flat_\varphi$.  

We next define Shimura data $(\wt H, \{h_{\wt H}\})$, $(\wt G, \{h_{\wt G}\})$, and $(\wtHG, \{h_{\wtHG}\})$.  For this, note that $\Phi$ induces an identification
\[
   Z^\BQ(\BR) \cong \bigl\{\,(z_\varphi)\in (\BC^\times)^\Phi \bigm| |z_\varphi| = |z_{\varphi'}| \ \text{for all}\ \varphi,\varphi' \in \Phi \,\bigl\} .
\]
In terms of this identification, we define $h_{Z^\BQ}\colon \BC^\times \to Z^\BQ(\BR)$ to be the  diagonal embedding, \emph{pre-composed with complex conjugation}.  We then obtain the desired Shimura data by defining the Shimura homomorphisms
\[
   h_{\wt H}\colon \BC^\times \xra{(h_{Z^\BQ}, h_{H^\BQ})} \wt H(\BR),
	\quad
	h_{\wt G}\colon \BC^\times \xra{(h_{Z^\BQ}, h_{G^\BQ})} \wt G(\BR),
	\quad
	h_{\wtHG}\colon \BC^\times \xra{(h_{Z^\BQ}, h_{H^\BQ}, h_{G^\BQ})} \wtHG(\BR).
\]
It is easy to see that $(\wt H, \{h_{\wt H}\})$, $(\wt G, \{h_{\wt G}\})$, and $(\wtHG, \{h_{\wtHG}\})$ have common reflex field $E \subset \BC$ characterized by
\begin{equation}\label{defE}
   \Aut(\BC/E) = \{\, \sigma\in\Aut(\BC)\mid \sigma \circ\Phi=\Phi \text{ and } \sigma\circ\varphi_0=\varphi_0 \,\} .
\end{equation}
We therefore obtain \emph{canonical models} over $E$ of the Shimura varieties
\begin{equation*}
   \Sh_{K_{\wt H}}\bigl(\wt H, \{h_{\wt H}\}\bigr), 
	\quad
	\Sh_{K_{\wt G}}\bigl(\wt G, \{h_{\wt G}\}\bigr),
	\quad
	\Sh_{K_{\wtHG}}\bigl(\wtHG, \{h_{\wtHG}\}\bigr) ,
\end{equation*}
where ${K_{\wt H}}$, resp.\ ${K_{\wt G}}$, resp.\ ${K_{\wtHG}}$ varies through the open compact subgroups of $\wt H(\BA_f)$, resp.\ $\wt G(\BA_f)$, resp.\ $\wtHG(\BA_f)$.

\begin{remark}\label{reflex rem}
When $F$ is of the form $K F_0$ for an imaginary quadratic field $K/\BQ$ and $\Phi$ is the unique CM type induced from $K$ containing $\varphi_0$ (the case taken by Harris--Taylor in \cite{HT}), $\varphi_0$ identifies $F \isoarrow E$.  In general, $E$ is the join of $F$ (embedded in $\ov\BQ$ via $\varphi_0$) and the reflex field $E_\Phi$ of $\Phi$ (the latter is the fixed field in $\BC$ of the group $\{\sigma \in \Aut(\BC) \mid \sigma \circ \Phi = \Phi\}$).  In particular, $\varphi_0$ always embeds $F$ into $E$, but $E$ may be larger.
\end{remark}

The morphisms \eqref{embedoverQ} are obviously compatible with the Shimura data $\{h_{\wt H}\}$ and $\{h_{\wt G}\}$, resp.\ $\{h_{\wt H}\}$ and $\{h_{\wtHG}\}$. We therefore obtain injective morphisms of Shimura varieties, i.e.\ injective morphisms of pro-varieties, in the sense of \cite[Prop.~1.15]{De},
\begin{equation}\label{Sh inj}
   \Sh\bigl(\wt H, \{h_{\wt H}\}\bigr) \hookrightarrow \Sh\bigl(\wt G, \{h_{\wt G}\}\bigr)
	\quad\text{and}\quad
	\Sh\bigl(\wt H, \{h_{\wt H}\}\bigr) \hookrightarrow \Sh\bigl(\wtHG, \{h_{\wtHG}\}\bigr) . 
\end{equation}

\begin{remark}\label{others}
The above Shimura varieties are related to other Shimura varieties, as follows. 
\begin{altenumerate}
\item\label{Z^Q} The pair $(Z^\BQ, \{h_{Z^\BQ}\})$ is a Shimura datum, and there are morphisms of Shimura data
$$
\bigl(\wt H, \{h_{\wt H}\}\bigr)\to \bigl(Z^\BQ, \{h_{Z^\BQ}\}\bigr), \quad \bigl(\wt G, \{h_{\wt G}\}\bigr)\to \bigl(Z^\BQ, \{h_{Z^\BQ}\}\bigr), \quad \bigl(\wtHG, \{h_{\wtHG}\}\bigr)\to \bigl(Z^\BQ, \{h_{Z^\BQ}\}\bigr)
$$
induced by the natural projections to $Z^\BQ$. These induce morphisms of Shimura varieties
\begin{equation}
\begin{gathered}
   \Sh\bigl(\wt H, \{h_{\wt H}\}\bigr)\to \Sh\bigl(Z^\BQ, \{h_{Z^\BQ}\}\bigr),\\
   \Sh\bigl(\wt G, \{h_{\wt G}\}\bigr)\to \Sh\bigl(Z^\BQ, \{h_{Z^\BQ}\}\bigr),\\
	\Sh\bigl(\wtHG, \{h_{\wtHG}\}\bigr)\to \Sh\bigl(Z^\BQ, \{h_{Z^\BQ}\}\bigr),
\end{gathered}
\end{equation}
which identify
\[
   \Sh\bigl(\wtHG, \{h_{\wtHG}\}\bigr) \cong \Sh\bigl(\wt H, \{h_{\wt H}\}\bigr) \times_{\Sh(Z^\BQ, \{h_{Z^\BQ}\})} \Sh\bigl(\wt G, \{h_{\wt G}\}\bigr).
\]
The reflex field of $(Z^\BQ, \{h_{Z^\BQ}\})$ is the subfield $E_\Phi$ of $E$.
\item There are morphisms of Shimura data
\[
   \bigl(\wt H, \{h_{\wt H}\}\bigr)\to \bigl(H^\BQ, \{h_{H^\BQ}\}\bigr)
   \quad\text{and}\quad
   \bigl(\wt G, \{h_{\wt G}\}\bigr)\to \bigl(G^\BQ, \{h_{G^\BQ}\}\bigr),
\]
both induced by the natural projections, which induce morphisms of Shimura varieties
\[
   \Sh\bigl(\wt H, \{h_{\wt H}\}\bigr)\to \Sh\bigl(H^\BQ, \{h_{H^\BQ}\}\bigr)
   \quad\text{and}\quad
   \Sh\bigl(\wt G, \{h_{\wt G}\}\bigr)\to \Sh\bigl(G^\BQ, \{h_{G^\BQ}\}\bigr).
\]
\item\label{ggp} One may also introduce Shimura data
\[
  \bigl(\Res_{F_0/\BQ} H, \{h_{H}\}\bigr),
  \quad
  \bigl(\Res_{F_0/\BQ} G, \{h_{ G}\}\bigr),
  \quad
  \bigl(\Res_{F_0/\BQ}(H\times G), \{h_{H\times G}\}\bigr),
\]
where $h_H$, resp.\ $h_G$, resp.\ $h_{H \times G}$ is defined by composing $h_{\wt H}$, resp.\ $h_{\wt G}$, resp.\ $h_{\wtHG}$ with the projection to the second factor in (\ref{proddec}). Note that the restrictions of $h_H$, $h_G$, and $h_{H\times G}$ to the subgroup $\BR^\times$ of $\BC^\times$ are trivial. In particular, the corresponding Shimura varieties are not of PEL type.  These are the Shimura varieties that appear in Gan--Gross--Prasad \cite[\S27]{GGP}. The product decompositions in (\ref{proddec}) induce product decompositions of Shimura data,
\begin{equation}\label{reltogross}
\begin{aligned}
	\bigl(\wt H, \{h_{\wt H}\}\bigr) &= \bigl(Z^\BQ,\{h_{Z^\BQ}\}\bigr) \times \bigl(\Res_{F/F_0} H,\{h_H\}\bigr),\\
   \bigl(\wt G, \{h_{\wt G}\}\bigr) &= \bigl(Z^\BQ,\{h_{Z^\BQ}\}\bigr) \times \bigl(\Res_{F/F_0} G,\{h_G\}\bigr),\\
	\bigl(\wtHG, \{h_{\wtHG}\}\bigr) &= \bigl(Z^\BQ,\{h_{Z^\BQ}\}\bigr) \times \bigl(\Res_{F_0/\BQ}(H\times G), \{h_{H\times G}\}\bigr) .
\end{aligned}
\end{equation}
In \cite[\S 27]{GGP}, the Shimura variety $\Sh(\Res_{F_0/\BQ}G, \{h_{G}\})$ is considered over its reflex field, which is $F$, embedded into $\BC$ via $\varphi_0$. By contrast, in the present paper, we consider the Shimura variety $\Sh(\wt G, \{h_{\wt G}\})$ over $E$, which is the join of the reflex fields of the two factors in the product decomposition \eqref{reltogross}. The natural projections then induce morphisms of Shimura varieties,
\begin{equation}
\begin{gathered}
   \Sh\bigl(\wt H, \{h_{\wt H}\}\bigr)\to \Sh\bigl(\Res_{F_0/\BQ}H, \{h_H\}\bigr),\\
	\Sh\bigl(\wt G, \{h_{\wt G}\}\bigr)\to \Sh\bigl(\Res_{F_0/\BQ}G, \{h_{G}\}\bigr),\\
	\Sh\bigl(\wtHG, \{h_{\wtHG}\}\bigr)\to \Sh\bigl(\Res_{F_0/\BQ}(H\times G), \{h_{H\times G}\}\bigr). 
\end{gathered}
\end{equation}
\end{altenumerate}
\end{remark}

\begin{remark}
Let us finally make our Shimura varieties more concrete.  We consider the case of $\Sh(\wt G, \{h_{\wt G}\})$; the other Shimura varieties are analogous.  In terms of the product decomposition $\wt G_\BR \cong Z^\BQ_\BR \times \prod_{\varphi \in \Phi} \U(W_\varphi)$ induced by by (\ref{proddec}), the conjugacy class $\{h_{\wt G}\}$ is the product of $\{h_{Z^\BQ}\}$ with the $\U(W_\varphi)(\BR)$-conjugacy class $\{h_{G,\varphi}\}$ for each $\varphi \in \Phi$, where $h_{G,\varphi}$ denotes the $\varphi$-component of the cocharacter $h_G$ defined in Remark \ref{others}(\ref{ggp}).  The conjugacy class $\{h_{Z^\BQ}\}$ consists of a single element; so does $\{h_{G,\varphi}\}$ for $\varphi \neq \varphi_0$, since in this case $h_{G,\varphi}$ is the trivial cocharacter.  For $\varphi = \varphi_0$, in terms of the basis for $W_{\varphi_0}$ chosen above, $h_{G,\varphi_0}$ is the cocharacter
\[
   h_{G,\varphi_0} \colon z \mapsto \diag(z/\ov z, 1,\dotsc, 1).
\]
The conjugacy class $\{h_{G,\varphi_0}\}$ then identifies with the open subset $\CD_{\varphi_0} \subset \BP(W_{\varphi_0})(\BC)$ of positive-definite lines for the hermitian form (send $h \in \{h_{G,\varphi_0}\}$ to the $-1$-eigenspace of $h({\sqrt{-1}})$; we remark that $\CD_{\varphi_0}$ is also isomorphic to the open unit ball in $\BC^{n-1}$).  Thus for $K_{\wt G} \subset \wt G(\BA_f)$ an open compact subgroup, we obtain the presentation
\[
   \Sh_{K_{\wt G}}\bigl(\wt G, \{h_{\wt G}\}\bigr) (\BC) = \wt G(\BQ) \big\backslash \big[\CD_{\varphi_0}\times \wt G(\BA_f)/K_{\wt G} \big] ,
\]
where the action of $\wt G(\BQ)$ is diagonal by the translation action on $\wt G(\BA_f)$ and by the action on $\CD_{\varphi_0}$ given via
\[
   \wt G(\BQ)\to \wt G_{\ad}(\BR)\to {\rm PU}(W_{\varphi_0})(\BR) .
\]
\end{remark}

\subsection{The moduli problem over $E$}\label{moduli problem over E}
In this subsection we define  moduli problems on the category of schemes over $\Spec E$ for the three Shimura varieties above.  Since this is almost identical in each case, let us do this for $\Sh(\wt G, \{h_{\wt G}\})$, and only indicate briefly the modifications needed for the other two (mostly for $\Sh(\wtHG, \{h_{\wtHG}\})$).  We will only consider open compact subgroups $ {K_{\wt G}}\subset \wt G(\BA_f)$ which, with respect to the product decomposition \eqref{proddec}, are of the form
\begin{equation}\label{K_wtG}
   K_{\wt G}= K_{Z^\BQ} \times K_G,
\end{equation}
where $K_G\subset G(\BA_{F_0,f})$ is an open compact subgroup and where $K_{Z^\BQ}\subset {Z^\BQ}(\BA_f)$ is the unique maximal compact subgroup 
\begin{equation}\label{K_Z^BQ}
   K_{Z^\BQ} := Z^\BQ(\wh\BZ) = \bigl\{\,z\in (O_F\otimes \wh\BZ)^\times \bigm| \Nm_{F/F_0}(z)\in \wh\BZ^\times \,\bigr\}.
\end{equation}
(Note that $Z^\BQ$ is defined over $\Spec \BZ$ in an obvious way.)

Before doing this, let us first introduce an auxiliary moduli problem $M_0$ over $E$. In fact, for use in the construction of integral models later, we will define a moduli problem $\CM_0$ over $O_E$ whose generic fiber will be $M_0$. For a locally noetherian $O_E$-scheme $S$, we define $\CM_0(S)$ to be the groupoid of triples $(A_0, \iota_0, \lambda_0)$, where 
\begin{altitemize}
\item $A_0$ is an abelian variety over $S$ with an $O_F$-action $\iota_0\colon O_F\to \End(A_0)$, which satisfies the Kottwitz condition of signature $((0, 1)_{\varphi\in\Phi})$, i.e., 
\begin{equation}\label{KottwitzA0}
   \charac\bigl(\iota(a)\mid\Lie A_0\bigr) 
	   = \prod_{\varphi\in\Phi}\bigl(T-\ov\varphi(a)\bigr)
		\quad\text{for all}\quad
		a\in O_F;
\end{equation}
and
\item $\lambda_0$ is a principal polarization of $A_0$ whose Rosati involution induces on $O_F$, via $\iota_0$, the nontrivial Galois automorphism of $F/F_0$ .
\end{altitemize}
A morphism between two objects $(A_0,\iota_0,\lambda_0)$  and $(A'_0,\iota'_0,\lambda'_0)$ is an $O_F$-linear isomorphism $\mu_0\colon A_0\to A'_0$ under which $\lambda_0'$ pulls back to $\lambda_0$. Then $\CM_0$ is a Deligne--Mumford stack, finite and \'etale over $\Spec O_E$, cf.\ \cite[Prop.\ 3.1.2]{Ho-kr}.\footnote{Strictly speaking loc.\ cit.\ is stated only for CM algebras and CM types $\Phi$ which are of a rather special sort, but the proof relies only on the very general Th.\ 2.2.1 in \cite{Ho-kr} and applies equally well to our situation.} (In fact, in loc.\ cit.,\ $\CM_0$ is defined over the ring of integers $O_{E_\Phi}$ in the reflex field of $\Phi$, which is contained in $O_E$, cf.\ Remark \ref{others}(\ref{Z^Q}).)  We let $M_0$ denote the generic fiber of $\CM_0$. 

Unfortunately, it may happen that $\CM_0$ is empty. In order to circumvent this problem, we also introduce the following variant of $\CM_0$, cf.\ \cite[Def.\ 3.1.1]{Ho-kr}. Fix a non-zero ideal $\fka$ of $O_{F_0}$. Then we introduce the Deligne--Mumford stack $\CM_0^\fka$  of triples $(A_0, \iota_0, \lambda_0)$ as before, except that we replace the condition that $\lambda_0$ is principal by the condition that $\ker \lambda_0=A_0[\fka]$. Then, again, $\CM_0^\fka$ is finite and \'etale over $\Spec O_E$, cf.\ \cite[Prop.\ 3.1.2]{Ho-kr}.  

If $\CM_0^\fka$ is non-empty, then the complex fiber $\CM^\fka_0 \otimes_{O_E} \BC$ is isomorphic to a finite number of copies of $\Sh_{K_{Z^\BQ}}(Z^\BQ, \{h_{Z^\BQ}\})$.  More precisely, let $\CL_\Phi^\fka$ be the set of isomorphism classes of pairs $(\Lambda_0,\aform_0)$ consisting of a locally free $O_F$-module $\Lambda_0$ of rank one equipped with a non-degenerate alternating form $\aform_0\colon \Lambda_0 \times \Lambda_0 \to \BZ$ such that $\la ax, y \ra_0 = \la x, \ov a y\ra_0$ for all $x,y \in \Lambda_0$ and $a \in O_F$, such that $x \mapsto \la \sqrt\Delta x, x \ra_0$ is a negative-definite quadratic form on $\Lambda_0$, and such that the dual lattice $\Lambda_0^\vee$ of $\Lambda_0$ inside $\Lambda_0 \otimes_\BZ \BQ$ equals $\fka\i \Lambda_0$.  Then $\CL_\Phi^\fka$ is a finite set.\footnote{Indeed, let $\fkd$ denote the different of $F/\BQ$, and set $\fkb := \sqrt\Delta \fkd\i\fka$. Let $\CH_{\ll 0}^\fkb$ be the set of isomorphism classes of hermitian $O_F$-modules $(\Lambda_0,\sform_0)$ of rank one (with $\sform_0$ valued in $F$) such that the $F$-hermitian space $\Lambda_0 \otimes_\BZ \BQ$ has signature $(0,1)$ at every archimedean place of $F_0$, and such that the dual lattice $\Lambda_0^*$ of $\Lambda_0$ inside $\Lambda_0 \otimes_\BZ \BQ$ with respect to $\sform_0$ equals $\fkb\i \Lambda_0$.  Then $(\Lambda_0,\sform_0) \mapsto (\Lambda_0, \tr_{F/\BQ}\sqrt\Delta\i \sform_0)$ defines a bijection $\CH_{\ll 0}^\fkb \isoarrow \CL_\Phi^\fka$.  It is easy to deduce from finiteness of the class group for $F$ and the product formula for hermitian spaces that $\CH_{\ll 0}^\fkb$ is finite.  Note that in the particular situation of \cite{KR-U2}, where $F_0 = \BQ$ and $\Delta$ is the discriminant of $F$, one has $\fka = \fkb$, but $\fka$ and $\fkb$ certainly may differ in general.} If $(\Lambda_0,\aform_0)$ defines a class in $\CL_\Phi^\fka$, then using $\Phi$ to identify $F \otimes_\BQ \BR \isoarrow \BC^\Phi$, and hence to define a complex structure on $\Lambda_0 \otimes_\BZ \BR$, the quotient $(\Lambda_0 \otimes_\BZ \BR) / \Lambda_0$ is a complex torus which defines a $\BC$-point on $\CM_0^\fka$.  On the other hand, given a $\BC$-point $(A_0,\iota_0,\lambda_0)$ on $\CM_0^\fka$, the first homology group $H_1(A_0(\BC),\BZ)$, endowed with its Riemann form induced by $\lambda_0$, defines a class in $\CL_\Phi^\fka$.  In this way $\CL_\Phi^\fka$ is in bijection with the isomorphism classes of objects in $\CM_0^\fka(\BC)$.

These inverse processes give rise to a decomposition of $\CM_0^\fka$ into a union of Shimura varieties, in terms of the
following equivalence relation on $\CL_\Phi^\fka$: define $\Lambda_0 \sim \Lambda_0'$ if $\Lambda_0 \otimes_\BZ \wh\BZ$ and $\Lambda_0' \otimes_\BZ \wh\BZ$ are $\wh O_F$-linearly similar up to a factor in $\wh\BZ^\times$, and if $\Lambda_0 \otimes_\BZ \BQ$ and $\Lambda_0' \otimes_\BZ \BQ$ are $F$-linearly similar up to a factor in $\BQ^\times$.  We have the following.

\begin{lemma}\label{decompM0 lem}
The stack $\CM_0^\fka$ admits a decomposition into open and closed substacks,
\begin{equation}\label{decompM0}
   \CM_0^\fka = \coprod_{\xi \in \CL_\Phi^\fka/\sim} \CM_0^{\fka,\xi},
\end{equation}
induced by sending an object $(A_0,\iota_0,\lambda_0)$ in $\CM_0^\fka(\BC)$ to the $\sim$-class of $H_1(A_0(\BC),\BZ)$ endowed with its Riemann form.
\end{lemma}

\begin{proof}
Since $\CM_0^\fka$ is finite and \'etale over $\Spec O_E$, it suffices to demonstrate the asserted decomposition in the generic fiber $M_0^\fka$,
\begin{equation}\label{M0a decomp}
   M_0^\fka = \coprod_{\xi \in \CL_\Phi^\fka/\sim} M_0^{\fka,\xi}.
\end{equation}
To establish \eqref{M0a decomp},  again by \'etaleness, it suffices to show that the $\sim$-class of $H_1(A_0(\BC),\BZ)$ is constant on the $\Aut(\BC/E)$-orbit of each object $(A_0,\iota_0,\lambda_0) \in M_0^\fka(\BC)$.  So let $\sigma \in \Aut(\BC/E)$.  Then $\sigma$ identifies the Tate modules $\wh T (A_0) \isoarrow \wh T (A_0^\sigma)$ compatibly with the Weil pairings on the two sides up to the similitude factor given by the image of $\sigma$ under the cyclotomic character $\Aut(\BC) \to \wh\BZ^\times$.  Hence, by the compatibility between singular homology and the Tate module, $H_1(A_0(\BC),\wh\BZ) $ and $H_1(A_0^\sigma(\BC),\wh\BZ)$ are $\wh O_F$-linearly similar up to a factor in $\wh\BZ^\times$.  On the other hand, the CM abelian variety $A_0$ is isomorphic to one defined over the algebraic closure of $E$ in $\BC$, and hence by \cite[Th.\ A.2.3.5]{CCO} there exists a quasi-isogeny of complex abelian varieties $A_0 \to A_0^\sigma$ pulling $\lambda_0^\sigma$ back to a $\BQ_{>0}$-multiple of $\lambda_0$.  Hence $H_1(A_0(\BC),\BQ)$ and $H_1(A_0^\sigma(\BC),\BQ)$ are $F$-linearly similar up to a factor in $\BQ^\times$, as desired.

\end{proof}

If $\CM_0^\fka \neq \emptyset$, then for fixed $\xi \in \CL_\Phi^\fka/{\sim}$, the complex fiber $\CM_0^{\fka,\xi} \otimes_{O_E} \BC$ of the summand in \eqref{decompM0} is canonically isomorphic to the Shimura variety $\Sh_{K_{Z^\BQ}}(Z^\BQ, \{h_{Z^\BQ}\})$; this can be checked similarly to \cite[Props.\ 4.3, 4.4]{KR-U2}, or see \cite[4.12, 4.18--4.20]{De}.

\begin{remark}
\begin{altenumerate}
\item\label{fka rel pr} Given finitely many prime numbers $p_1, \ldots, p_r$, there always exists an ideal $\fka$ relatively prime to $p_1,\ldots, p_r$ such that $\CM_0^\fka$ is non-empty.
\item\label{rem exa satisfied} If $F/F_0$ is ramified at some finite place, then $\CM_0^\fka$ is non-empty for any $\fka$, cf.\ \cite[proof of Prop.\ 3.1.6]{Ho-kr}.  A special case of this is when $F= K F_0$, where $K$ is an imaginary quadratic field and  the discriminants of $K/\BQ$ and $F_0/\BQ$ are relatively prime. We further remark that in the context of the  \emph{global} integral models we define in Section \ref{section global} below, we will eventually impose conditions on the hermitian spaces $W$ and $W^\flat$ that force $F/F_0$ to be ramified at some finite place, cf.\ Remark \ref{latticecond herm}.
\item When $F_0 = \BQ$, the set $\CL_\Phi^\fka/{\sim}$ is a singleton, so that the decomposition \eqref{decompM0} is trivial. 
\item The decomposition \eqref{decompM0} in fact holds over $\Spec O_{E_\Phi}$, via the same proof.  We have worked over $\Spec O_E$ only for notational simplicity throughout the rest of the paper.
\item Let $(A_0,\iota_0,\lambda_0)\in\CM_0^\fka (k)$, where $k$ is a perfect field of positive characteristic $p$. Here is the recipe to determine on which summand $\CM_0^{\fka, \xi}$ this point lies. Let $(\wt A_0,\wt \iota_0,\wt\lambda_0)$ be the canonical lift of $(A_0,\iota_0,\lambda_0)$ over the ring of $p$-typical Witt vectors $W(k)$.  Then $(\wt A_0,\wt \iota_0,\wt\lambda_0)$ is defined over a subring $R \subset W(k)$ which is finitely generated over $\BZ$, and after choosing an embedding $R \to \BC$, we obtain a point of $\CM_0^{\fka}(\BC)$ which determines $\xi\in \CL_\Phi^\fka/{\sim}$ as in Lemma \ref{decompM0 lem}. This is the label of the summand containing $(A_0,\iota_0,\lambda_0)$.
\end{altenumerate}\label{rem exa}
\end{remark}

In the following, we fix an ideal $\fka$ such that $\CM_0^\fka$ is non-empty. We continue to denote the generic fiber of this stack by $M_0^\fka$. We also fix $\xi\in \CL_\Phi^\fka/{\sim}$. We now define a groupoid $M_{K_{\wt G}}(\wt G)$ fibered over the category of locally noetherian schemes over $E$. Here, to lighten  notation, we have suppressed the ideal $\fka$ and the element $\xi$. For such a scheme $S$, the objects of $M_{K_{G}}(\wt G)(S)$ are collections $(A_0,\iota_0,\lambda_0,A,\iota,\lambda,\ov\eta)$, where
\begin{altitemize}
\item $(A_0, \iota_0, \lambda_0)$ is an object of $M^{\fka, \xi}_0(S)$;
\item $A$ is an abelian scheme over $S$ with an $F$-action $\iota\colon F\to\End^\circ(A)$ satisfying the Kottwitz condition of signature $((1, n-1)_{\varphi_0}, (0, n)_{\varphi\in\Phi\ssm\{\varphi_0\}})$, i.e.,  
\begin{equation}\label{kottwitzF}
   \charac\bigl(\iota(a)\mid\Lie A\bigr) 
	   = \bigl(T-\varphi_0(a)\bigr)\bigl(T-\varphi_0(\bar a)\bigr)^{n-1}\prod_{\varphi\in\Phi\ssm \{\varphi_0\}}\bigl(T-\varphi(\bar a)\bigr)^n
	\quad\text{for all}\quad
	 a\in O_F;
\end{equation}
\item $\lambda$ is a polarization $A \to A^\vee$ whose Rosati involution induces on $F$, via $\iota$, the nontrivial Galois automorphism of $F/F_0$; and
\item $\ov\eta$ is a ${K_{\wt G}}$-level structure, by which we mean a ${K_{G}}$-orbit (equivalently, a ${K_{\wt G}}$-orbit, where $K_{\wt G}$ acts through its projection $K_{\wt G}\to K_G$) of $\BA_{F, f}$-linear isometries
\begin{equation}\label{level str}
   \eta\colon \wh V(A_0, A)\simeq -W\otimes_F\BA_{F, f}.
\end{equation}
\end{altitemize}
Here we regard the right-hand side of \eqref{level str} as a constant smooth $\BA_{F,f}$-sheaf on $S$, and
\begin{equation*}
   \wh V(A_0,A) := \Hom_{F}\bigl(\wh V(A_0), \wh V(A)\bigr) ,
\end{equation*}
endowed with its natural $\BA_{F, f}$-valued hermitian form $h_A$, defined by
\begin{equation}\label{h_A def}
   h_A(x, y) := \lambda_0^{-1}\circ y^\vee\circ\lambda\circ x\in\End_{\BA_{F, f}}\bigl(\wh V(A_0)\bigr)=\BA_{F, f},
\end{equation}
cf.\ \cite[\S2.3]{KR-U2}.  Over each connected component of $S$, upon choosing a geometric point $\ov s \to S$, we may view $\wh V (A_0, A)$ as the space $\Hom_{\BA_{F,f}}(\wh V(A_{0,\ov s}), \wh V(A_{\ov s}))$ endowed with its $\pi_1(S,\ov s)$-action, and in this way we require the orbit $\ov \eta$ to be $\pi_1(S,\ov s)$-invariant; comp.\ \cite[Rem.\ 4.2]{KR-U2}.  We note that the notion of level structure is independent of the choice of $\ov s$ on each connected component of $S$.

A morphism between two objects $(A_0,\iota_0,\lambda_0,A,\iota,\lambda,\ov\eta)$ and $(A'_0,\iota'_0,\lambda'_0,A',\iota',\lambda',\ov\eta')$ is  given by an  isomorphism $\mu_0\colon (A_0, \iota_0, \lambda_0) \isoarrow (A'_0, \iota'_0, \lambda'_0)$ in $M_0^{\fka,\xi}(S)$  and an $F$-linear quasi-isogeny $\mu\colon A\to A'$ pulling $\lambda'$ back to $\lambda$ and $\ov\eta'$ back to $\ov\eta$. 

\begin{remark}
\begin{altenumerate}
\item\label{Kottwitzr}
Let $r \colon \Hom(F,\BC) \to \{0,1,n-1,n\}$, $\varphi \mapsto r_\varphi$, be the function defined by
\begin{equation}\label{r main body}
\begin{aligned}
   r_\varphi :=
	\begin{cases}
		1,  &  \varphi = \varphi_0;\\
		0,  &  \varphi \in \Phi \ssm \{\varphi_0\};\\
		n-r_{\ov\varphi},  &  \varphi \notin \Phi.
	\end{cases}
\end{aligned}
\end{equation}
Then the Kottwitz condition \eqref{kottwitzF} is
\[
   \charac\bigl(\iota(a)\mid\Lie A\bigr) 
	   = \prod_{\varphi\in\Hom(F,\BC)} \bigl(T-\varphi(a)\bigr)^{r_\varphi} 
		\quad\text{for all}\quad
		a \in O_F,
\]
comp.\ \cite[(8.4)]{RZ14}.
\item The intervention of the sign on the right-hand side of \eqref{level str} arises from our conventions on the signatures of $W$ at the archimedean places given at the beginning of Section \ref{Sh data} and on the signatures in the Kottwitz conditions \eqref{KottwitzA0} and \eqref{kottwitzF}, cf.\ the proof of Proposition \ref{moduliforshim} below.  If one took the opposite signatures for $W$ and in the two Kottwitz conditions, then no sign would be needed.
\end{altenumerate}\label{M remarks}
\end{remark}

The following proposition is a special case of Deligne's description of Shimura varieties of PEL type.

\begin{proposition}\label{moduliforshim}
$M_{K_{\wt G}}(\wt G)$ is a Deligne--Mumford stack  smooth  of relative dimension $n-1$ over $\Spec E$.  The coarse moduli scheme of $M_{K_{\wt G}}(\wt G)$ is a quasi-projective scheme over $\Spec E$, naturally isomorphic to the canonical model of $\Sh_{K_{\wt G}}(\wt G, \{h_{\wt G}\})$. For ${K_{\wt G}}$ sufficiently small, the forgetful morphism $M_{K_{\wt G}}(\wt G)\to M^{\fka, \xi}_0$ is relatively representable. 
\end{proposition}

\begin{proof}
We will content ourselves with exhibiting a map, which turns out to be an  isomorphism,
$$
M_{K_{\wt G}}(\wt G)\otimes_{E}\BC \to  \Sh_{K_{\wt G}}\bigl(\wt G, \{h_{\wt G}\}\bigr) .
$$
By Remark \ref{others}(\ref{ggp}), the target is the product of Shimura varieties
\[
   \Sh_{K_{Z^\BQ}}\bigl(Z^\BQ,\{h_{Z^\BQ}\}\bigr) \times \Sh_{K_G}\bigl(\Res_{F/F_0} G,\{h_G\}\bigr).
\]
For the map into the first factor, we simply compose the forgetful map $M_{K_{\wt G}}(\wt G)\otimes_{E}\BC \to M_0^{\fka, 
\xi}\otimes_{E}\BC$ with the isomorphism $M_0^{\fka, \xi}\otimes_{E}\BC \simeq \Sh_{K_{Z^\BQ}}(Z^\BQ,\{h_{Z^\BQ}\})$.

To explain the map into the second factor, let $(A_0,\iota_0,\lambda_0,A,\iota,\lambda,\ov\eta)$ be a $\BC$-valued point of $M_{K_{\wt G}}(\wt G)$. Let $\CH := H_1(A, \BQ)$ and $\CH_0 := H_1(A_0, \BQ)$. The polarization  $\lambda$ endows $\CH$ with a $\BQ$-valued alternating form $\aform$ satisfying $\la \iota(a)x,y\ra = \la x, \iota(\ov a)y\ra$ for all $a \in F$, and such that the induced form $x,y \mapsto \la \sqrt{-1} \cdot x, y \ra$ on $\CH \otimes_\BQ \BR$ is symmetric and positive definite, where multiplication by $\sqrt{-1}$ is defined in terms of the right-hand side of the canonical isomorphism $\CH_\BR \cong \Lie A$.  Similarly, $\lambda_0$ endows $\CH_0$ with a Riemann form $\aform_0$.

Let $V(A_0, A):=\Hom_F(\CH_0, \CH)$. Then $V(A_0, A)$ is an $F$-vector space of the same dimension as $W$, and we make it into an $F/F_0$-hermitian space by defining the pairing $(\alpha,\beta)$ to be the composite
\[
   \xymatrix@R=0ex{
      \CH_0 \ar[r]^-\alpha  &  \CH \ar[r]  &  \CH^\vee \ar[r]^-{\beta^\vee}  &  \CH_0^\vee \ar[r]  &  \CH_0,\\
        &  x \ar@{|->}[r]  &  \la -, x \ra
   }
\]
where the checks denote $\BQ$-linear duals and the last arrow is the inverse of $y \mapsto \la -, y \ra_0$; this composite is an $F$-linear endomorphism of the one-dimensional $F$-vector space $\CH_0$, and hence identifies with an element in $F$.  Clearly $V(A_0, A) \otimes_F \BA_{F,f} \cong \wh V(A_0, A)$ as hermitian spaces. Hence, by the existence of a level structure, $V(A_0, A) \otimes_F \BA_{F,f} \simeq -W \otimes_F \BA_{F,f}$.  Furthermore, it is easy to see that the Kottwitz conditions \eqref{KottwitzA0} and \eqref{kottwitzF} imply that $V(A_0, A)_\varphi$ has signature $(n-1,1)$ if $\varphi=\varphi_0$ and $(n,0)$ if $\varphi\in\Phi\ssm\{\varphi_0\}$. Hence, by the Hasse principle for hermitian spaces, $V(A_0, A)$ and $-W$ are isomorphic. Choose an isometry $j\colon V(A_0, A)\isoarrow -W$.  Using the complex structures on $\CH_{0,\BR}$ and $\CH_\BR$, let $z \in \BC^\times$ act on $V(A_0, A)$ by sending the $F$-linear map $\alpha$ to $z\alpha z\i$. This defines a homomorphism $\BC^\times \to \U(V(A_0, A))(\BR)$, and composing this with $j_*\colon \U(V(A_0, A))(\BR) \isoarrow \U(-W)(\BR) = \U(W)(\BR)$ gives an element in $\{h_G\}$. The level structure $\ov \eta$ corresponds to an  element of $(\Res_{F_0/\BQ} G)(\BA_f)/K_{G}$, and eliminating the choice of $j$ corresponds to dividing out by the action of $(\Res_{F_0/\BQ} G)(\BQ)$. 
\end{proof}

An analogous description holds for the model $M_{K_{\wt H}}(\wt H)$ of the Shimura variety $\Sh_{K_{\wt H}}(\wt H, \{h_{\wt H}\})$ (replace $n$ by $n-1$, and $W$ by $W^\flat$).  

There is also an analog for the Shimura variety $\Sh_{K_{\wtHG}}(\wtHG, \{h_{\wtHG}\})$. In this case we take the level subgroup to be of the form
\begin{equation}\label{K_wtHG decomp}
   K_{\wtHG} = K_{Z^\BQ} \times K_H \times K_G,
\end{equation}
where as always $K_{Z^\BQ}$ is the subgroup \eqref{K_Z^BQ}, and $K_H \subset H(\BA_{F_0,f})$ and $K_G \subset G(\BA_{F_0,f})$ are open compact subgroups.  The value of the corresponding moduli functor $M_{K_{\wtHG}}(\wtHG)$ on a locally noetherian scheme $S$ over $E$ is the set of isomorphism classes of tuples
\[
   \bigl(A_0,\iota_0,\lambda_0,A^\flat,\iota^\flat,\lambda^\flat, A, \iota,\lambda, \ov{(\eta^\flat, \eta)}\bigr) ,
\]
 where the last entry is a pair of  $\BA_{F, f}$-linear isometries
\[
   \eta^\flat \colon \wh V(A_0, A^\flat)\simeq -W^\flat\otimes_F\BA_{F, f}
	\quad\text{and}\quad
	\eta\colon \wh V(A_0, A)\simeq -W\otimes_F\BA_{F, f} ,
\]
modulo the action of ${K_{H}\times K_G}$.  In other words, the moduli functor $M_{K_{\wtHG}}(\wtHG)$ is simply the fibered product $M_{K_{\wt H}}(\wt H) \times_{M^{\fka,\xi}_0} M_{K_{\wt G}}(\wt G)$.

In terms of these moduli problems, the injective morphisms \eqref{Sh inj} can be described as follows.  After possibly scaling  the special vector $u \in W$, we may and do assume that the norm $(u,u)$ is in $O_{F_0}$. Further assume that ${K_{H}}\subset H(\BA_{F_0,f})\cap {K_{ G}}$. Then the first morphism of Shimura varieties in \eqref{Sh inj} arises by base change from $E$ to $\BC$ from the functor morphism
\begin{equation}\label{modembHG}
   \begin{gathered}
   \xymatrix@R=0ex{
	   M_{K_{\wt H}} (\wt H) \ar[r]  &  M_{K_{\wt G}}(\wt G)\\
		\bigl( A_0,\iota_0,\lambda_0,A^\flat,\iota^\flat,\lambda^\flat,\ov\eta^\flat \bigr) \ar@{|->}[r]
		   & \bigl( A_0,\iota_0,\lambda_0,A^\flat \times  A_0,\iota^\flat \times \iota_0,\lambda^\flat \times \lambda_0(u),\ov{\eta}\bigr) .
	}
   \end{gathered}
\end{equation}
Here $\lambda_0(u)= -(u,u) \lambda_0$, which is indeed a polarization by \cite[Lem.\ 1.4]{Zi} since $-(u,u)$ is totally positive. Furthermore, the isomorphism $\eta\colon \wh V(A_0,A^\flat\times A_0) \simeq -W \otimes_F \BA_{F, f} \bmod K_{G}$ is given, with respect to the decomposition
\[
  \Hom_F\bigl( \wh V(A_0), \wh V(A^\flat \times A_0)\bigr) 
     \cong \Hom_F\bigl( \wh V(A_0), \wh V(A^\flat) \oplus \wh V(A_0)\bigr)
	  \cong \wh V(A_0,A^\flat) \oplus { \BA}_{F, f},
\]
by the trivialization
\[
   \eta^\flat\colon \wh V(A_0,A^\flat) \simeq -W^\flat \otimes_F \BA_{F, f} \bmod K_{ H}
\]
in the first summand, and by identifying the basis element $1$ in the second summand with $u \otimes 1 \in -W \otimes_F \BA_{F,f}$.  The morphism \eqref{modembHG} is finite and unramified.
 
The second morphism in \eqref{Sh inj} then arises from the graph morphism of \eqref{modembHG},
\begin{equation}\label{modembHHG}
   M_{K_{\wt H}} (\wt H) \to M_{K_{\wtHG}}(\wtHG) = M_{K_{\wt H}}(\wt H) \times_{M^{\fka,\xi}_0} M_{K_{\wt G}}(\wt G).
\end{equation}
The morphism \eqref{modembHHG} is a closed embedding.

\begin{remark}\label{toroidcomp}
For $F_0 \neq \BQ$, the Shimura varieties above are compact.  For $F_0 = \BQ$, it may happen that the Shimura variety  $M_{K_{\wtHG}}(\wtHG)$ is non-compact. In fact, this will be automatic when $n\geq 3$.  In this case, we will need to use its canonical toroidal compactification, cf.\ \cite[\S 2]{Ho-kr2}.  
\end{remark}

\section{Semi-global integral models}\label{section semi-global}
Fix a prime number $p$ and an embedding $\wt\nu\colon \ov\BQ \to \ov\BQ_p$. This determines a $p$-adic place $\nu$ of $E$ and, via $\varphi_0$, places $v_0$ of $F_0$ and $w_0$ of $F$.  In this section we are going to define ``semi-global'' integral models over $O_{E,(\nu)}$ of the moduli spaces introduced in Section \ref{Shimura varieties}, in the case of various level structures at $p$.  We denote by $\CV_p$ the set of places $v$ of $F_0$ over $p$.  Throughout this section, we assume that the ideal $\fka$ occurring in the definition of $\CM_0^\fka$ is prime to $p$, cf.\ Remark \ref{rem exa}(\ref{fka rel pr}). We also fix $\xi\in \CL_\Phi^\fka/{\sim}$.

\subsection{Hyperspecial level at $v_0$}\label{subsec hyper} 
In this case we assume that the place $v_0$ is unramified over $p$, and that $v_0$ either splits in $F$ or is inert in $F$ and the hermitian space $W_{v_0}$ is split. We also assume $p \neq 2$ if there is any $v \in \CV_p$ which is non-split in $F$.\footnote{In the construction of semi-global models in \cite[\S\S4--5]{RSZ4}, this assumption is relaxed to the assumption that, when $p = 2$, all $v \in \CV_2$ are unramified in $F$.} We are going to define smooth models over $O_{E,(\nu)}$. 

For each $v\in \CV_{p}$, choose a vertex lattice $\Lambda_{v}$ in the $F_v/F_{0,v}$-hermitian space $W_{v}$.  By our assumptions on $v_0$ in the previous paragraph, we may and do take $\Lambda_{v_0}$ to be self-dual. Recalling the subgroup $K_{\wt G} = K_{Z^\BQ} \times K_G$ from \eqref{K_wtG}, we take $K_G$ to be of the form
\[
   K_G = K_G^p \times K_{G,p},
\]
where $K_G^p \subset G(\BA_{F_0,f}^p)$ is arbitrary, and where the level subgroup at $p$ is the product
\begin{equation}\label{prodlocK}
   K_{G,p} :=\prod_{v\in \CV_p}K_{G, v} \subset G(F_0\otimes \BQ_p)=\prod_{v\in \CV_p}G(F_{0, v}),
\end{equation}
with $K_{G, v}$ the stabilizer of $\Lambda_v$ in $G(F_{0,v})$.

We formulate a moduli problem over $\Spec O_{E,(\nu)}$ as follows.  To each locally noetherian $O_{E,(\nu)}$-scheme $S$, we associate the groupoid of tuples $(A_0,\iota_0,\lambda_0,A,\iota,\lambda,\ov\eta^p)$, where $(A_0,\iota_0,\lambda_0)$ is an object of $\CM^{\fka, \xi}_0(S)$.  Furthermore:
\begin{altitemize}
\item $(A,\iota)$ is an abelian scheme over $S$ with an $O_{F}\otimes\BZ_{(p)}$-action $\iota$ satisfying the Kottwitz condition \eqref{kottwitzF} of signature $((1, n-1)_{\varphi_0}, (0, n)_{\varphi\in\Phi\ssm\{\varphi_0\}})$.
\item $\lambda$ is a polarization on $A$ whose Rosati involution induces on $O_{F}\otimes\BZ_{(p)}$ the non-trivial Galois automorphism of $F/F_0$, subject to the following condition. First note that the action of the ring $O_{F_0}\otimes\BZ_p \cong \prod_{v \in \CV_p} O_{F_0,v}$ on the $p$-divisible group $A[p^\infty]$ induces a decomposition of $p$-divisible groups, 
\begin{equation}\label{decofpdivgp}
   A[p^\infty] =  \prod_{v \in \CV_p} A[v^\infty] .
\end{equation}
Since $\Ros_\lambda$ is trivial on $O_{F_0}$, $\lambda$ induces a polarization $\lambda_v \colon A[v^\infty] \to A^\vee[v^\infty] \cong A[v^\infty]^\vee$ for each $v$.
The condition we impose is that $\ker\lambda_v$ is contained in $A[\iota(\pi_v)]$ of rank $\#(\Lambda_v^*/ \Lambda_v)$ for each $v \in \CV_p$.
\item $\ov\eta^p$  is a $K_G^p$-orbit of $\BA_{F,f}^p$-linear isometries
\begin{equation}\label{levelprimetop}
   \eta^p\colon \wh V^p(A_0,A) \simeq -W \otimes_F \BA_{F,f}^p,
\end{equation}
where
\[
   \wh V^p(A_0,A) := \Hom_{F}\bigl(\wh V^p(A_0), \wh V^p(A)\bigr),
\]
and the hermitian form on $\wh V^p(A_0,A)$ is the evident prime-to-$p$ analog of (\ref{h_A def}).
\end{altitemize}
We  also impose for each $v\neq v_0$ over $p$ the \emph{sign condition} and the \emph{Eisenstein condition}. Let us explain these conditions. 

The sign condition at $v$ is only non-empty when  $v$ does not split in $F$, in which case it demands that at every point $s$ of $S$,
\begin{equation}\label{signcond}
   \inv^{r}_v(A_{0,s}, \iota_{0,s}, \lambda_{0,s}, A_s, \iota_s, \lambda_s) = \inv_v(-W_v) .
\end{equation}
Here the left-hand side is the sign factor defined in \eqref{deftrmodsign} and \eqref{defmodsign} in Appendix \ref{sign invariants} (in the definition of (\ref{defmodsign}), one may use the embedding $\wt\nu$ fixed at the beginning of this section). The right-hand side is the Hasse invariant of the hermitian space $-W_v$ defined above in (\ref{Hasse def}).  Note that by Proposition \ref{inv_v^r const}, the left-hand side of \eqref{signcond} is a locally constant function in $s$.

The Eisenstein  condition is only non-empty  when the base scheme $S$ has non-empty special fiber. In this case, we may even base change via $\wt\nu\colon O_{E, (\nu)}\to \ov\BZ_p$ (the ring of integers in $\ov \BQ_p$) and pass to completions and assume that $S$ is a scheme over $\Spf \ov\BZ_p$. Similarly to \eqref{decofpdivgp}, there is a decomposition of the $p$-divisible group $A[p^\infty]$,  
\begin{equation}\label{decomp by w}
   A[p^\infty] =  \prod_{w\mid p} A[w^\infty] ,
\end{equation}
where the indices range over the places $w$ of $F$ lying over $p$. Since we assume that $p$ is locally nilpotent on $S$, there is a natural isomorphism 
\begin{equation*}
   \Lie A\cong \Lie A[p^\infty]=\bigoplus_{w\mid p}  \Lie A[w^\infty].
\end{equation*}
For each place $w$, by the Kottwitz condition (\ref{kottwitzF}), the $p$-divisible group $A[w^\infty]$ is of height $n \cdot [F_{w}:\BQ_p]$ and dimension
\begin{equation}\label{r_w}
   \dim A[w^\infty] = \sum_{\varphi \in \Hom(F_w, \ov \BQ_p)} r_\varphi.
\end{equation}
Here $r_\varphi$ is as in \eqref{r main body}, and we have used the embedding $\wt\nu\colon \ov\BQ \to \ov\BQ_p$ to identify
\begin{equation}\label{Hom(F,Q) id}
   \Hom_\BQ(F, \ov\BQ) \simeq \Hom_\BQ(F, \ov\BQ_p) , 
\end{equation}
which in turn identifies
\begin{equation}\label{idlochom}
   \bigl\{\, \varphi \in \Hom_{\BQ}(F, \ov\BQ) \bigm| w_\varphi=w \,\bigr\} \simeq \Hom_{\BQ_p}(F_w, \ov \BQ_p),
\end{equation}
where $w_\varphi$ denotes the $p$-adic place in $F$ induced by $\wt\nu \circ \varphi$.

Now suppose that $w$ lies over a place $v$ different from $v_0$. Then the action of $F$ on $A[w^\infty]$ is of a \emph{banal} signature type, in the sense that each integer $r_\varphi$ occurring in (\ref{r_w}) is equal to $0$ or $n$, cf.\ Appendix \ref{appendix}.  Let $\pi = \pi_w$ be a uniformizer in $F_w$, and let $F_w^t \subset F_w$ be the maximal unramified subextension of $\BQ_p$. 
For each $\psi\in \Hom_{\BQ_p}(F^t_w, \ov\BQ_p)$, let 
\[
   A_\psi := \bigl\{\,\varphi\in  \Hom_{\BQ_p}(F_w, \ov\BQ_p) \bigm| \varphi|_{F_w^t}=\psi \text{ and }  r_\varphi=n \,\bigr\} .
\]
Set 
\[
   Q_{A_\psi}(T):=\prod_{\varphi\in A_\psi}\bigl(T-\varphi(\pi)\bigr)\in \ov\BZ_p[T] .
\]
Since we assume that $S$ is a scheme over $\Spf \ov\BZ_p$, there is a natural isomorphism
\[
   O_{F_w^t} \otimes_{\BZ_p} \CO_S \isoarrow \prod_{\psi \in \Hom_{\BQ_p}(F_w^t, \ov\BQ_p)} \CO_S,
\]
whose $\psi$-component is $\psi \otimes \id$.  Hence the $O_{F_w^t}$-action on $\Lie A[w^\infty]$ induces a decomposition
\begin{equation}\label{Lie A big decomp}
   \Lie A[w^\infty] \cong \bigoplus_{\psi \in \Hom_{\BQ_p}(F_w^t, \ov\BQ_p)} \Lie_\psi A[w^\infty].
\end{equation}
The Eisenstein condition at $v$ demands the identity of endomorphisms of $\Lie_\psi A[w^\infty]$,
\begin{equation}\label{Eisatv}
   Q_{A_\psi}\bigl(\iota(\pi) \mid \Lie_{\psi} A[w^\infty]\bigr)=0 
	\quad\text{for each}\quad
   w \mid v
   \quad\text{and each}\quad
	\psi\in  \Hom_{\BQ_p}(F^t_w, \ov\BQ_p) .
\end{equation}
This condition is the analog in our context of the  condition with the same name in \cite[(8.2)]{RZ14}.  The Kottwitz condition implies that the Eisenstein condition at $v$ is automatically satisfied when the one or two places $w$ over $v$  are unramified over $p$, cf.\ Lemma \ref{kott=>eis}. 

A morphism between two objects $(A_0, \iota_0, \lambda_0, A, \iota, \lambda, \ov\eta^p)$ and $(A_0', \iota_0', \lambda_0', A', \iota', \lambda', \ov\eta'^p)$ is given by an isomorphism
 $(A_0,\iota_0,\lambda_0) \isoarrow (A_0',\iota_0',\lambda_0')$ in $\CM^{\fka,\xi}_0(S)$ and a quasi-isogeny $A \to A'$ which induces an isomorphism
\[
   A[p^\infty] \isoarrow A'[p^\infty] ,
\]
compatible with $\iota$ and $\iota'$, with $\lambda$ and $\lambda'$, and with $\ov\eta^p$ and $\ov\eta^{p\prime}$.

\begin{theorem}\label{semi-global hyperspecial smooth}
The moduli problem just formulated is representable by a Deligne--Mumford stack $\CM_{K_{\wt G}}(\wt G)$ smooth over $\Spec O_{E,(\nu)}$. For $K_G^p$ small enough,  $\CM_{K_{\wt G}}(\wt G)$ is relatively representable over $\CM^{\fka,\xi}_0$.   Furthermore, the generic fiber $\CM_{K_{\wt G}} \times_{\Spec O_{E,(\nu)}} \Spec E$ is canonically isomorphic to $M_{K_{\wt G}}(\wt G)$. 
\end{theorem}

\begin{proof} Representability and relative representability are standard, cf.\ \cite[p.\ 391]{Ko}. Smoothness follows as usual from the theory of local models, cf.\ \cite{PRS}.  More precisely, the local model for $\CM_{K_{\wt G}}(\wt G)$ decomposes into a product of local models, one for each of the abelian schemes $A_0$ and $A$ in the moduli problem.  The local model corresponding to $A_0$ is \'etale because $\CM_0^{\fka,\xi}$ is.  Now, under the identifications \eqref{Hom(F,Q) id} and \eqref{idlochom}, we have
\begin{equation}\label{Hom decomp}
   \Hom_\BQ(F,\ov\BQ) \simeq \bigsqcup_{v \in \CV_p} \Hom_{\BQ_p}(F_v, \BQ_p).
\end{equation}
In this way the completion $E_\nu$ identifies with the join of the local reflex fields $E_{\Phi_v}$ and $E_{r|_v}$ in $\ov\BQ_p$ as $v$ varies through $\CV_p$, where $\Phi_v := \Phi \cap \Hom_{\BQ_p}(F_v,\ov\BQ_p)$ in terms of the identification (\ref{Hom decomp}), and where $r|_v\colon \Hom_{\BQ_p}(F_v,\ov\BQ_p) \to \BZ$ denotes the restriction of the function $r$ to the $v$-summand on the right-hand side of (\ref{Hom decomp}). The local model $M$ corresponding to $A$ then decomposes as
\begin{equation}\label{LM prod decomp}
   M = \prod_{v \in \CV_p} M_v \times_{\Spec O_{E_{r|_v}}} \Spec O_{E_\nu}.
\end{equation}
Here for $v \neq v_0$, by the Kottwitz condition (\ref{kottwitzF}), $M_v = M(F_v/F_{0,v},r|_v,\Lambda_v)$ is a \emph{banal} local model, i.e.\ of the form defined in Appendix \ref{appendix}. Hence $M_v = \Spec O_{E_{r|_v}}$ by Lemmas \ref{banal LM GL_n triv} and \ref{banal LM GU_n triv}.  (The same is true for the local model corresponding to $A_0$ at every $v \in \CV_p$, by these lemmas and Remark \ref{n=1}.)  The local model $M_{v_0}$ is smooth by \cite{Goertz}.
   
It remains to prove the last assertion. Let $S$ be a scheme over $E$, and let $(A_0, \iota_0, \lambda_0, A, \iota, \lambda, \ov \eta^p)$ be a point of $\CM_{K_{\wt G}}(S)$. We want to associate to this a point of $M_{K_{\wt G}}(S)$, i.e., we want to add the $p$-component of $\ov\eta$. The product of hermitian spaces $W \otimes_\BQ \BQ_p=\prod_{v \in \CV_p} W_v$ contains the lattice $\prod_{v \in \CV_p}\Lambda_v$, where $\Lambda_v$ is a vertex lattice in $W_v$. By assumption on the polarization $\lambda$, the product of hermitian spaces
\[
   V_p(A_0, A) := \Hom_F\bigl(V_p(A_0),  V_p(A)\bigr) = \prod_{v \in \CV_p} V_v(A_0, A)
\]
contains $\Hom_{O_F}(T_p(A_0),  T_p(A))$ as a product of vertex lattices, where the factor at each $v$ is of the same type as $\Lambda_v$.
Since $p\neq 2$ when there are non-split places in $\CV_p$, it follows that, if there exists an isometry $\eta_p\colon V_p(A_0, A)\simeq -W \otimes_\BQ \BQ_p$ at all, then there also exists one that maps these two vertex lattices of identical type into one another, and the class modulo $K_{G, p}$ of such an isometry is then uniquely determined. 

Hence we are reduced to showing that there exists an isometry $\eta_p$, i.e., the equality of Hasse invariants  $\inv_v(V_v(A_0, A))=\inv_v(-W_v)$ for all $v \in \CV_p$. By the Hasse principle and the product formula for hermitian spaces, it suffices to prove that for any $\BC$-valued point $(A_0, \iota_0, \lambda_0, A, \iota, \lambda, \ov \eta^p)$ of $\CM_{K_{\wt G}}$, 
\[
   \inv_v \bigl(V(A_0, A)_v\bigr)=\inv_v(-W_v) 
   \quad \text{for all}\quad
	v\neq v_0 ,
\]
where $v$ runs through all places of $F_0$, including the archimedean ones. Here, as in the proof of Proposition \ref{moduliforshim}, $V(A_0, A)=\Hom_F (\CH_0, \CH)$, where $\CH_0=H_1(A_0,\BQ)$ and $\CH= H_1(A, \BQ)$. For the non-archimedean places not lying over $p$, this follows from the existence of the  level structure; for the places $v\in \CV_p\ssm \{v_0\}$, this follows from the sign condition \eqref{signcond} at $v$; and finally, for the archimedean places, this follows from the   fact that the signatures of $V(A_0, A)$ and $-W$ at all archimedean places are identical, cf.\ the proof of Proposition \ref{moduliforshim}. 
\end{proof}

We analogously define the DM stacks $\CM_{K_{\wt H}}(\wt H)$ and $\CM_{K_{\wtHG}}(\wtHG)$ over $\Spec O_{E,(\nu)}$, comp.\ the end of Section \ref{moduli problem over E}.  Both are again smooth over $\Spec O_{E,(\nu)}$.

Let us now assume that the special vector $u \in W$ has norm $(u, u)\in O_{F_0,(p)}^\times$ . Then we obtain a finite unramified morphism, resp.\ a closed embedding, in analogy with \eqref{modembHG}, resp.\ \eqref{modembHHG}, 
\begin{equation}\label{embeddings semi-global hyperspecial}
   \CM_{K_{\wt H}}(\wt H) \to \CM_{K_{\wt G}}(\wt G)
	\quad\text{and}\quad
	\CM_{K_{\wt H}}(\wt H) \inj \CM_{K_{\wtHG}}(\wtHG).
\end{equation}
For this we assume that $K_H^p \subset H(\BA_{F_0,f}^p) \cap K_G^p$.  Furthermore, we assume for each $v\in \CV_p$, the lattices in $W_v$ and $W_v^\flat$ satisfy the relation
\begin{equation}\label{Lambda_v Lambda_v^flat relation}
   \Lambda_v = \Lambda_v^\flat \oplus O_{F,v}u.
\end{equation}
For the lattice in $W_v^\flat \oplus W_v$, we take the direct sum $\Lambda_v^\flat \oplus \Lambda_v$.

We end this subsection by defining Hecke correspondences attached to adelic elements prime to $p$.  We first consider the case of $\CM_{K_{\wt G}}(\wt G)$.  Fix $g \in \wt G(\BA_f^p)$.  Let $K_{\wt G}^p := Z^\BQ(\wh\BZ^p) \times K_G^p$ and $K_{\wt G,p} := Z^\BQ(\BZ_p) \times K_{G,p}$, and set
\[
   K_{\wt G}^{\prime p} := K_{\wt G}^p \cap g K_{\wt G}^p g\i
	\quad\text{and}\quad
	K_{\wt G}' := K_{\wt G}^{\prime p} \times K_{\wt G,p}.
\]
Then we obtain in the standard way a diagram of finite \'etale morphisms,
\begin{equation}\label{hecke unram semi-global G}
\begin{gathered}
   \xymatrix{
	     & \CM_{K_{\wt G}'}(\wt G) \ar[dl]_-{\text{nat}_1} \ar[dr]^-{\text{nat}_g}\\
	   \CM_{K_{\wt G}}(\wt G)  & &  \CM_{K_{\wt G}}(\wt G),
	}
\end{gathered}
\end{equation}
which we view as a correspondence from $\CM_{K_{\wt G}}(\wt G)$ to itself.  Note that for a central element  $g=z\in Z(G)(\BA_{F_0,f}^p)=\{z\in (\BA^p_{F,f})^\times\mid \Nm_{F/F_0}(z)=1\}$,  the diagram \eqref{hecke unram semi-global G} collapses to a map 
\begin{equation}\label{act center}
  \CM_{K_{\wt G}}(\wt G)  \xra{z}  \CM_{K_{\wt G}}(\wt G),
\end{equation}
and this induces an action of $Z(G)(\BA_{F_0,f}^p)$ on $\CM_{K_{\wt G}}(\wt G)$. 

The cases of $\CM_{K_{\wt H}}(\wt H)$ and $\CM_{K_{\wtHG}}(\wtHG)$ are completely analogous, simply replacing $\wt G$ everywhere by $\wt H$ and $\wtHG$, respectively.  For later use, we record the diagram of finite \'etale morphisms we obtain for $\wtHG$:
\begin{equation}\label{hecke unram semi-global HG}
\begin{gathered}
   \xymatrix{
	     & \CM_{K_{\wtHG}'}(\wtHG) \ar[dl]_-{\text{nat}_1} \ar[dr]^-{\text{nat}_g}\\
	   \CM_{K_{\wtHG}}(\wtHG)  & &  \CM_{K_{\wtHG}}(\wtHG).
	}
\end{gathered}
\end{equation}

\subsection{Split level at $v_0$}\label{split level}
We continue with the setup and assumptions of the previous subsection, except we now allow $v_0$ to be ramified over $p$.  In addition, we assume that $v_0$ splits in $F$, say into $w_0$ and another place $\ov w_0$.  We are again going to define smooth integral models over $O_{E,(\nu)}$.

We define the moduli functor $\CM_{K_{\wt G}}(\wt G)$ as follows.  To each locally noetherian $O_{E,(\nu)}$-scheme $S$, we associate the groupoid of tuples $(A_0,\iota_0,\lambda_0,A,\iota,\lambda,\ov\eta^p)$ as in the previous subsection, except that we impose the further condition corresponding to $w_0$ that when $p$ is locally nilpotent on $S$, the $p$-divisible group $A[w_0^\infty]$ is a Lubin--Tate group of type $r|_{w_0}$, in the sense of \cite[\S 8]{RZ14} (note that this involves the \emph{Eisenstein condition} of loc.~cit.). Here $r|_{w_0}$ is the restriction of the function $r$ on $\Hom_\BQ(F, \ov\BQ)$ to $\Hom_{\BQ_p}(F_{w_0}, \ov \BQ_p) $, in the sense of the identification \eqref{idlochom}.  We note that if $v_0$ is unramified over $p$, then this further Eisenstein condition is redundant, and the moduli functor $\CM_{K_{\wt G}}(\wt G)$ is the same as the one defined in the previous subsection, cf.\ \cite{RZ14}.

\begin{theorem}\label{semi-global split smooth}
The moduli problem just formulated is representable by a Deligne--Mumford stack $\CM_{K_{\wt G}}(\wt G)$ smooth over $\Spec O_{E,(\nu)}$. For $K_G^p$ small enough,  $\CM_{K_{\wt G}}(\wt G)$ is relatively representable over $\CM^{\fka,\xi}_0$.   Furthermore, the generic fiber $\CM_{K_{\wt G}} \times_{\Spec O_{E,(\nu)}} \Spec E$ is canonically isomorphic to $M_{K_{\wt G}}(\wt G)$. 
\end{theorem}

\begin{proof}
Same as the proof of Theorem \ref{semi-global hyperspecial smooth}, using in addition that the factor at $w_0$ for the local model for $A$ for the newly introduced Eisenstein condition is smooth, cf.\ \cite[\S8]{RZ14}.
\end{proof}

We analogously define the DM stacks $\CM_{K_{\wt H}}(\wt H)$ and $\CM_{K_{\wtHG}}(\wtHG)$ over $\Spec O_{E,(\nu)}$, comp.\ the end of Section \ref{moduli problem over E}.  Both are again smooth over $\Spec O_{E,(\nu)}$.  We then obtain a finite unramified morphism, resp.\ a closed embedding,
\begin{equation}\label{embeddings semi-global split}
   \CM_{K_{\wt H}}(\wt H) \to \CM_{K_{\wt G}}(\wt G)
	\quad\text{and}\quad
	\CM_{K_{\wt H}}(\wt H) \inj \CM_{K_{\wtHG}}(\wtHG),
\end{equation}
and Hecke correspondences exactly as in the previous subsection.

\subsection{Drinfeld level at $v_0$}\label{semi-global drinfeld}
We continue with the setup and assumptions of Section \ref{split level}, with $v_0$ split in $F$ and possibly ramified over $p$. In this subsection we are going to define integral models over $O_{E,(\nu)}$ where we impose a \emph{Drinfeld level structure} at $v_0$. To do this, we require that the \emph{matching condition} between the CM type $\Phi$ and the chosen place $\nu$ of $E$ is satisfied, which demands  that 
\begin{equation}\label{cond CM}
   \bigl\{\,\varphi\in \Hom(F, \ov\BQ) \bigm| w_\varphi=w_0\,\bigr\}\subset \Phi, 
\end{equation}
where $w_\varphi$ is the place of $F$ induced by $\wt \nu \circ \varphi$, as in (\ref{idlochom}).
We note that this condition only depends on the place $\nu$ of $E$ induced by $\wt\nu$. When condition \eqref{cond CM} is satisfied, we also say that  the CM type $\Phi$ and the place $\nu$ of $E$ are  \emph{matched}. Here are some examples in which the matching condition is guaranteed to hold.

\begin{lemma}\label{CM matching lem} The matching condition for $\Phi$ and $\nu$ is satisfied in each of the following two situations.
\begin{altenumerate}
\item\label{HT match} $F$ is of the form $K F_0$ for an imaginary quadratic field $K/\BQ$, $\Phi$ is the unique CM type induced from $K$ containing $\varphi_0$, and $p$ splits in $K$.
\item\label{deg 1 match} The place $v_0$ is of degree $1$ over $\BQ$.
\end{altenumerate}
\end{lemma}

\begin{proof} The matching condition is obvious in (\ref{HT match}), and in \eqref{deg 1 match} it holds because the left-hand side of \eqref{cond CM} is the singleton set $\{\varphi_0\}$.
\end{proof}

\begin{remark}\label{rem HT CM field}
We call the case \eqref{HT match} the \emph{Harris--Taylor case}, cf.\ \cite{HT}.
\end{remark}

Now let $m$ be a nonnegative integer.  We define the level subgroup $K_G^m \subset G(\BA_{F_0,f})$ in exactly the same way as $K_G$ in Section \ref{subsec hyper} (subject to the choice of certain vertex lattices $\Lambda_v$ for $v \in \CV_p$), except in the $v_0$-factor where $G(F_{0, v_0})\simeq\GL_n(F_{0, v_0})$, we take $K^m_{G, v_0}$ to be the principal congruence subgroup modulo $\fkp_{v_0}^m$ inside $K_{G,v_0}$. 
In particular, $K_G$ coincides with $K^m_G$ for $m=0$.  We define $K_{\wt G}^m := K_{Z^\BQ} \times K_G^m$ as in \eqref{K_wtG}.

We now define the moduli functor $\CM_{K^m_{\wt G}}(\wt G)$ over $\CM_{K_{\wt G}}(\wt G)$. Let $(A_0,\iota_0,\lambda_0,A,\iota,\lambda, \ov\eta^p)\in \CM_{K_{\wt G}}(\wt G)(S)$.  Consider the factors occurring in the decomposition (\ref{decomp by w}) of the $p$-divisible group $A[p^\infty]$,
\begin{equation}\label{dec pdiv}
   A[v_0^\infty] = A[w_0^\infty] \times A[\ov w_0^\infty] .
\end{equation}
When $p$ is locally nilpotent on $S$, the $p$-divisible group $A[w_0^\infty]$ satisfies the Kottwitz condition of type $r|_{w_0}$ for the action of $O_{F,w_0}$ on its Lie algebra, in the sense of the previous subsection. Thus by the formulation of the Kottwitz condition in Remark \ref{M remarks}\eqref{Kottwitzr} and the matching condition \eqref{cond CM},
\[
   \charac\bigl(\iota(a)\mid\Lie A[w_0^\infty]\bigr) 
	   = \prod_{\substack{\varphi \colon \! F \rightarrow \BQ\\ w_\varphi=w_0}} \bigl(T-\varphi(a)\bigr)^{r_\varphi} 
      = T - \varphi_0(a)
		\quad\text{for all}\quad
		a \in O_{F,w_0}.
\]
Hence $A[w_0^\infty]$ is a one-dimensional formal $O_{F,w_0}$-module.  Similarly, $A[\ov w_0^\infty]$ has dimension $n \cdot [F_{w_0}:\BQ_p]-1$. Since $\Ros_\lambda$ induces the conjugation automorphism on $O_F$, $\lambda$ furthermore identifies $A[w_0^\infty]$ and $A[\ov w_0^\infty]$ with the (absolute) duals of each other, and hence both have (absolute) height $n \cdot [F_{w_0}:\BQ_p]$.  Analogously,
\[
   A_0[v_0^\infty] = A_0[w_0^\infty] \times A_0[\ov w_0^\infty] ,
\]
where, by \eqref{cond CM} and the Kottwitz condition on $A_0$, the $p$-divisible group $A_0[w_0^\infty]$ with $O_{F,w_0}$-action is \'etale of height $[F_{w_0}:\BQ_p]$, whereas $A_0[\ov w_0^\infty]$ is identified with the dual of $A_0[w_0^\infty]$.  

In analogy with the prime-to-$p$ theory, we introduce the finite flat group scheme over $S$,
\[
   T_{w_0}(A_0,A)[w_0^m] := \uHom_{O_{F,w_0}}(A_0[w_0^m],A[w_0^m]).
\]
Note that as $m$ varies, the right-hand side is naturally an inverse system under restriction of homomorphisms, and to make it into a directed system depends on the choice of uniformizer $\pi_{w_0}$ of $F_{0, w_0}$.  The colimit $T_{w_0}(A_0,A) := \varinjlim_{m} T_{w_0}(A_0,A)[w_0^m]$ is a $1$-dimensional formal $O_{F,w_0}$-module since $A[w_0^\infty]$ is.

For the moduli problem $\CM_{K^m_{\wt G}}(\wt G)$, we equip the object $(A_0,\iota_0,\lambda_0,A,\iota,\lambda, \ov\eta^p)\in\CM_{K_{\wt G}}(\wt G)$ with the following  additional datum.  Let $\Lambda_{v_0} = \Lambda_{w_0} \oplus \Lambda_{\ov w_0}$ denote the natural decomposition of the lattice $\Lambda_{v_0}$ attached to the split place $v_0$.  The additional datum is
\begin{altitemize}
\item an $O_{F,w_0}$-linear homomorphism of finite flat group schemes, 
\begin{equation}\label{Drinstruc}
   \varphi\colon \pi_{w_0}^{-m}\Lambda_{w_0}/\Lambda_{w_0} \to \uHom_{O_{F,w_0}}(A_0[w_0^m],A[w_0^m]),
\end{equation}
which is a Drinfeld $w_0^m$-structure on the target, cf.\ \cite[\S II.2]{HT}.
\end{altitemize}

\begin{theorem}\label{semi-global drinfeld regular}
The moduli problem $\CM_{K^m_{\wt G}}(\wt G)$  is relatively representable by a finite flat morphism to  $\CM_{K_{\wt G}}(\wt G)$. It is regular and flat over $\Spec O_{E,(\nu)}$.    Furthermore, the generic fiber $\CM_{K^m_{\wt G}}(\wt G) \times_{\Spec O_{E,(\nu)}} \Spec E$ is canonically isomorphic to $M_{K^m_{\wt G}}(\wt G)$.
\end{theorem}

\begin{proof}
After an \'etale base change, the subgroup $\CA_0[w_0^m]$ of the universal abelian scheme $\CA_0$ over $\CM^{\fka,\xi}_0$ becomes constant, and then the proof of \cite[Lem.\ III.4.1(4)(5)]{HT} applies.
\end{proof}

We analogously define the DM stack $\CM_{K^m_{\wt H}}(\wt H)$ over $\Spec O_{E,(\nu)}$, and obtain for it the analog of Theorem \ref{semi-global drinfeld regular}.  We also define $\CM_{K^m_{\wtHG}}(\wtHG) := \CM_{K_{\wt H}^m}(\wt H) \times_{\CM^{\fka,\xi}_0} \CM_{K_{\wt G}^m}(\wt G)$, but we note that this stack is \emph{not} regular for $m > 0$.  We then obtain a finite unramified morphism, resp.\ a closed embedding,
\begin{equation*}
   \CM_{K^m_{\wt H}}(\wt H) \to \CM_{K^m_{\wt G}}(\wt G)
	\quad\text{and}\quad
	\CM_{K^m_{\wt H}}(\wt H) \inj \CM_{K^m_{\wtHG}}(\wtHG),
\end{equation*}
provided that $K_H^p \subset H(\BA_{F_0,f}^p) \cap K_G^p$ and that the lattices $\Lambda_v$ and $\Lambda_v^\flat$ are as in \eqref{Lambda_v Lambda_v^flat relation}. Indeed, the Drinfeld level structure $\varphi^\flat \colon \pi_{w_0}^{-m}\Lambda_{w_0}^\flat/\Lambda_{w_0}^\flat \to \uHom_{O_{F,w_0}}(A_0[w_0^m],A^\flat[w_0^m])$ induces a Drinfeld level structure
\[
   \varphi\colon \pi_{w_0}^{-m}\Lambda_{w_0}/\Lambda_{w_0} \to \uHom_{O_{F,w_0}}\bigl(A_0[w_0^m], (A^\flat \times A_0)[w_0^m]\bigr)
\]
by taking the direct sum of $\varphi^\flat$ and the $O_{F,w_0}$-linear homomorphism
\begin{equation}\label{varphi0}
   \begin{gathered}
   \varphi_0\colon 
   \xymatrix@R=0ex{
      \pi_{w_0}^{-m} O_{F,w_0}u/ O_{F,w_0} u \ar[r]  &  \uHom_{O_{F,w_0}}(A_0[w_0^m], A_0[w_0^m])\\
		\pi_{w_0}^{-m}u \ar@{|->}[r]  &  \id,
	}
   \end{gathered}
\end{equation}
with respect to the natural decompositions in the source and target of $\varphi$.

In the present situation there are more Hecke correspondences than those defined for adelic elements prime to $p$ in the previous two subsections, cf.\ \cite[\S III.4]{HT}.  We again first treat the case for $\wt G$. With respect to the decomposition obtained from \eqref{proddec},
\[
   \wt G(\BQ_p) = Z^\BQ(\BQ_p) \times \prod_{v \in \CV_p} G(F_{0,v}),
\]
let $g \in G(F_{0,v_0})$, considered as an element in the left-hand side.  In the special case that $g \in K_{G,v_0}$, since $K_{G,v_0}^m$ is a normal subgroup of $K_G$, we get a diagram of finite flat morphisms
\begin{equation*}
\begin{gathered}
   \xymatrix{
	     & \CM_{K_{\wt G}^m}(\wt G) \ar[dl]_-{\text{nat}_1} \ar[dr]^-{\text{nat}_g}\\
	   \CM_{K_{\wt G}^m}(\wt G)  & &  \CM_{K_{\wt G}^m}(\wt G),
	}
\end{gathered}
\end{equation*}
analogously to \eqref{hecke unram semi-global G}.  For general $g \in G(F_{0,v_0})$, choose $\mu$ large enough that
\[
   p^\mu\Lambda_{v_0} \subset g\Lambda_{v_0} \subset \varpi_{v_0}^{-m} g \Lambda_{v_0} \subset p^{-\mu} \Lambda_{v_0}.
\]
Then $K_{G,v_0}^{2\mu e} \subset K_{G,v_0}^m \cap gK_{G,v_0}^m g\i$, where $e = e_{v_0}$ is the ramification index of $v_0$ over $p$. 
Hence we obtain a diagram of finite flat morphisms
\begin{equation}\label{hecke drinfeld semi-global wtG}
\begin{gathered}
   \xymatrix{
	     & \CM_{K_{\wt G}^{2\mu e}}(\wt G) \ar[dl]_-{\text{nat}_1} \ar[dr]^-{\text{nat}_g}\\
	   \CM_{K_{\wt G}^m}(\wt G)  & &  \CM_{K_{\wt G}^m}(\wt G)
	}
\end{gathered}
\end{equation}
as before. In terms of the moduli descriptions above, these morphisms are given as follows.  Consider a point  $(A_0,\iota_0,\lambda_0,A,\iota,\lambda, \ov\eta^p,\varphi)$ of  $\CM_{K_{\wt G}^{2\mu e}}(\wt G)$.  Then $\operatorname{nat}_1$ sends this point to the point represented by $(A_0,\iota_0,\lambda_0,A,\iota,\lambda, \ov\eta^p,\varphi|_{(\pi_{w_0}^{-m}\Lambda_{w_0}/\Lambda_{w_0})})$.  To describe $\operatorname{nat}_g$, let $C_{w_0} \subset A[w_0^{2\mu e}]$ be the unique closed subgroup scheme for which the set of $\varphi(x)$, $x \in p^{-\mu} g\Lambda_{w_0}/\Lambda_{w_0} \subset p^{-2\mu}\Lambda_{w_0}/\Lambda_{w_0}$, is a complete set of sections, cf.\ \cite[Cor.\ 1.10.3]{KM} (note that $\uHom_{O_{F,w_0}}(A_0[w_0^{2\mu e}],A[w_0^{2\mu e}])$ is \'etale-locally isomorphic to $A[w_0^{2\mu e}]$).  Since $\lambda$ induces a principal polarization on the $p$-divisible group $A[v_0^\infty]$, the $\lambda$-Weil pairing on $A[p^{2\mu}]$ induces a perfect duality between $A[w_0^{2\mu e}]$ and $A[\ov w_0^{2\mu e}]$; in this way, let $C_{\ov w_0} \subset A[\ov w_0^{2\mu e}]$ be the annihilator of $C_{w_0}$.  Let $C$ be the (totally isotropic) subgroup of $A[p^{2\mu}]$,
\[
   C := C_{w_0} \times C_{\ov w_0} \times \prod_{v \in \CV_p \ssm \{v_0\}} A[v^{\mu e_v}],
\]
where $e_v$ denotes the ramification index of $v$ over $p$.  Let
\[
   A' := A/C.
\]
Let $\iota'$ denote the induced $O_F \otimes \BZ_{(p)}$-action on $A'$, and let $\alpha\colon A \to A'$ be the corresponding $O_F$-linear isogeny.  Then there is a unique polarization $\lambda'$ on $A'$ such that the diagram
\[
   \xymatrix{
	   A \ar[r]^-{p^{2\mu}\lambda} \ar[d]_-{\alpha}  &  A^\vee\\
	   A' \ar[r]^-{\lambda'}  &  (A')^\vee \ar[u]_-{\alpha^\vee}
	}
\]
commutes, and $\lambda'$ is a polarization of the type in the moduli problem for $\CM_{K_{\wt G}}(\wt G)$.  Furthermore, the level structure $\ov\eta^p$ induces, via the quasi-isogeny $p^{-\mu}\alpha$, a level structure $\ov\eta^{\prime p}$ for the pair $(A_0,A')$.  We obtain a Drinfeld level $w_0^m$-structure $\varphi'$ on $\uHom_{O_{F,w_0}}(A_0[w_0^m],A'[w_0^m])$ via the following commutative diagram:
\[
   \xymatrix@R-3ex{
	   \pi_{w_0}^{-m}\Lambda_{w_0}/\Lambda_{w_0} \ar[dd]_-{p^{-\mu}g}^-\sim \ar[rdd]^-{\varphi'}\\
		\\
		\pi_{w_0}^{-m}p^{-\mu} g\Lambda_{w_0}/p^{-\mu}g\Lambda_{w_0} \ar@{}[d]|-*{\cap} \ar@{-->}[r]  &  \uHom(A_0[w_0^m],A'[w_0^m])  \ar@{}[d]|-*{\cap} \\
		p^{-2\mu}\Lambda_{w_0}/p^{-\mu}g\Lambda_{w_0} \ar@{-->}[r]  &  \uHom(A_0[w_0^{2\mu e}],A[w_0^{2\mu e}]/C_{w_0})\\
		\\
		p^{-2\mu}\Lambda_{w_0}/\Lambda_{w_0} \ar@{->>}[uu] \ar[r]^-{\varphi}  &  \uHom(A_0[w_0^{2\mu e}],A[w_0^{2\mu e}]). \ar@{->>}[uu]
	}
\]
Here the dashed arrows are induced by $\varphi$.  Then the image of $(A_0,\iota_0,\lambda_0,A,\iota,\lambda, \ov\eta^p,\varphi)$ under $\operatorname{nat}_g$ is $(A_0,\iota_0,\lambda_0,A',\iota',\lambda',\ov\eta^{\prime p}, \varphi')$.  We remark that the construction of this latter tuple is, up to canonical isomorphism in the moduli problem for $\CM_{K_{\wt G}^m}(\wt G)$, independent of the choice of $\mu$.

Again, as in \eqref{act center}, this defines an action of the center of $G(F_{0, v_0})$ on $\CM_{K_{\wt G}^m}(\wt G)$.

The above construction carries over in the obvious way with $\wt H$ in place of $\wt G$.  Taking the product over $\CM^{\fka,\xi}_0$ of the diagram (\ref{hecke drinfeld semi-global wtG}) with the one attached to $\wt H$ and an element $h \in H(F_{0,v_0})$, we obtain, for $\mu$ sufficiently large, an analogous  diagram of finite flat morphisms for $\wtHG$. We record this as the following diagram, where $g\in (H\times G)(F_{0,v_0})$:
\begin{equation}\label{hecke drinfeld semi-global}
\begin{gathered}
   \xymatrix{
	     & \CM_{K_{\wtHG}^\mu}(\wtHG) \ar[dl]_-{\text{nat}_1} \ar[dr]^-{\text{nat}_g }\\
	   \CM_{K_{\wtHG}^m}(\wtHG)  & &  \CM_{K_{\wtHG}^m}(\wtHG).
	}
\end{gathered}
\end{equation}

\subsection{AT parahoric level at $v_0$}\label{subsec AT}
In this subsection, we assume that $p \neq 2$ and that the place $v_0$ is \emph{unramified} over $p$.  We again choose a vertex lattice $\Lambda_v \subset W_v$ for each $v \in \CV_p$, as in Section \ref{subsec hyper}. We require that the pair $(v_0,\Lambda_{v_0})$ is of one of the following four types, which we call \emph{AT types}.
\begin{altenumerate}
\renewcommand{\theenumi}{\arabic{enumi}}
\item\label{almost self-dual type} $v_0$ is inert in $F$ and $\Lambda_{v_0}$ is almost self-dual as an $O_{F,v_0}$-lattice.
\item\label{pi-mod type} $v_0$ ramifies in $F$, $n$ is even, and $\Lambda_{v_0}$ is $\pi_{v_0}$-modular.
\item\label{almost pi-mod type} $v_0$ ramifies in $F$, $n$ is odd, and $\Lambda_{v_0}$ is almost $\pi_{v_0}$-modular.
\item\label{n=2 self-dual type} $v_0$ ramifies in $F$, $n = 2$, and $\Lambda_{v_0}$ is self-dual.
\end{altenumerate}
We refer to the end of the Introduction for the terminology on lattice types.

Again recalling the decomposition $K_{\wt G} = K_{Z^\BQ} \times K_G$ from \eqref{K_wtG}, we take the subgroup $K_G$ to be of the form
\[
   K_G = K_G^p \times K_{G,p},
\]
where $K_G^p \subset G(\BA_{F_0,f}^p)$ is arbitrary, and where $K_{G,p} \subset G(F_0\otimes \BQ_p)$ is given by
\[
   K_{G,p}=\prod_{v \in \CV_p} K_{G, v} \subset \prod_{v \in \CV_p} G(F_{0,v}),
\]
with $K_{G, v}$ the stabilizer of $\Lambda_v$ in $G(F_{0,v})$.

We define the moduli functor $\CM_{K_{\wt G}}(\wt G)$ over $\Spec O_{E,(\nu)}$ similarly to the way in Section \ref{subsec hyper}. Precisely, to each $O_{E,(\nu)}$-scheme $S$, we associate the groupoid of tuples $(A_0,\iota_0,\lambda_0,A,\iota,\lambda,\ov\eta^p)$, where $(A_0,\iota_0,\lambda_0)$ is an object of $\CM^{\fka,\xi}_0(S)$, where $A$ is an abelian scheme over $S$ up to isogeny prime to $p$, where $\iota$ is an $O_{F}\otimes\BZ_{(p)}$-action on $A$ satisfying the Kottwitz condition \eqref{kottwitzF} of signature $((1, n-1)_{\varphi_0}, (0, n)_{\varphi\in\Phi\ssm\{\varphi_0\}})$, and  where $\lambda$ is a polarization on $A$ whose Rosati involution induces on $O_{F}\otimes\BZ_{(p)}$ the non-trivial Galois automorphism of $F/F_0$, subject to the  condition  that under the decomposition \eqref{decofpdivgp} of the $p$-divisible group $A[p^\infty]$,  $\ker\lambda_v$ is contained in $A[\iota(\pi_v)]$ of rank $\#(\Lambda_v^*/\Lambda_v)$ for all $v \in \CV_p$. Finally, 
$\ov\eta^p$  is a $K_G^p$-equivalence class of $\BA_{F,f}^p$-linear isometries \eqref{levelprimetop}.

We  also impose for each $v\neq v_0$ over $p$ the sign condition \eqref{signcond} at $v$. To ensure flatness of $\CM_{K_{\wt G}}(\wt G)$, we impose the following additional conditions. For each  $v\neq v_0$ over $p$ we impose the Eisenstein condition \eqref{Eisatv} at $v$. In addition, when the pair $(v_0,\Lambda_{v_0})$ is of AT type \eqref{pi-mod type}, \eqref{almost pi-mod type}, or \eqref{n=2 self-dual type}, we impose the following conditions.\footnote{In particular, no further conditions are required in AT type \eqref{almost self-dual type}; this was already observed by Cho \cite{Cho}, who based himself on a preprint version of the present paper in which we required $v_0$ to be of degree one over $p$ in all AT types.}  As in the case of the Eisenstein condition, it suffices to impose the conditions when the base scheme $S$ lies over $\Spf \ov \BZ_p$, so that, as in \eqref{Lie A big decomp}, there is a decomposition
\begin{equation}\label{onemore}
   \Lie A[v_0^\infty] \cong \bigoplus_{\psi \in \Hom_{\BQ_p}(F_{0,v_0}, \ov\BQ_p)} \Lie_\psi A[v_0^\infty];
\end{equation}
here, in the index set for the direct sum, we have used that $F_{v_0}^t = F_{0,v_0}$ ($v_0$ is unramified over $p$ in this section, and in types \eqref{pi-mod type}--\eqref{n=2 self-dual type} $v_0$ ramifies in $F$).  For each $\psi \in \Hom_{\BQ_p}(F_{0,v_0}, \ov\BQ_p)$, the extensions of $\psi$ to $F_{v_0}$ form a conjugate pair
\begin{equation}\label{nota}
   \varphi_\psi, \ov\varphi_\psi.
\end{equation}
Furthermore, in terms of the identification \eqref{idlochom}, we have $\{r_{\varphi_\psi},r_{\ov\varphi_\psi}\} = \{0,n\}$ unless
\[
   \psi = \psi_0 := \varphi_0|_{F_{v_0}},
\]
in which case $\{r_{\varphi_{\psi_0}},r_{\ov\varphi_{\psi_0}}\} = \{r_{\varphi_0},r_{\ov\varphi_0}\} = \{1,n-1\}$. For all $\psi \in \Hom_{\BQ_p}(F_{0,v_0}, \ov\BQ_p) \ssm \{\psi_0\}$, we impose the Eisenstein condition \eqref{Eisatv} on the summand $\Lie_\psi A[v_0^\infty]$.  This completes the list of conditions in type \eqref{n=2 self-dual type}.  In types \eqref{pi-mod type} and \eqref{almost pi-mod type}, we impose the following further conditions on $\Lie_{\psi_0} A[v_0^\infty]$.

\begin{altitemize}
\item If $(v_0,\Lambda_{v_0})$ is of type \eqref{pi-mod type}, then we impose the \emph{wedge condition}
\begin{equation}\label{wedge condition}
	\bigwedge\nolimits^2 \bigl(\iota(\pi_{v_0})+\pi_{v_0} \mid \Lie_{ \psi_0} A[v_0^\infty]\bigr) = 0
\end{equation}
and the \emph{spin condition}
\begin{equation}\label{spin condition}
   \text{\emph{the endomorphism $\iota(\pi_{v_0})\mid \Lie_{ \psi_0} A[v_0^\infty]$ is nonvanishing at each point of $S$,}}
\end{equation}
cf.\ \cite[\S6]{RSZ2}.
\item If $(v_0,\Lambda_{v_0})$ is of type \eqref{almost pi-mod type}, then we impose the \emph{refined spin condition} (7.9) of \cite{RSZ2} on $\Lie_{ \psi_0} A[v_0^\infty]$. More precisely, loc.\ cit.\ applies as written (taking our $F_{v_0}/F_{0,v_0}$ as the local extension $F/F_0$ there) in the case that $v_0$ has degree one over $p$, so that $\Lie_{\psi_0} A[v_0^\infty] = \Lie A[v_0^\infty]$.  For general $v_0$ unramified over $p$, analogously to \eqref{onemore}, the $\CO_S$-module $M(A[v_0^\infty])$ (notation of loc.\ cit.)\ decomposes under the action of $O_{F_0,v_0}$ into a direct sum $\bigoplus_{\psi} M_\psi(A[v_0^\infty])$, and for each $\psi$ there is an exact sequence
\[
   0 \to \Fil_\psi^1 \to M_\psi(A[v_0^\infty]) \to \Lie_\psi A[v_0^\infty] \to 0.
\]
Now taking $\psi = \psi_0$, the condition we impose is given by (7.9) in loc.\ cit.,\ except with $\Fil_{\psi_0}^1$ in place of $\Fil^1$, with the almost $\pi_{v_0}$-modular lattice $\Lambda_{v_0}$ in place of $\Lambda_{-m}$, and with the tensor products $- \otimes_{O_{F,v_0}} \CO_S$ and $- \otimes_{O_{F_0,v_0}} \CO_S$ in terms of the structure morphism
\[
   O_{F_0,v_0} \subset O_{F,v_0} \xra{\varphi_0} \ov\BZ_p \to \CO_S.
\]
(The same applies to the definition of $L_{-m,-1}^{n-1,1}(S)$ in (7.6) in loc.\ cit.)
\end{altitemize}

\begin{remark}\label{rel deg 1 remark}
When $v_0$ is of degree one over $p$, the triple $(A[v_0^\infty], \iota[v_0^\infty], \lambda[v_0^\infty])$ arising from our moduli problem is of the type occurring in the moduli problem for one of the RZ spaces in \cite[\S5--8]{RSZ2} (because in the degree one case, the relative dual of $A[v_0^\infty]$ which is used in \cite{RSZ2} is the Serre dual). 

\end{remark}

\begin{theorem}\label{semi-global AT smooth/regular}
The moduli problem just formulated is representable by a Deligne--Mumford stack $\CM_{K_{\wt G}}(\wt G)$ flat over $\Spec O_{E,(\nu)}$.  For $K_G^p$ small enough,  $\CM_{K_{\wt G}}(\wt G)$ is relatively representable over $\CM^{\fka,\xi}_0$. The generic fiber $\CM_{K_{\wt G}}(\wt G) \times_{\Spec O_{E,(\nu)}} \Spec E$ is canonically isomorphic to $M_{K_{\wt G}}(\wt G)$.  Furthermore:
\begin{altenumerate}
\item\label{semi-global AT smooth/regular i} $\CM_{K_{\wt G}}(\wt G)$ is smooth over $\Spec O_{E,(\nu)}$ provided that $(v_0,\Lambda_{v_0})$ is of AT type \eqref{pi-mod type} or \eqref{almost pi-mod type}.
\item\label{semi-global AT smooth/regular ii} $\CM_{K_{\wt G}}(\wt G)$ has semi-stable reduction over $\Spec O_{E,(\nu)}$ provided that $(v_0,\Lambda_{v_0})$ is of AT type \eqref{almost self-dual type} and $E_\nu$ is unramified over $\BQ_p$.
\end{altenumerate}
\end{theorem}

\begin{proof} 
The representability assertion and the assertion for the generic fiber are proved in the same way as in Theorem \ref{semi-global hyperspecial smooth}.  The assertions concerning the local structure of $\CM_{K_{\wt G}}(\wt G)$ all reduce to statements about the local model.  As in the proof of Theorem \ref{semi-global hyperspecial smooth}, the local model is a product of local models, with one factor for $A_0$ and one for $A$, and the factor for $A_0$ is trivial.  The factor for $A$ furthermore decomposes as in \eqref{LM prod decomp} into a product of local models $M_v \times_{\Spec O_{E_{r|_v}}} \Spec O_{E_\nu}$ indexed by the places $v \in \CV_p$, and the factor for each $v \neq v_0$ is again trivial. At the place $v_0$, let $E_{r|_{v_0}}^\un$ denote the maximal unramified extension of $E_{r|_{v_0}}$ in $\ov\BQ_p$, and let $O_{E_{r|_{v_0}}^\un}$ denote its ring of integers.  After extending scalars $O_{E_{r|_{v_0}}} \to O_{E_{r|_{v_0}}^\un}$, the canonical isomorphism $O_{F_0,v_0} \otimes_{\BZ_p} O_{E_{r|_{v_0}}^\un} \cong \prod_{\psi \in \Hom_{\BQ_p}(F_{0,v_0},\ov\BQ_p)} O_{E_{r|_{v_0}}^\un}$ (recall that $v_0$ is unramified over $p$) induces a further decomposition of the local model
\[
   M_{v_0} \times_{\Spec O_{E_{r|_{v_0}}}} \Spec O_{E_{r|_{v_0}}^\un} = \prod_{\psi \in \Hom_{\BQ_p}(F_{0,v_0},\ov\BQ_p)} M_{v_0,\psi} \times_{\Spec O_{E_{r|_{v_0,\psi}}}} \Spec O_{E_{r|_{v_0}}^\un}.
\]
Here, in terms of the identifications \eqref{Hom(F,Q) id} and \eqref{idlochom} and the notation \eqref{nota}, $r|_{v_0,\psi}$ denotes the restriction of $r$ to the set
\[
   \bigl\{\, \varphi \in \Hom_{\BQ_p}(F_{v_0}, \ov\BQ_p) \bigm| \varphi|_{F_{0,v_0}} = \psi \,\bigr\} = \{\varphi_\psi,\ov\varphi_\psi\},
\]
and $M_{v_0,\psi} = M(F_v/F_{0,v},r|_{v_0,\psi},\Lambda_v)$ is a local model attached to the algebraic $F_{0,v_0}$-group $\GU(W_{v_0})$, the signature function $r|_{v_0,\psi}$, and the $O_{F,v_0}$-lattice $\Lambda_{v_0}$.  For $\psi \neq \psi_0$, the signature function  $r|_{v_0,\psi}$ is banal, and $M_{v_0,\psi}$ is again trivial.  For $\psi = \psi_0$, the local model $M_{v_0,\psi_0}$ is smooth over $\Spec O_{E_{r|_{v_0,\psi_0}}}$ in types \eqref{pi-mod type} \cite[Prop.\ 3.10]{RSZ1} and \eqref{almost pi-mod type} \cite[Th.\ 1.4]{S}, and is flat of semi-stable reduction $\Spec O_{E_{r|_{v_0,\psi_0}}}$ in types \eqref{almost self-dual type} \cite[pf.\ of Th.\ 5.1]{RSZ2} and \eqref{n=2 self-dual type} \cite[\S8, pp.\ 1119--20]{RSZ2}.  This completes the proof, noting in assertion \eqref{semi-global AT smooth/regular ii} that semi-stable reduction is preserved under an unramified base extension.
\end{proof}

\begin{remark}\label{rem ram cond satisfied HT}
The unramifiedness condition on $E_\nu$ in part \eqref{semi-global AT smooth/regular ii} of Theorem \ref{semi-global AT smooth/regular} is always satisfied when $F$ contains an imaginary quadratic field $K$ and $\Phi$ is induced from $K$, since $\varphi_0\colon F \isoarrow E$ in this case (cf.\ Remark \ref{reflex rem}) and $F_{v_0}$ is unramified over $\BQ_p$ in AT type \eqref{almost self-dual type}.
\end{remark}

\begin{remark}\label{excl type4}
In the case of AT type \eqref{n=2 self-dual type}, in the notation of the proof of Theorem \ref{semi-global AT smooth/regular}, we have $E_{r|_{v_0,\psi_0}} = F_{0,v_0}$, and the local model $M_{v_0,\psi_0}$ has semi-stable reduction  over $\Spec O_{F_0,v_0}$, comp.\ \cite{RSZ2}. However, the extension $E_\nu/F_{0,v_0}$ is always ramified (because $F_{v_0}$ maps into $E_\nu$), and therefore $\CM_{K_{\wt G}}(\wt G)$ is not regular over the place $\nu$ in this case. For this reason, we will exclude AT type \eqref{n=2 self-dual type} when considering arithmetic intersections.
\end{remark}

We analogously define the DM stack $\CM_{K_{\wt H}}(\wt H)$ over $\Spec O_{E,(\nu)}$ when the pair $(v_0,\Lambda_{v_0}^\flat)$ is of AT type.
To define $\CM_{K_{\wtHG}}(\wtHG)$, we take as before the lattice $\Lambda_v^\flat \oplus \Lambda_v$ in $W_v^\flat \oplus W$ for each $v \in \CV_p$, but the relation between $\Lambda_v^\flat$ and $\Lambda_v$ that we allow can be more complicated.  Furthermore, the definition of the analog of the morphisms \eqref{embeddings semi-global hyperspecial} requires more care.

Let us first suppose that $(v_0,\Lambda_{v_0})$ is not of AT type \eqref{pi-mod type}.  In this case, we assume that the lattices $\Lambda_v$ and $\Lambda_v^\flat$ satisfy $\Lambda_v = \Lambda_v^\flat \oplus O_{F,v}u$ for all $v$. We also assume that $(u,u) \in O_{F_0,v}^\times$ for all $v$, with the single exception of $v = v_0$ when $(v_0,\Lambda_{v_0})$ is of AT type \eqref{almost self-dual type}, in which case we impose that $\Lambda_{v_0}^\flat$ is self-dual and $\ord_{v_0} (u,u) = 1$ (and of course we use the definition of $\CM_{K_{\wt H}}(\wt H)$ from Section \ref{subsec hyper}).  Provided that $K_H^p \subset H(\BA_{F_0,f}^p) \cap K_G^p$, we then obtain a finite unramified morphism, resp.\ a closed embedding,
\begin{equation}\label{embeddings semi-global AT}
   \CM_{K_{\wt H}}(\wt H) \to \CM_{K_{\wt G}}(\wt G)
	\quad\text{and}\quad
	\CM_{K_{\wt H}}(\wt H) \inj \CM_{K_{\wtHG}}(\wtHG),
\end{equation}
as in \eqref{embeddings semi-global hyperspecial}.

\begin{remark}\label{unit}
The unit $(u,u)$ was chosen very carefully in \cite{RSZ2} because in loc.\ cit.\ we made a definite choice between the two isomorphic RZ spaces $\CN_n^{(0)}$ and $\CN_n^{(1)}$ in the odd ramified case (more precisely, a definite choice of the framing object).  Here we make no such choice, and therefore $(u,u)$ can be an arbitrary unit at ramified places.
\end{remark}

Now suppose that $(v_0,\Lambda_{v_0})$ is of AT type \eqref{pi-mod type}.  Then we cannot define such simple embeddings, and it is necessary to consider more complicated diagrams involving additional spaces, cf.\ \cite[\S12]{RSZ2}.  In fact we will consider two variants.  For both variants, we assume that $(u,u) \in O_{F_0,(p)}^\times$, and that $\Lambda_{v_0}^\flat$ and $\Lambda_{v_0}$ are related by a chain of inclusions
\begin{equation}\label{type 2 lattice relation}
   \pi_{v_0} (\Lambda_{v_0}^\flat \oplus O_{F, v_0} u)^* \subset^1 \Lambda_{v_0} \subset^1 \Lambda_{v_0}^\flat \oplus O_{F, v_0} u.
\end{equation}
Note that \eqref{type 2 lattice relation} is equivalent to the condition that $\Lambda_{v_0}^\flat$ is almost $\pi_{v_0}$-modular in $W_{v_0}^\flat$, and that $\Lambda_{v_0}$ is one of the two $\pi_{v_0}$-modular lattices contained in $\Lambda_{v_0}^\flat \oplus O_{F, v_0} u$.  For simplicity, we also assume in type \eqref{pi-mod type} that $v_0$ is of degree one over $p$, as we will do later in Section \ref{s:conjaip} for all places of AT type in the context of arithmetic intersections.

\smallskip
\noindent\emph{Variant \refstepcounter{variant}\arabic{variant}\label{variant 1}:} In the first variant, we continue to assume that for all $v \neq v_0$, we have $\Lambda_v = \Lambda_v^\flat \oplus O_{F,v}u$.  We define $\CP_{\wt G}$ to be the moduli stack defined in the same way as $\CM_{K_{\wt G}}(\wt G)$, except that at the place $v_0$,
\begin{altitemize}
\item $\ker\lambda_{v_0}$ has rank $p^{n-2}$; and
\item when $p$ is locally nilpotent on the base scheme, $\Lie A[v_0^\infty]$ satisfies the condition (9.2) of \cite{RSZ2}.
\end{altitemize}

Now let $\CL$ denote the self-dual multichain of $O_F \otimes \BZ_p$-lattices in $W \otimes \BQ_p = \bigoplus_{v \in \CV_p} W_v$ generated by $\Lambda_{v_0}$ and $\Lambda_{v_0}^\flat \oplus O_{F, v_0} u$ (and its dual), and by $\Lambda_v$ for all $v \neq v_0$. We define $\CP_{\wt G}'$ to be the moduli stack of tuples $(A_0,\iota_0,\lambda_0,\bA,\boldsymbol{\lambda},\ov\eta^p)$, where $(A_0,\iota_0,\lambda_0)$ is an object of $\CM^{\fka,\xi}_0$, $\bA = \{A_\Lambda\}$ is an $\CL$-set of abelian varieties, $\boldsymbol\lambda$ is a $\BQ$-homogeneous principal polarization of $\bA$, and $\ov\eta^p$  is a $K_G^p$-equivalence class of $\BA_{F,f}^p$-linear isometries $\eta^p\colon \wh V^p(A_0,\bA) \simeq -W \otimes_F \BA_{F,f}^p$, cf.\ \cite[Def.\ 6.9]{RZ}.  We require that $A_\Lambda$ satisfies the Kottwitz condition \eqref{kottwitzF} for all $\Lambda$. We further require that over a base on which $p$ is locally nilpotent, when the $v_0$-summand of $\Lambda$ is $\Lambda_{v_0}$, $\Lie A_\Lambda[v_0^\infty]$ satisfies the wedge condition \eqref{wedge condition} and the spin condition \eqref{spin condition} above; and when the $v_0$-component of $\Lambda$ is $\Lambda_{v_0}^\flat \oplus O_{F, v_0} u$, $\Lie A_\Lambda[v_0^\infty]$ satisfies the condition (9.2) of \cite{RSZ2}.  We obtain a diagram
\begin{equation}\label{semi-global P diag}
	\begin{gathered}
   \xymatrix{
	     &  & \CP_{\wt G}' \ar[dl]_-{\pi_1} \ar[dr]^-{\pi_2}\\
	   \CM_{K_{\wt H}}(\wt H) \ar[r]  &  \CP_{\wt G}  & &  \CM_{K_{\wt G}}(\wt G).
	}
	\end{gathered}
\end{equation}
Here the lower left morphism is defined in the usual way, i.e.\ analogously to \eqref{modembHG}, provided that $K_H^p \subset H(\BA_{F_0,f}^p) \cap K_G^p$.  It is again finite and unramified.  The arrows $\pi_1$ and $\pi_2$ are induced by $\bA \mapsto A_{(\Lambda_{v_0}^\flat \oplus O_{F, v_0}u)  \oplus \bigoplus_{v \neq v_0} \Lambda_v}$ and $\bA \mapsto A_{\Lambda_{v_0} \oplus \bigoplus_{v \neq v_0} \Lambda_v}$, respectively.

\begin{lemma}\label{pi_1 fin et}
The morphism $\pi_1$ is finite \'etale of degree $2$.
\end{lemma}

\begin{proof}
The morphism is obviously proper.  It is also finite because each geometric fiber has precisely two points.  This last assertion follows over the complex numbers by looking at the homology of the abelian varieties in play, and in positive characteristic by looking at their Dieudonn\'e modules.  The question of \'etaleness reduces to the local models, which are isomorphic by \cite[Prop.\ 9.12(ii)]{RSZ2}.
\end{proof}

The morphism $\pi_2$ is proper, and it is finite \'etale over the generic fiber of degree $(p^n-1)/(p-1)$.  However, $\pi_2$ is not finite when $n \geq 4$, cf.\ \cite[Rem.\ 9.5]{RSZ2}.

Now define
\[
   \CP_{\wtHG} := \CM_{K_{\wt H}}(\wt H) \times_{\CM^{\fka,\xi}_0} \CP_{\wt G}
	\quad\text{and}\quad
	\CP_{\wtHG}' := \CM_{K_{\wt H}}(\wt H) \times_{\CM^{\fka,\xi}_0} \CP_{\wt G}'.
\]
Applying the functor $\CM_{K_{\wt H}}(\wt H) \times_{\CM^{\fka,\xi}_0} \!-$ to the rightmost three spaces in \eqref{semi-global P diag}, we obtain a diagram
\begin{equation}\label{semi-global corresp diag G variant 1}
	\begin{gathered}
   \xymatrix{
	     &  & \CP_{\wtHG}' \ar[dl] \ar[dr]\\
	   \CM_{K_{\wt H}}(\wt H) \ar@{^{(}->}[r]  &  \CP_{\wtHG}  & &  \CM_{K_{\wtHG}}(\wtHG).
	}
	\end{gathered}
\end{equation}
Here the lower left embedding is the graph of the one in \eqref{semi-global P diag}, again provided that $K_H^p \subset H(\BA_{F_0,f}^p) \cap K_G^p$. Of course, the oblique arrows inherit the properties of the corresponding ones in \eqref{semi-global P diag} under base change.  Set
\[
   \CM_{K_{\wt H}'}(\wt H) := \CM_{K_{\wt H}}(\wt H) \times_{\CP_{\wtHG}} \CP_{\wtHG}'.
\]
Note that the generic fiber of $\CM_{K_{\wt H}'}(\wt H)$ is equal to $M_{K_{\wt H}'}(\wt H)$, where $K_{\wt H}' = K_{Z^\BQ} \times K_H^p \times K_{H,p}'$, with $K_{H,v}' = K_{H,v}$ at all places $v \neq v_0$, and $K_{H,v_0}'$ the simultaneous stabilizer of $\Lambda_{v_0}^\flat$ and $\Lambda_{v_0}$.

\begin{lemma}\label{primed emb 1}
The morphism
\begin{equation*}
   \CM_{K_{\wt H}'}(\wt H) \to \CM_{K_{\wtHG}}(\wtHG)
\end{equation*}
induced by \eqref{semi-global corresp diag G variant 1} is a closed embedding.
\end{lemma}

\begin{proof}
The proof of \cite[Prop.\ 12.1]{RSZ2} applies.
\end{proof}

\smallskip
\noindent\emph{Variant \refstepcounter{variant}\arabic{variant}:} For the second variant, in addition to the place $v_0$, we allow there to be places $v_1,\dotsc,v_{m-1} \in \CV_p$ for which the lattice $\Lambda_{v_i}$ is $\pi_{v_i}$-modular.  For each $i = 0,\dotsc,m-1$, we then assume that the relation \eqref{type 2 lattice relation} holds with $v_i$ in place of $v_0$.  For all $v \neq v_0,\dotsc,v_{m-1}$, we again assume that $\Lambda_v = \Lambda_v^\flat \oplus O_{F,v} u$.  We then define the stack $\CP_{\wt G}$ exactly as above.  We also define $\CP_{\wt G}'$ exactly as above, except we now take $\CL$ to be the self-dual multichain of $O_F \otimes \BZ_p$-lattices in $W \otimes \BQ_p = \bigoplus_{v \in \CV_p} W_v$ generated by the lattices $\Lambda_{v_i}$ and $\Lambda_{v_i}^\flat \oplus O_{F_{v_i}}u$ for each $i = 0,\dotsc,m-1$, and by $\Lambda_v$ for all $v \neq v_0,\dotsc,v_{m-1}$.  (To be clear, the conditions above on the Lie algebra of the $p$-divisible group when $p$ is locally nilpotent on the base still only involve the place $v_0$.)

In complete analogy with \eqref{semi-global P diag}, there is a diagram
\begin{equation}\label{semi-global corresp diag G}
	\begin{gathered}
   \xymatrix{
	     &  & \CP_{\wt G}' \ar[dl]_-{\pi_1} \ar[dr]^-{\pi_2}\\
	   \CM_{K_{\wt H}}(\wt H) \ar[r]  &  \CP_{\wt G}  & &  \CM_{K_{\wt G}}(\wt G),
	}
	\end{gathered}
\end{equation}
where the lower left morphism is defined provided that $K_H^p \subset H(\BA_{F_0,f}^p) \cap K_G^p$, and where the arrows $\pi_1$ and $\pi_2$ are induced by
\[
   \bA \mapsto A_{\Lambda_{v_1}'\oplus \dotsb \oplus \Lambda_{v_m}' \oplus \bigoplus_{v \neq v_1,\dotsc,v_m} \Lambda_v}
   \quad\text{and}\quad
   \bA \mapsto A_{\Lambda_{v_1}\oplus \dotsb \oplus \Lambda_{v_m} \oplus \bigoplus_{v \neq v_1,\dotsc,v_m} \Lambda_v},
\]
respectively.  The proof of Lemma \ref{pi_1 fin et} transposes to yield the following.

\begin{lemma} The morphism $\pi_1$ is finite \'etale of degree $2^m$. \qed
\end{lemma}

In analogy with Variant \ref{variant 1}, the morphism $\pi_2$ is proper, and finite \'etale over the generic fiber of degree $((p^n-1)/(p-1))^m$.  However, it is again not finite when $n \geq 4$.

Finally, we again define
\[
   \CP_{\wtHG} := \CM_{K_{\wt H}}(\wt H) \times_{\CM^{\fka,\xi}_0} \CP_{\wt G}
	\quad\text{and}\quad
	\CP_{\wtHG}' := \CM_{K_{\wt H}}(\wt H) \times_{\CM^{\fka,\xi}_0} \CP_{\wt G}',
\]
and we apply the functor $\CM_{K_{\wt H}}(\wt H) \times_{\CM^{\fka,\xi}_0} \!-$ to obtain
\begin{equation}\label{semi-global corresp diag HG}
	\begin{gathered}
   \xymatrix{
	     &  & \CP_{\wtHG}' \ar[dl] \ar[dr]\\
	   \CM_{K_{\wt H}}(\wt H) \ar@{^{(}->}[r]  &  \CP_{\wtHG}  & &  \CM_{K_{\wtHG}}(\wtHG).
	}
	\end{gathered}
\end{equation}
Here for the lower left embedding we assume, as always, that $K_H^p \subset H(\BA_{F_0,f}^p) \cap K_G^p$.  Set
\[
   \CM_{K_{\wt H}'}(\wt H) := \CM_{K_{\wt H}}(\wt H) \times_{\CP_{\wtHG}} \CP_{\wtHG}'.
\]
Note that the generic fiber of $\CM_{K_{\wt H}'}(\wt H)$ is equal to $M_{K_{\wt H}'}(\wt H)$, where $K_{\wt H}' = K_{Z^\BQ} \times K_H^p \times K_{H,p}'$, with $K_{H,v}' = K_{H,v}$ at all places $v \neq v_0,\dotsc,v_{m-1}$, and $K_{H,v_i}'$ the simultaneous stabilizer of $\Lambda_{v_i}^\flat$ and $\Lambda_{v_i}$ at all places $v_i$ for $i = 0,\dotsc,m-1$.  As in the case of Lemma \ref{primed emb 1}, we obtain the following.

\begin{lemma}
The morphism
\[
   \CM_{K_{\wt H}'}(\wt H) \to \CM_{K_{\wtHG}}(\wtHG)
\]
induced by \eqref{semi-global corresp diag HG} is a closed embedding.\qed
\end{lemma}

We note that if $(v_0,\Lambda_{v_0})$ is of type \eqref{pi-mod type} or \eqref{almost pi-mod type}, then the spaces $\CM_{K_{\wt H}}(\wt H)$ and $\CM_{K_{\wtHG}}(\wtHG)$ are smooth. If $(v_0,\Lambda_{v_0})$ is of type \eqref{almost self-dual type}, then $\CM_{K_{\wt H}}(\wt H)$ is smooth, and $\CM_{K_{\wtHG}}(\wtHG)$ has semi-stable reduction provided that $E_\nu$ is unramified over $\BQ_p$.  

In these cases, one can define Hecke correspondences prime to $p$, as at the end of Section \ref{subsec hyper}. 

\begin{remark}\label{more atc}
In analogy with the above cases in which $(v_0,\Lambda_{v_0})$ is of AT type \eqref{pi-mod type} and $(v_0,\Lambda_{v_0}^\flat)$ is of AT type \eqref{almost pi-mod type}, one may also consider a situation in which $(v_0,\Lambda_{v_0})$ is hyperspecial and $(v_0,\Lambda_{v_0}^\flat)$ is of AT type \eqref{almost self-dual type}.  One obtains a closed embedding analogous to the one in Lemma \ref{primed emb 1}, where the source is a finite covering of $\CM_{K_{\wt H}}(\wt H)$.
\end{remark}

\section{Global integral models}\label{section global}
In this section, we define integral models of the above moduli spaces over $\Spec O_E$. We will take $\fka=O_{F_0}$, i.e., we will assume that $\CM_0=\CM_0^{O_{F_0}}$ is non-empty. Recall from Remark \ref{rem exa}(\ref{rem exa satisfied}) that this hypothesis is satisfied whenever $F/F_0$ is ramified at some finite place, a condition which we will eventually impose below in the context of arithmetic intersections, cf.\ Remark \ref{latticecond herm}.  We fix $\xi\in \CL_\Phi^\fka/{\sim}$ and set $\CM_0^\xi := \CM_0^{O_{F_0},\xi}$.

\subsection{Trivial level structure}\label{global mod prob trivial level}
In this subsection, we are going to define integral models over $\Spec O_E$ of the previously defined moduli spaces in the case that the open compact subgroup is the stabilizer of a lattice of a certain form.  Let us start with the case of $\wt G$. We consider the following finite set of finite places of $F_0$,
\begin{equation}\label{V_AT^W}
   \CV_\AT^W := \{\, v \mid \text{$v$ is either inert in $F$ and $W_v$ is non-split, or $v$ ramifies in $F$} \,\} .
\end{equation}
Let
\[
   \fkd_\AT^W := \prod_{v\in \CV_\AT^W} \fkq_v \subset O_F,
\]
where $\fkq_v$ denotes the (unique) prime in $O_F$ determined by $v \in \CV_\AT^W$.
We fix an $O_F$-lattice $\Lambda$ in $W$ with 
\begin{equation}\label{ATvertex}
   \Lambda\subset \Lambda^*\subset (\fkd_\AT^W)\i\Lambda .
\end{equation}
We assume that the triple $(F/F_0, W, \Lambda)$ satisfies the following conditions.
\begin{altenumerate}
\renewcommand{\theenumi}{\arabic{enumi}}
\item All finite places $v$ of $F_0$ ramified over $\BQ$ or dividing $2$ are split in $F$.\footnote{In \cite[\S6.1]{RSZ4} this assumption is relaxed to the assumption that all such $v$ are unramified in $F$, with the analog of Theorem \ref{globalnolevel} below continuing to hold true.}
\item\label{AT assumptions} All places $v\in \CV_\AT^W$ are unramified over $\BQ$, and the pair $(v, \Lambda_v)$ is isomorphic to one of the AT types \eqref{almost self-dual type}--\eqref{n=2 self-dual type} in Section \ref{subsec AT}.  
\end{altenumerate}
As a consequence, for any finite place $\nu$ of $E$, denoting by  $v_0$  the place of $F_0$ induced by $\nu$  via $\varphi_0$, the pair $(v_0,\Lambda_{v_0})$ is of the type considered in one of the four subsections Section \ref{semi-global hyperspecial smooth}--\ref{subsec AT}. Associated to these data is the open compact subgroup
\[
   K_G^\circ := \bigl\{\, g\in G(\BA_{F_0, f})\bigm| g(\Lambda\otimes_{O_F}\wh O_F)=\Lambda\otimes_{O_F}\wh O_F \,\bigr\},
\]
and as usual we define $K_{\wt G}^\circ := K_{Z^\BQ} \times K_G^\circ$.

We formulate a moduli problem over $\Spec O_{E}$ as follows.  To each $O_{E}$-scheme $S$, we associate the groupoid of tuples $(A_0,\iota_0,\lambda_0,A,\iota,\lambda)$, where $(A_0,\iota_0,\lambda_0)$ is an object of $\CM_0^\xi(S)$.  Furthermore,
\begin{altitemize}
\item $(A,\iota)$ is an abelian scheme over $S$, with $O_{F}$-action $\iota$ satisfying the Kottwitz condition \eqref{kottwitzF} of signature $((1, n-1)_{\varphi_0}, (0, n)_{\varphi\in\Phi\ssm\{\varphi_0\}})$; and
\item $\lambda$ is a  polarization whose Rosati involution induces on $O_{F}$ the non-trivial Galois automorphism of $F/F_0$. 
\end{altitemize}
We impose the sign condition that at every point $s$ of $S$,
\begin{equation}\label{signcond2}
   \inv^r_v(A_{0,s}, \iota_{0,s}, \lambda_{0,s}, A_s, \iota_s, \lambda_s) = \inv_v(-W_v) ,
\end{equation}
for every finite place $v$ of $F_0$ which is non-split in $F$. Furthermore, we impose that for any finite place $\nu$ of $E$, denoting by $p$ its residue characteristic, the triple up to isogeny prime to $p$ over $S \times_{\Spec O_E} \Spec O_{E,(\nu)}$,
\[
   (A\otimes \BZ_{(p)}, \iota \otimes \BZ_{(p)}, \lambda \otimes \BZ_{(p)}),
\]
satisfies the conditions in the semi-global moduli problem for $\nu$ defined in Section \ref{section semi-global}. 

The morphisms in this category are the isomorphisms. 

\begin{remark}
\begin{altenumerate}
\item Suppose that $n$ is even.  Then when $v$ is a finite place of $F_0$ inert (resp.,\ ramified) in $F$, the two isometry types of the $n$-dimensional $F_v/F_{0,v}$-hermitian spaces are distinguished by whether they contain a self-dual (resp.,\ $\pi_v$-modular) lattice.  This implies that when there are no places $v$ such that the pair $(v,\Lambda_v)$ above is of AT type \eqref{n=2 self-dual type}, the sign condition \eqref{signcond2} is automatically satisfied. 
\item Suppose that $F_0 = \BQ$.  Then there is no need for  the sign condition. Indeed, by the Hasse principle for hermitian forms, it is equivalent to impose the condition that for every geometric point $\ov s$ of $S$, there exists an isomorphism of hermitian $O_{F,\ell}$-lattices
\[
   \Hom_{O_F}\bigl(T_\ell(A_{0,\ov s}),T_\ell(A_{\ov s})\bigr) \simeq -\Lambda_\ell
\]
for every prime number $\ell \neq \charac \kappa(\ov s)$. Hence  we recover  in this case  the definition of the integral moduli problem of \cite[\S2.3]{BHKRY} (where the principal polarization in the moduli problem of loc.\ cit.\ is replaced by the polarization type we have specified above). 
\end{altenumerate}
\end{remark}

The only point requiring proof in the next theorem is the representability of $\CM_{K^\circ_{\wt G}}(\wt G)$, which is, however, routine.

\begin{theorem}\label{globalnolevel}
The moduli problem just formulated is representable by a Deligne--Mumford stack $\CM_{K^\circ_{\wt G}}(\wt G)$ flat over $\Spec O_{E}$. For every place $\nu$ of $E$, the base change $\CM_{K^\circ_{\wt G}} \times_{\Spec O_{E}} \Spec O_{E, (\nu)}$ is canonically isomorphic to the semi-global moduli space defined in one of Sections \ref{subsec hyper}, \ref{split level}, or \ref{subsec AT} above.  Hence:
\begin{altenumerate}
\item $\CM_{K^\circ_{\wt G}}(\wt G)$ is smooth of relative dimension $n-1$ over the open subscheme of $\Spec O_E$ obtained by removing all places $\nu$ for which the induced pair $(v_0,\Lambda_{v_0})$ is of AT type \eqref{almost self-dual type} or \eqref{n=2 self-dual type} in Section \ref{subsec AT}.
\item $\CM_{K^\circ_{\wt G}}(\wt G)$ has semi-stable reduction over the open subscheme of $\Spec O_E$ obtained by removing all places $\nu$ for which either $(v_0,\Lambda_{v_0})$ is of AT type \eqref{n=2 self-dual type}, or is of AT type \eqref{almost self-dual type} and for which  $E_\nu$ is ramified over $\BQ_p$.\qed
\end{altenumerate}   
\end{theorem}

Replacing $W$ by $W^\flat$ and choosing a lattice $\Lambda^\flat \subset W^\flat$ analogously to above, we define the DM stack $\CM_{K_{\wt H}^\circ}(\wt H)$, where $K_{\wt H}^\circ := K_{Z^\BQ} \times K_H^\circ$ with $K_H^\circ \subset H(\BA_{F_0,f})$ the stabilizer of $\Lambda^\flat \otimes_{O_F} \wh O_F$. Similarly, we define $\CM_{K_{\wtHG}^\circ}(\wtHG)$, where $K_\wtHG^\circ := K_{Z^\BQ} \times K_H^\circ \times K_G^\circ$, and where we impose the following additional conditions on $\Lambda^\flat$ and $\Lambda$. We first require that $\Lambda_v^\flat$ is self-dual for all $v$ which are split or inert in $F$ (i.e.\ we require that $\CV_\AT^{W^\flat}$ consists of exactly the finite places of $F_0$ which ramify in $F$), and that $\Lambda_v^\flat$ and $\Lambda_v$ are $\pi_v$-modular or almost $\pi_v$-modular for all $v$ which ramify in $F$ (i.e.\ $\CV_\AT^{W^\flat}$ and $\CV_\AT^W$ contain no $v$ for which $(v,\Lambda_v^\flat)$ or $(v,\Lambda_v)$ is of AT type \eqref{n=2 self-dual type}).

\begin{remark}\label{latticecond herm}
The conditions that we have just imposed on $\Lambda^\flat$ place non-trivial constraints on the extension $F/F_0$ and on the hermitian space $W^\flat$. Let $d := [F_0 : \BQ]$.

First consider the case when $n=2m+1$ is odd. Then our assumptions on $\Lambda^\flat$ imply that $W^\flat$ is split at all finite places of $F_0$. On the other hand, at each archimedean place $\varphi$, the Hasse invariant $\inv_\varphi(W_\varphi^\flat)$ is equal to $(-1)^{m-1}$ if $\varphi=\varphi_0$, and to $(-1)^m$ if $\varphi\neq\varphi_0$.   Hence the product formula \eqref{herm prod fmla} imposes the congruence
\begin{equation}\label{congruence}
   dm\equiv 1 \bmod 2.
\end{equation}
In particular, since this requires $d$ to be odd, the extension $F/F_0$ is forced to be ramified at at least one finite place (otherwise the product formula for the norm residue symbol $(-1,F/F_0)$ would fail). On the other hand, if the congruence (\ref{congruence}) is satisfied, then the hermitian space $W^\flat$ will exist and admit a lattice $\Lambda^\flat$ as above.

Now consider the case when $n=2m$ is even. If $F/F_0$ is ramified at some finite place, then a hermitian space $W^\flat$ admitting a lattice $\Lambda^\flat$ as above will exist for any $d$ and $m$. On the other hand, we claim that that $F/F_0$ being everywhere unramified is again disallowed.  Indeed, the same argument applies: our assumptions on $\Lambda^\flat$ would again imply that $W^\flat$ is split at all finite places of $F_0$, and the Hasse invariants at the infinite places are again given by $(-1)^{m-1}$ at $\varphi_0$ and by $(-1)^m$ at each $\varphi \neq \varphi_0$.  Hence we again obtain the congruence $dm \equiv 1 \bmod 2$, forcing $d$ to be odd.

In future work, we plan to handle more cases of an AT conjecture which will allow us to weaken the constraints we have placed above on the extension $F/F_0$ and on the lattices $\Lambda$ and $\Lambda^\flat$; this will allow for more general hermitian spaces $W$ and $W^\flat$.  More precisely, we expect to allow places $v$ which are inert in $F$ and unramified over $\BQ$ such that $\Lambda_v$ is self-dual and $\Lambda_v^\flat$ is almost self-dual, cf.\ Remark \ref{more atc}.  
\end{remark}

Let us continue with the conditions we impose on $\Lambda^\flat$ and $\Lambda$.  When $n$ is odd, we require that
\[
   \Lambda = \Lambda^\flat \oplus O_Fu,
\]
and that $(u,u)$ is a unit at each finite place $v$ unless $v$ is inert in $F$ and $W_v$ is non-split, in which case $\ord_v(u,u) = 1$ (and hence $(v,\Lambda_v)$ is of AT type \eqref{almost self-dual type}).  In this case, we have closed embeddings
\begin{equation}
   \CM_{K_{\wt H}^\circ}(\wt H) \inj \CM_{K_{\wt G}^\circ}(\wt G)
	\quad\text{and}\quad
	\CM_{K_{\wt H}^\circ}(\wt H) \inj \CM_{K_{\wtHG}^\circ}(\wtHG)
\end{equation}
completely analogous to those we have considered before, e.g.\ (\ref{embeddings semi-global hyperspecial}), (\ref{embeddings semi-global split}), and (\ref{embeddings semi-global AT}).

Now suppose that $n$ is even, and let $v_1,\dotsc,v_m$ be the finite places of $F_0$ which ramify in $F$, cf.\ Remark \ref{latticecond herm}.  As in the discussion in Section \ref{subsec AT} after Remark \ref{unit}, for simplicity we assume that each $v_i$ is of degree one over $\BQ$.  By our assumptions already made, each $(v_i,\Lambda_{v_i}^\flat)$ is of AT type (\ref{almost pi-mod type}).  For each $i$, we further require that $(u,u)$ is a unit at $v_i$, and that $\Lambda_{v_i}$ is one of the two lattices for which the relation \eqref{type 2 lattice relation} holds with $v_i$ in place of $v_0$. Then $(v_i,\Lambda_{v_i})$ is indeed of AT type \eqref{pi-mod type}.  At the split and inert places $v$, we again require that $\Lambda_v = \Lambda_v^\flat \oplus O_{F,v}u$, where $(u,u)$ is a unit at $v$ unless $v$ is inert in $F$ and $W_v$ is non-split, in which case $\ord_v(u,u) = 1$.  Then we obtain natural global analogs of the diagrams \eqref{semi-global corresp diag G} and \eqref{semi-global corresp diag HG},
\begin{equation*}
	\begin{gathered}
   \xymatrix{
	     &  & \CR_{\wt G}' \ar[dl]_-{\pi_1} \ar[dr]^-{\pi_2}\\
	   \CM_{K_{\wt H}^\circ}(\wt H) \ar@{^{(}->}[r]  &  \CR_{\wt G}  & &  \CM_{K_{\wt G}^\circ}(\wt G)
	}
	\end{gathered}
\end{equation*}
and
\begin{equation*}
	\begin{gathered}
   \xymatrix{
	     &  & \CR_{\wtHG}' \ar[dl]_-{\pi_1} \ar[dr]^-{\pi_2}\\
	   \CM_{K_{\wt H}^\circ}(\wt H) \ar@{^{(}->}[r]  &  \CR_{\wtHG}  & &  \CM_{K_{\wtHG}^\circ}(\wtHG).
	}
	\end{gathered}
\end{equation*}
Here the global analogs $\CR_{\wt G}$, $\CR_{\wt G}'$, $\CR_{\wtHG}$, and $\CR_{\wtHG}'$ of the auxiliary spaces appearing in \eqref{semi-global corresp diag G} and \eqref{semi-global corresp diag HG} are defined in the obvious way.  Similarly, we obtain a closed embedding
\begin{equation}\label{global corresp diag HG}
   \CM_{K_{\wt H}^{\circ\prime}}(\wt H) := \CM_{K_{\wt H}^\circ}(\wt H) \times_{\CR_{\wtHG}} \CR_{\wtHG}' \inj \CM_{K_{\wtHG}^\circ}(\wtHG).
\end{equation}

We note that $\CM_{K_{\wt H}}(\wt H)$ and $\CM_{K_{\wtHG}}(\wtHG)$ are smooth over the open subscheme of $\Spec O_E$ obtained by removing all places $\nu$ for which the induced pair $(v_0,\Lambda_{v_0})$ is of AT type \eqref{almost self-dual type}. Furthermore, $\CM_{K^\circ_{\wtHG}}(\wtHG)$ has semi-stable reduction over the open subscheme of $\Spec O_E$ obtained by removing all places $\nu$ for which $E_\nu$ is ramified over $\BQ_p$.

\subsection{Drinfeld level structure}\label{global mod prob nontrivial level}
We continue with the setup at the beginning of the previous subsection.  In particular, we have the lattice $\Lambda \subset W$ satisfying the relation (\ref{ATvertex}) and the assumptions that follow it.  Let
\begin{equation}\label{Sigma spl Phi}
\begin{aligned}
\Sigma^{\spl} &:= \{\,  \text{places $v$ of $F_0\mid v$ splits in $F$} \,\},\\
   \Sigma^{\spl,\Phi} &:= \{\,  \text{$v\in \Sigma^{\spl}\mid$ every place $\nu$ of $E$ above $v$ matches $\Phi$} \,\},
   \end{aligned}
\end{equation}
cf.\ (\ref{cond CM}). In addition, we fix a function
\begin{equation}\label{bm}
   \bm\colon \Sigma^\spl\to \BZ_{\geq 0} 
\end{equation}
with finite support contained in $\Sigma^{\spl,\Phi}$.

Associated to these data is the open compact subgroup
\[
   K^\bm_G := \bigl\{\, g\in G(\BA_{F_0, f})\bigm| g(\Lambda\otimes_{O_F}\wh O_F)=\Lambda\otimes_{O_F}\wh O_F \text{ and } g\equiv \id \bmod N(\bm) \,\bigr\},
\]
where
\[
   N(\bm) := \prod_{v \in\Sigma^\spl}\fkp_v^{\bm (v)}.
\] 
As usual, we define $K_{\wt G}^\bm := K_{Z^\BQ} \times K_G^\bm$ as in \eqref{K_wtG}.  Note that if $\bm=0$, then $K^\bm_G=K^\circ_{G}$ and $K_{\wt G}^\bm =K^\circ_{\wt G}$.  

The subgroup $K_{\wt G}^\bm$ defines a moduli stack  $M_{K_{\wt G}^\bm}(\wt G)$ as in Section \ref{moduli problem over E}, which maps via a finite flat morphism to $M_{K_{\wt G}^\circ}(\wt G)$.  We then define $\CM_{K_{\wt G}^\bm}(\wt G)$ to be the normalization of $\CM_{K_{\wt G}^\circ}(\wt G)$ in $M_{K_{\wt G}^\bm}(\wt G)$.

As with Theorem \ref{globalnolevel}, the proof of the following theorem is routine.

\begin{theorem}\label{globdrin}
$\CM_{K^\bm_{\wt G}}(\wt G)$ is a regular Deligne--Mumford stack finite and flat  over $\CM_{K^\circ_{\wt G}}(\wt G)$. For $N(\bm)$ big enough,  $\CM_{K^\bm_{\wt G}}(\wt G)$ is relatively representable over $\CM_0^\xi$.   For every finite place $\nu$ of $E$, the base change $\CM_{K^\bm_{\wt G}} \times_{\Spec O_{E}} \Spec O_{E, (\nu)}$ is canonically isomorphic to (variants of) the semi-global moduli space defined in one of Sections \ref{subsec hyper}, \ref{semi-global drinfeld}, or \ref{subsec AT} above.\qed
\end{theorem}

Note that in the last assertion of this theorem, in the case that $\bm$ is not identically zero, we are implicitly allowing smaller open compact subgroups of $K_p$ in these subsections (this is what is meant by `variants'). 
We analogously define the DM stacks $\CM_{K_{\wt H}^\bm}(\wt H)$ and $\CM_{K_{\wtHG}^\bm}(\wtHG)$ over $\Spec O_E$.

To define embeddings between the stacks we have introduced, we make the same assumptions on the lattices, on $(u,u)$, and on the ramified places of $F_0$ when $n$ is even as after Theorem \ref{globalnolevel} in Section \ref{global mod prob trivial level}. When $n$ is odd, we analogously obtain closed embeddings
\begin{equation}
   \CM_{K_{\wt H}^\bm}(\wt H) \inj \wt\CM_{K_{\wt G}^\bm}(\wt G)
	\quad\text{and}\quad
	\CM_{K_{\wt H}^\bm}(\wt H) \inj \wt\CM_{K_{\wtHG}^\bm}(\wtHG).
\end{equation}
When $n$ is even, we analogously obtain a closed embedding
\begin{equation}
   \CM_{K_{\wt H}^{\prime\bm}}(\wt H) \inj \CM_{K_{\wtHG}^\bm}(\wtHG).
\end{equation}
Also, in all cases we obtain Hecke correspondences for elements $g\in \prod_{v\in \Sigma^{\spl,\Phi}}'(H\times G)(F_{0,v})$ (restricted direct product),
\begin{equation}\label{globalDrHecke}
\begin{gathered}
   \xymatrix{
	     & \CM_{K_{\wtHG}^{\boldsymbol{\mu}}}(\wtHG) \ar[dl]_-{\text{nat}_1} \ar[dr]^-{\text{nat}_g}\\
	   \CM_{K_{\wtHG}^\bm}(\wtHG)  & &  \CM_{K_{\wtHG}^\bm}(\wtHG).
	}
\end{gathered}
\end{equation}

\section{The Arithmetic Gan--Gross--Prasad conjecture}\label{section ggp}
In this section, we state a version of the Arithmetic Gan--Gross--Prasad conjecture \cite{GGP}. It is based on some widely open standard conjectures about algebraic cycles.

\subsection{Standard conjectures on height pairing}
\label{sec std conj}

Consider the category $\sV$ of smooth proper varieties over a number field $E$, and let $H^*\colon \sV\to \mathrm{grVec}_K$ be a Weil cohomology theory with coefficients in a field $K$ of characteristic zero. Let $X$ be an object in $\sV$, and let $\Ch^i(X)$ be the group of codimension-$i$ algebraic cycles in $X$ modulo rational equivalence. We have a cycle class map
$$
   \cl_i\colon \Ch^i(X)_\BQ\to H^{2i}(X).
$$
Its kernel is the group of cohomologically trivial cycles, denoted by $\Ch^i(X)_{\BQ,0}$. We take the Weil cohomology theory $H^*$ as either the Betti cohomology
$
   H^*(X(\BC),\BQ),
$
or, for a prime $\ell$, the $\ell$-adic cohomology $H^*(X\otimes_E \ov  E,\BQ_\ell)$, endowed with its continuous $\Gal(\ov E/E)$-action. 
Comparison theorems between Betti cohomology and \'etale cohomology  show that the subspace $\Ch^i(X)_{\BQ,0}$ is independent of the choice of these two.

We are going to base ourselves on the following conjectures of Beilinson and Bloch, cf.~\cite[\S2]{Ja}. 
\begin{conjecture}\label{ex reg}
There exists a regular proper flat model $\CX$ of $X$ over  $\Spec O_E$.
\end{conjecture}

Let $\CX$ be such a model, and consider its $i$th Chow group $\Ch^i(\CX)_\BQ$. 
Restriction to the generic fiber defines a map
$$
\Ch^i(\CX)_{\BQ}\to \Ch^i(X)_{\BQ}.
$$
Let  $\Ch^i_\fin(\CX)_\BQ$ be the kernel of this map (cycles supported on ``finite fibers"), and let $\Ch^i(\CX)_{\BQ,0}$ be the pre-image of $\Ch^i(X)_{\BQ,0}$. We are going to use the Arakelov  pairing defined in Gillet--Soul\'e \cite[\S4.2.10]{GS},
\begin{equation}\label{eqn GS0}
   \sform_{\GS}\colon \Ch^i(\CX)_{\BQ,0}\times \Ch^{d+1-i}(\CX)_{\BQ,0}\to \BR,\quad d:=\dim X.
\end{equation}
Let $\Ch^{d+1-i}_\fin(\CX)_\BQ^\perp\subset \Ch^i(\CX)_{\BQ,0}$ be the orthogonal complement of $\Ch^{d+1-i}_\fin(\CX)_\BQ$ under the pairing \eqref{eqn GS0}.

\begin{conjecture}\label{conj lift}
The natural map $\Ch^{d+1-i}_\fin(\CX)_\BQ^\perp\to \Ch^i(X)_{\BQ,0}$ is surjective. 
\end{conjecture}

Assuming Conjecture \ref{ex reg} and \ref{conj lift} above, the height pairing of Beilinson and Bloch 
\begin{equation}
\label{eqn BB}
  \sform_\mathrm{BB}\colon \Ch^i(X)_{\BQ,0}\times \Ch^{d+1-i}(X)_{\BQ,0}\to \BR
\end{equation}
can be defined as follows. Lift the elements $c_1\in \Ch^i(X)_{\BQ,0}$ and $c_2\in \Ch^{d+1-i}(X)_{\BQ,0}$ to $\wt c_1\in \Ch^{d+1-i}_\fin(\CX)_\BQ^\perp$ and $\wt c_2\in \Ch^{i}_\fin(\CX)_\BQ^\perp$, respectively. Define 
\begin{equation}\label{BB=GS}
   (c_1,c_2)_{\mathrm{BB}} :=(\wt c_1,\wt c_2)_{\GS}.
\end{equation}
It is easy to see that this is independent of the choices of the liftings.

\begin{remark}Assuming Conjectures \ref{ex reg} and \ref{conj lift}, the pairing \eqref{eqn BB} is independent of the choice of the (regular proper flat) integral model $\CX$, cf.\  \cite[Lem.\ 1.5]{Ku-ens}. 
\end{remark}

\begin{remark}\label{everygood}
Assume that there exists a smooth proper model $\CX$ of $X$ over $O_E$. Then by \cite[Th.\ 6.11]{Ku-duke}, Conjecture \ref{conj lift} holds for $\CX$ and therefore the intersection product \eqref{eqn BB} is defined; again, by \cite[Lem.\ 1.5]{Ku-ens}, the intersection product is independent of the choice of  $\CX$ (assumed to be smooth and proper). 

\end{remark}

\subsection{Cohomology and Hecke--K\"unneth projectors}\label{ss:HKproj}

We apply these considerations to the Shimura varieties defined in Section \ref{Shimura varieties}.\footnote{Note that these spaces can in fact be (and in particular cases of interest to us, are) DM stacks, but we will suppress this point in our discussion throughout the rest of the paper. The extension of the usual intersection theory to DM stacks is supplied by Gillet's paper \cite{Gi}.} In order to simplify notation, we write $K$ for $K_{\wtHG}$ in  $\Sh_{K_{\wtHG}}(\wtHG)$ throughout the rest of this section.

Denote by  $ \sH_K$ the Hecke algebra of bi-$K$-invariant $\BQ$-valued  functions with compact support, with multiplication given by the convolution product,
  \begin{equation}\label{eqn HK}
   \sH_K := C_c^\infty\bigl(\wtHG(\BA_f)/\!\!/K, \BQ\bigr).
\end{equation}The variety $\Sh_K(\wtHG)$ is equipped with a collection of algebraic correspondences ($\BQ$-linear combinations of algebraic cycles on the self-product of $\Sh(\wtHG)_K$ with itself), the Hecke correspondences associated to $f\in \sH_K$.  

The Hodge conjecture implies that the K\"unneth projector from $\bigoplus_{i\in\BZ}H^i(\Sh_K(\wtHG),\BQ)$ to each summand $H^i(\Sh_K(\wtHG),\BQ)$, or to the primitive cohomology, is induced by algebraic cycles (with $\BQ$-linear combinations).  Morel and Suh \cite{MS}  prove a partial result on the algebraicity of K\"unneth projectors for Shimura varieties (the so-called ``standard sign conjecture"), conditional on Arthur's conjecture. Now we recall their theorem. 
\begin{remark}
The hypotheses on which \cite{MS} is based are in fact known in our setting: 1) The stabilization of the twisted trace formula is known by the work of Waldspurger. 2)  The Arthur conjecture on the expression of the discrete spectrum in terms of discrete Arthur parameters is known (cf.\ Mok \cite{Mok} for quasi-split unitary groups, and Kaletha--Minguez--Shin--White \cite{KMSW} (and its sequels) for inner forms of unitary groups). Here we note that our group $\wtHG$ is a product of a unitary group and a torus by \eqref{proddec}. 3) The comparison between the Adams--Johnson classification and the Arthur classification of cohomological Arthur parameters of real groups is known by Arancibia--Moeglin--Renard \cite{AMR}.
 \end{remark}

\begin{theorem}[Morel--Suh]
\label{thm Kunneth}
Let $\varepsilon\in \BZ/2\BZ$. 
Then there exists $f^\varepsilon$  in $\sH_{K}$ such that the associated Hecke correspondence induces the projector to the even, resp.\ odd,  degree cohomology,
\[
   \bigoplus_{i\in\BZ}H^i\bigl(\Sh_K(\wtHG),\BQ\bigr)\to \bigoplus_{i\equiv\varepsilon \bmod  2}H^{i}\bigl(\Sh_K(\wtHG),\BQ\bigr).
\]
\end{theorem}

\begin{definition}\label{def:HK proj}
 We set $f^+=f^0$ and $f^-=f^1$ as in Theorem \ref{thm Kunneth}, and call them  the even, resp.\ odd, \emph{Hecke--K\"unneth projectors} in $\sH_{K}$. 
\end{definition}

\begin{proof}[Proof of Theorem \ref{thm Kunneth}]

This follows from \cite[Th.\ 1.4, Lem.\ 2.2]{MS}. We outline the proof in loc.\ cit. It suffices to prove the assertions after tensoring with $\BC$. We henceforth consider $H^i(\Sh_K(\wtHG),\BC)$ with the action by $\sH_{K,\BC}=\sH_K\otimes\BC$.

By Matsushima's formula \cite[VII]{BW}, we have a decomposition of the Betti cohomology\footnote{When the Shimura variety $\Sh_K(\wtHG)$ is non-compact, one has to replace $H^*(\Sh_K(\wtHG),\BC)$  by the image of the Betti cohomology of the toroidal compactification in the Betti cohomology of $\Sh_K(\wtHG)$. This also coincides with the intersection cohomology of the Baily--Borel compactification of  $\Sh_K(\wtHG)$. } into a finite direct sum, 
\begin{equation}\label{eqn coh}
   H^*\bigl(\Sh_K(\wtHG),\BC\bigr) \simeq 
	   \bigoplus_{\pi=\pi_f\otimes \pi_\infty} V_{\pi, K},
\end{equation}
 where we set
$$
V_{\pi, K}:= m_{\disc}(\pi)\,\bigl( \pi_f^K\otimes H^*(\wt{\fkh\fkg},K_\infty;\pi_\infty)\bigr).
$$
Here 
\begin{altitemize}
\item $\pi$ runs through the set  $\Pi_{\rm disc}(\wt\HG)$ of  irreducible representations of $\wt\HG(\BA)$ in the  discrete automorphic spectrum $L^2_{\disc}(Z^\BQ(\BR) \wtHG(\BQ)\bs \wtHG(\BA))$, and $ m_{\disc}(\pi)$ is the  multiplicity of $\pi$;\footnote{When the group $\wtHG$ is anisotropic modulo its center, the quotient is compact and then $L^2_{\disc}=L^2$.}
\item $\wt{\fkh\fkg}$ is the complex Lie algebra of $\wtHG(\BR)$;
\item $K_\infty$ is the centralizer of $h_{\wtHG}$ in $\wtHG(\BR)$;\footnote{In particular, $Z^\BQ(\BR)$ is contained in $K_\infty$.} and
\item $\pi_f^K$ denotes the invariants of $\pi_f$ under $K$.
\end{altitemize}
The isomorphism \eqref{eqn coh} is $\sH_K$-equivariant, where  $\sH_K$ acts on the right hand side through the space $\pi^K_f$. 

We have a decomposition of the discrete spectrum as $\wtHG(\BA)$-modules (cf.\ \cite{Ar}, \cite[\S8]{Ko89}),
\[
   L^2_{\disc} \bigl(Z^\BQ(\BR) \wtHG(\BQ)\bs \wtHG(\BA)\bigr)\simeq   \bigoplus_{\psi}  \bigoplus_{\pi\in \Pi_\psi}
m(\psi,\pi)\cdot \pi .
\]
Here the sum is taken over all equivalence classes of global Arthur parameters $\psi$, with trivial associated quasi-character of $Z^\BQ(\BR)$. Furthermore, $\Pi_\psi$ denotes the Arthur packet attached to $\psi$, and $m(\psi,\pi)$ denotes the Arthur multiplicity. Hence we may rewrite the decomposition \eqref{eqn coh} according to 
global Arthur parameters as 
\begin{equation}\label{eqn coh A}
   H^*\bigl(\Sh_K(\wtHG),\BC\bigr) \simeq 
	   \bigoplus_{\psi}  V_{\psi,K} ,
\end{equation}
where we set
\begin{equation}\label{eqn coh pi}
   V_{\psi,K}:=
   \bigoplus_{\pi=\pi_f\otimes\pi_\infty\in \Pi_\psi} m(\psi,\pi) \bigl(\pi_f^K\otimes 
      H^*(\wt{\fkh\fkg},K_\infty;\pi_\infty)\bigr).
\end{equation}
The isomorphism \eqref{eqn coh A} is $\sH_{K}$-equivariant. Moreover,  there is a canonical Lefschetz class (coming from the cup product with the Killing form) which induces an $\SL_2(\BC)$-action on the graded vector space $H^*(\Sh_K(\wtHG),\BC)$, cf.\ \cite[Prop.\ 9.1]{Ar}. Correspondingly there is an  $\SL_2(\BC)$-action on the graded vector space $H^*(\wt{\fkh\fkg},K_\infty;\pi_\infty)$. The decomposition respects the $\SL_2(\BC)$-action  and the grading on both sides of \eqref{eqn coh A}. We refer to \cite{Ar} for details.

In general the definition of Arthur parameter involves the hypothetical automorphic Langlands group. In the case of classical groups, Arthur avoids the Langlands group by using cuspidal (or isobaric) automorphic representations of general linear groups as substitute parameters \cite{Ar13}. In our setting, our group is a product $\wtHG=Z^\BQ\times \HG$, 
where
\begin{equation}
   \HG := \Res_{F_0/\BQ}(H\times G) ,
\end{equation}
cf.\ \eqref{proddec}. Accordingly, we will write an Arthur parameter $\psi$ as  a pair $(\psi_0,\psi_1)$ where $\psi_0$ and $\psi_1$ are Arthur parameters for the two factors $Z^\BQ$ and $\HG$  and those are defined in terms of cuspidal (or isobaric) automorphic representations of general linear groups.

By the product decomposition  \eqref{reltogross} we have
$$
\Sh_K(\wtHG)_\BC\simeq \Sh_{K_{Z^\BQ}}(Z^\BQ)_\BC\times_{\Spec\BC} \Sh_{K_{\HG}}(\HG)_\BC
$$and an induced isomorphism (note that the first factor above is zero-dimensional)
\begin{equation}\label{eqn coh prod}
   H^*\bigl(\Sh_K(\wtHG),\BC\bigr) \simeq   H^0\bigl(\Sh_{K_{Z^\BQ}}(Z^\BQ),\BC\bigr) \otimes
  H^*\bigl(\Sh_{K_{\HG}}(\HG),\BC\bigr).
\end{equation}
We may therefore replace $\Sh_K(\wtHG)$ by $ \Sh_{K_{\HG}}(\HG)$ in the conclusion of Theorem \ref{thm Kunneth}. We record the decomposition similar to \eqref{eqn coh} 
\begin{equation}\label{eqn coh '}
   H^*\bigl(\Sh_{K_{\HG}}(\HG),\BC\bigr) \simeq 
	   \bigoplus_{\pi_1}  V_{\pi_1,K_{\HG}}.
\end{equation}

 For a fixed $\pi_{1,f}$, let 
$$
\Pi_\infty(\pi_ {1,f}) := 
   \bigl\{\, \pi_{1,\infty}\in {\Pi}\bigl(\HG(\BR)\bigr) \bigm| \pi_{1,f}\otimes\pi_{1,\infty}\in \Pi_{\rm disc}(\HG) \,\bigr\} .
$$
Here ${\Pi}(\HG(\BR))$ denotes the set of equivalence classes of irreducible admissible representations of $\HG(\BR)$. 
Then we may rewrite \eqref{eqn coh '} as
\begin{equation}\label{eqn coh ''}
   H^*\bigl(\Sh_{K_{\HG}}(\HG),\BC\bigr) \simeq 
	   \bigoplus_{\pi_{1,f}}  V_{\pi_{1,f},K_{\HG}},
\end{equation}  
where 
$$
V_{\pi_{1,f},K_{\HG}}:= \bigoplus_{ \pi_{1,\infty}
\in \Pi_\infty(\pi_ {1,f}) } V_{\pi_{1,\infty}\otimes\pi_{1,f},K_{\HG}}.
$$
Now by \cite[Th.\ 1.4]{MS}, the degree $i$ modulo $2$  such that $H^i(\fkh\fkg,K_\infty;\pi_{1,\infty})\neq 0$ is  constant (in $\BZ/2\BZ$)  as  $\pi_{1,\infty}$ varies through $\Pi_\infty(\pi_{1,f})$ and  $i$ varies through $\BZ$. We denote the constant $i$ mod $2$ above by $\varepsilon(\pi_{1,f})\in\BZ/2\BZ$.

The (finitely many) non-zero direct summands $V_{\pi_{1,f},K_{\HG}}$ are distinct $\sH_{K,\BC}$-modules. Therefore, for each $\pi_{1,f}$, there exists $f_{\pi_{1,f}}\in \sH_{K,\BC}$ that induces the projector to  $V_{\pi_{1,f},K_{\HG}}$. Now set 
$$
   f^+:=\sum_{\substack{\pi_{1,f}\\ \varepsilon(\pi_{1,f})=0}}f_{\pi_{1,f}}
   \quad\text{and}\quad
   f^-:=\sum_{\substack{\pi_{1,f}\\ \varepsilon(\pi_{1,f})=1}}f_{\pi_{1,f}}.
$$
This completes the proof of Theorem \ref{thm Kunneth}.
\end{proof}

\begin{remark}\label{rmk support}
Let $H^\varepsilon_K:=\bigoplus_{i\equiv\varepsilon \bmod  2}H^{i}(\Sh_K(\wtHG),\BQ)$. Then \cite[Th.\ 1.4]{MS} asserts that $H^+_K$ and $H^-_K$, both being semisimple $\sH_{K}$-modules, do not share any  common irreducible $\sH_K$-submodule.
\end{remark}

\subsection{Arithmetic diagonal cycles}
We now apply the considerations of Sections \ref{sec std conj} and \ref{ss:HKproj} to the canonical model $M_K(\wt\HG)$ of $\Sh_K(\wtHG)$ over $E$. When $K=K_{\wt\HG}$ is as in \eqref{K_wtHG decomp}, we have given in Section \ref{moduli problem over E} a moduli interpretation of $M_K(\wt\HG)$. 
 
 We consider  the cycle class map in degree $n-1$ (for the Betti cohomology),
$$
   \cl_{n-1}\colon \Ch^{n-1}\bigl(M_K(\wtHG)\bigr)_\BQ \to H^{2(n-1)}\bigl(\Sh_K(\wtHG), \BQ\bigr).
$$
Note that $\dim M_K(\wtHG))=2n-3$ is odd. In particular the above cohomology group is \emph{not} in the middle degree, but just above.

Let $K_{\wt H}$ be a compact open subgroup of $\wt H(\BA_f)$ contained in $K\cap \wt H(\BA_f)$. We have a finite and unramified morphism
$$
M_{K_{\wt H}}(\wt H)\to M_{K}(\wtHG).
$$
The proper push-forward defines a cycle class $[M_{K_{\wt H}}(\wt H)]\in \Ch^{n-1}(M_K(\wtHG))_\BQ$.
We now fix a Haar measure on $\wt H(\BA)$ such that the volume  $\vol(K_{\wt H})\in \BQ$. We then define the normalized class
\begin{align}
\label{eqn zK}
z_K=\vol(K_{\wt H})[M_{K_{\wt H}}(\wt H)]\in  \Ch^{n-1}\bigl(M_K(\wtHG)\bigr)_\BQ,
\end{align}
which is independent of the choice of the group $K_{\wt H}$. 
We call $z_K$ the \emph{arithmetic diagonal cycle} since it lies in the arithmetic middle dimension, i.e., $2\dim z_K= \dim M_K(\wtHG)+1$. Let $\CZ_{K}$ be the cyclic Hecke submodule of $\Ch^{n-1}(M_K(\wtHG))_\BQ$ generated by $z_K$.

Let $f^-\in\sH_{K}$ be an odd Hecke--K\"unneth projector, cf.\ Definition \ref{def:HK proj}. We obtain a map to the Chow group of cohomologically trivial cycles, 
\begin{align}\label{eqn Rf}
 R(f^-)\colon \Ch^{n-1}\bigl(M_K(\wtHG)\bigr)_\BQ\to \Ch^{n-1}\bigl(M_K(\wtHG)\bigr)_{\BQ,0}.
 \end{align}
 We obtain a cohomologically trivial cycle in the Chow group,
\begin{equation}
   z_{K,0}=R(f^-)( z_K)\in \Ch^{n-1}\bigl(M_K(\wtHG)\bigr)_{\BQ,0}.
\end{equation}
Accordingly we call $z_{K,0}$ the \emph{cohomologically trivial arithmetic diagonal cycle}. 

\begin{remark}
\begin{altenumerate}
\item\label{rct i} Since the Hecke--K\"unneth projector $f^-\in\sH_{K}$ is not unique, we comment on the (in)dependence of the induced  cohomologically trivial cycle in the Chow group $z_{K,0}=R(f^-)( z_K)\in \Ch^{n-1}(M_K(\wtHG))_{\BQ,0}$. 
Conjecturally,  for any smooth projective variety $X$ over a number field,  it should be true that,  given $i$, $\Ch^i(X)_{\BQ}$ is finite dimensional, and if a correspondence induces the zero endomorphism on $H^{2i-1}(X)$, then the induced endomorphism on $\Ch^i(X)_{\BQ, 0}$ is zero, cf.\  \cite[Lem.\ 5.6]{Be} and \cite[Conj.\ 9.12]{Ja}.  Moreover, one expects that $\Ch^{i}(M_K(\wtHG))_{\BQ}$ is a \emph{semisimple} $\sH_K$-module. If these were true for $X=M_K(\wtHG)$ and $i=n-1$, then by Remark \ref{rmk support} it would follow that the simple $\sH_K$-modules appearing in  $\Ch^{n-1}(M_K(\wtHG))_{\BQ, 0}$ and $  \Im(\cl_{n-1})$  are disjoint. Furthermore,  one expects that there is a unique decomposition of (finite dimensional) semisimple $\sH_K$-modules
\begin{equation}\label{dec ashecke}
   \Ch^{n-1}\bigl(M_K(\wtHG)\bigr)_{\BQ}= \Ch^{n-1}\bigl(M_K(\wtHG)\bigr)_{\BQ, 0}\oplus \Im(\cl_{n-1}).
\end{equation}
The endomorphism $R(f^-)$ would then define the projector onto the summand $\Ch^{n-1}(M_K(\wtHG))_{\BQ, 0}$. In particular, the induced endomorphism of $\Ch^{n-1}(M_K(\wtHG))_{\BQ}$, and hence the element $z_{K,0}=R(f^-)( z_K)\in \Ch^{n-1}(M_K(\wtHG))_{\BQ,0}$, would be independent of the choice of  $f^-\in\sH_{K}$.
The summands in \eqref{dec ashecke} would decompose as $\sH_K$-modules,
  \begin{equation}\label{dec bothsums}
  \begin{aligned}
  \Ch^{n-1}(M_K(\wtHG))_{\BC,0}&=\bigoplus_{\pi}\Ch^{n-1}\bigl(M_K(\wtHG)\bigr)_{\BC, 0}[\pi^K_f] , \\
   \Im(\cl_{n-1})_\BC&=\bigoplus_{\sigma}\Im(\cl_{n-1})_{\BC}[\sigma^K_f] . 
  \end{aligned}
  \end{equation}
Here $\pi$ runs through all \emph{automorphic} representations contributing to $H^{2n-3}(\Sh_K(\wtHG), \BC)$, and $\sigma $ runs through all \emph{automorphic} representations contributing to $H^{2n-2}(\Sh_K(\wtHG), \BC)$. Also, $\Ch^{n-1}(M_K(\wtHG))_{\BC,0}[\pi_f^K]$ denotes the $\pi^K_f$-isotypic component of the $\sH_K$-module, i.e., the image under the evaluation map, which is injective, 
$$
   \pi_f^K\otimes \Hom_{\sH_K}\Bigl(\pi_f^K,\Ch^{n-1}\bigl(M_K(\wtHG)\bigr)_{\BC,0}\Bigr) \to \Ch^{n-1}\bigl(M_K(\wtHG)\bigr)_{\BC,0} .
$$ 
Then  
\begin{equation}
   \Ch^{n-1}\bigl(M_K(\wtHG)\bigr)_{\BC,0}[\pi_f^K]\simeq \pi_f^K\otimes \Hom_{\sH_K}\Bigl(\pi_f^K,\Ch^{n-1}\bigl(M_K(\wtHG)\bigr)_{\BC,0}\Bigr).
\end{equation}
Similar definitions apply to   $\Im(\cl_{n-1})_{\BC}[\sigma^K_f]$.  
\item In \cite{Z09} the \emph{Chow--K\"unneth decomposition} was used to modify the arithmetic diagonal cycle to make its cohomology class trivial. However, it is  difficult to show the existence of a Chow--K\"unneth decomposition except in some special cases.  Therefore, the procedure above is preferable.
\end{altenumerate}\label{rem coh triv}
\end{remark}

\subsection{The Arithmetic Gan--Gross--Prasad conjecture, for a fixed level $K\subset \wtHG(\BA_f)$}
\label{sec AGGP}
Let $\CZ_{K,0}$ denote the cyclic Hecke submodule of $\Ch^{n-1}(M_K(\wtHG))_{\BQ,0}$ generated by $z_{K,0}$ or, equivalently, the image of $\CZ_K$ under the map $ R(f^-)$.
We would like to decompose as $\sH_K$-modules, 
$$
\CZ_{K,0}\subset \Ch^{n-1}\bigl(M_K(\wtHG)\bigr)_{\BC,0}.
$$
In the conjectural decomposition \eqref{dec bothsums}, we only consider the \emph{tempered} part of the spectrum. The non-tempered part is also interesting but will be postponed to the future.

Let $\pi$ be an automorphic representation of $\wtHG(\BA)$ with trivial restriction to the central subgroup $Z^\BQ(\BA)$.  By \eqref{proddec} we may consider $\pi$ as an automorphic representation of $(H\times G)(\BA_{F_0})$. Let  $R$ be the tensor product representation of the $L$-group of $H\times G$ defined in \cite[\S 22]{GGP}.  The $L$-function $L(s,\pi,R)$ depends only on the Arthur parameter $\psi$ of $\pi$. 

To explain this $L$-function, we write a tempered Arthur parameter $\psi=\psi^{(n-1)}\boxtimes \psi^{(n)}$ formally as  
\begin{align}\label{A para}
\psi^{(n-1)}=\bigoplus_{i}\psi^{(n-1)}_{i},\quad \psi^{(n)}=\bigoplus_{j}\psi^{(n)}_{j} ,
\end{align}
where $\psi_i^{(n-1)}$ correspond to \emph{distinct} cuspidal automorphic representations of $\GL_{N^{(n-1)}_i}$ such that $\sum N^{(n-1)}_i=n-1$, and similarly for $\psi^{(n)}$. Then the $L$-function in question equals 
\begin{align}\label{RS}
   L(s,\pi,R):=\prod_{i,j} L\bigl(s,\psi^{(n-1)}_{i}\boxtimes \psi^{(n)}_{j}\bigr),
\end{align}
where each factor is a Rankin--Selberg convolution. 

We assume that $\pi$ lies in the packet $\Pi_\psi$ of a cohomological tempered packet $\psi$. Here ``cohomological" is a condition on the archimedean component $\psi_\infty$ and refers to  the trivial coefficient system.

The first version of the Arithmetic Gan--Gross--Prasad conjecture can now be stated as follows.
\begin{conjecture} \label{conj AGGP 1}
Let $K\subset \wtHG(\BA_f)$ be an open compact subgroup. Let $\pi$ be as above, i.e., with trivial restriction to the central subgroup $Z^\BQ(\BA)$ and lying in a cohomological tempered Arthur packet.
Consider  the following conditions on $\pi$. 
\begin{altenumerate}
\renewcommand{\theenumi}{\alph{enumi}}
\item\label{conj AGGP 1 a}  $\dim \Hom_{\sH_K}(\pi_f^K, \CZ_{K,0})=1$.\footnote{Note that we always have $\dim \Hom_{\sH_K}(\pi_f^K, \CZ_{K,0})\leq 1$, because $ \CZ_{K,0}$ is a cyclic $\sH_K$-module. }
\item\label{conj AGGP 1 b} The order of vanishing $\ord_{s=1/2}L(s,\pi,R)$ equals one, the space $\Hom_{\wt H(\BA_f)}(\pi_f,\BC)$ is one-dimensional, and its generator does not vanish on the subspace $\pi_f^K\subset\pi_f$. 
\item\label{conj AGGP 1 c} $\Hom_{\sH_K}(\pi_f^K, \Ch^{n-1}(M_K(\wtHG))_{\BC,0})\neq 0$.
\end{altenumerate}
Then (\ref{conj AGGP 1 a}) and (\ref{conj AGGP 1 b}) are equivalent and imply (\ref{conj AGGP 1 c}). If $E=F$, then  (\ref{conj AGGP 1 a}), (\ref{conj AGGP 1 b}) and (\ref{conj AGGP 1 c'}) are equivalent, where
\begin{altenumerate}
\renewcommand{\theenumi}{\alph{enumi}$'$}
\setcounter{enumi}{2}
\item\label{conj AGGP 1 c'} $\dim\Hom_{\sH_K}(\pi_f^K, \Ch^{n-1}(M_K(\wtHG))_{\BC,0})=1$,  the space $\Hom_{\wt H(\BA_f)}(\pi_f,\BC)$ is one-dimensional, and its generator does not vanish on the subspace $\pi_f^K\subset\pi_f$.  
\end{altenumerate}
\end{conjecture}

\begin{remark}
If $E=F$, the equivalence between (\ref{conj AGGP 1 b}) and (\ref{conj AGGP 1 c'}) is  part of the Beilinson--Bloch conjecture \cite{Be,Bl} that generalizes the Birch--Swinnerton-Dyer conjecture.
\end{remark}

We would like to test the conjecture quantitatively through height pairings. Now we assume that Conjectures \ref{ex reg} and \ref{conj lift} hold for  $M_K(\wtHG)$. For instance, this is the case when $M_K(\wtHG)$ has everywhere good reduction,\footnote{Instances of everywhere good reduction can in fact be constructed, cf.\ Remark \ref{ex verygood}.} cf.\ Remark \ref{everygood}.  
In particular, we have the Beilinson--Bloch height pairing \eqref{eqn BB} between cohomologically trivial cycles for $i=n-1$. We extend it to a hermitian form on $\Ch^{n-1}(M_K(\wtHG))_{\BC,0}$.
Pairing against the distinguished element $z_{K,0}\in \Ch^{n-1}(M_K(\wtHG))_{\BQ,0}$ then defines a linear functional 
\begin{equation*}
   \ell_{K}\colon 
	\xymatrix@R=0ex{
	   \Ch^{n-1}\bigl(M_K(\wtHG)\bigr)_{\BC,0} \ar[r]  &  \BC\\
      z \ar@{|->}[r]  &  (z, z_{K,0})_\mathrm{BB}.
	}
\end{equation*}

Let  $\CZ_{K,0}[\pi_f^K]$ be the $\pi^K_f$-isotypic component of $\CZ_{K,0}$ as an $\sH_K$-module, so that  
\begin{equation*}
\begin{aligned}
   \CZ_{K,0}[\pi_f^K]&\simeq \pi_f^K\otimes \Hom_{\sH_K}(\pi_f^K,\CZ_{K,0}) .\\
\end{aligned}
\end{equation*}
The second version of the Arithmetic Gan--Gross--Prasad conjecture in terms of the height pairing can be stated as follows.

\begin{conjecture} \label{conj AGGP 2}
Let $K\subset \wtHG(\BA_f)$ be an open compact subgroup.  Let $\pi$ be as above. Then the following conditions on $\pi$ are equivalent.
\begin{altenumerate}
\renewcommand{\theenumi}{\alph{enumi}}
\item   $\ell_{K}|_{\CZ_{K,0}[\pi_f^K]}\neq 0 $.
\item The order of vanishing $\ord_{s=1/2}L(s,\pi,R)$ equals one, the space  $\Hom_{\wt H(\BA_f)}(\pi_f,\BC)$ is one-dimensional, and its generator does not vanish on the subspace $\pi_f^K\subset\pi_f$. 
\item $\ell_{K}|_{\Ch^{n-1}(M_K(\wtHG))_{\BC,0}[\pi_f^K]}\neq 0 $.
\end{altenumerate}
\end{conjecture}

\begin{remark}
Our formulation differs in several aspects from \cite[Conj.\ 27.1]{GGP}. First, in \cite{GGP}, the Shimura varieties are associated to unitary groups,  whereas here we consider Shimura varieties associated to groups which differ from those in loc.~cit.\ by a central subgroup, cf.\ Remark \ref{others}(\ref{ggp}). Correspondingly, the varieties in loc.~cit.\ are defined over $F$, whereas our varieties are defined over the extension $E$ of $F$. As a consequence, we cannot predict the dimension of  $ \Hom_{\sH_K}(\pi_f^K,\Ch^{n-1}(M_K(\wtHG))_{\BC,0})$ in Conjecture \ref{conj AGGP 1}(\ref{conj AGGP 1 b}) when $F\neq E$ (in the version of loc.~cit., this space is one-dimensional if it is non-zero). Note, however, that if $F= K F_0$ for an imaginary quadratic field $K$ and the CM type is induced from $K$ (as in \cite{HT}), then $F=E$. Second, we exploit that the standard sign conjecture is known in our case, and we use it to construct the cohomologically trivial diagonal cycle $z_{K,0}$ and the corresponding linear functional $\ell_K$ that occur in our version of the conjecture. Third,  we work with a fixed level $K$ and specify the compact open subgroup $K$ over which the linear functional $\ell_K$ should be non-zero. Finally, we note that, in the terminology of \cite{GGP}, we are only considering the case of a trivial local system $\CF$. 
\end{remark}

\begin{remark}
The space $\Hom_{\wt H(\BA_f)}(\pi_f,\BC)$ is at most one-dimensional. It is one-dimensional if and only if  the member $\pi$ in the packet $ \Pi_\psi$ is as prescribed by the local Gan--Gross--Prasad conjecture \cite[Conj.\ 17.3]{GGP}. The local Gan--Gross--Prasad conjecture \cite[Conj.\ 17.1]{GGP}  predicts that there is a unique $\pi$ in the packet $\Pi_\psi$ such that $\Hom_{\wt H(\BA_f)}(\pi_f,\BC)\neq 0$ (in which case $\dim\Hom_{\wt H(\BA_f)}(\pi_f,\BC)=1$).
\end{remark}

\begin{remark}
Let us restrict  our attention to the \emph{tempered part} in the decomposition in the first line of \eqref{dec bothsums}, 
$$
\Ch^{n-1}\bigl(M_K(\wtHG)\bigr)_{\BC, 0, \temp} =\bigoplus_{\pi \text{ tempered}}\Ch^{n-1}\bigl(M_K(\wtHG)\bigr)_{\BC, 0}[\pi^K_f] .
$$
Let $\CZ_{K, 0,\temp}$ be the  Hecke submodule of $\Ch^{n-1}(M_K(\wtHG))_{\BC, 0, \text{ temp}}$ generated by (the projection of) $z_{K, 0}$.   Then  Conjecture \ref{conj AGGP 1} (together with the  expectations in Remark \ref{rem coh triv}(\ref{rct i})) implies that when $E=F$,
$$
   \CZ_{K, 0, \temp}=\bigoplus_{\pi}\Ch^{n-1}\bigl(M_K(\wtHG)\bigr)_{\BC, 0}[\pi^K_f] ,
$$
where the sum runs over all tempered automorphic representations $\pi$ such that $$\ord_{s=1/2}L(\pi, s, R)=1$$
and such that the space
$\Hom_{\wt H(\BA_f)}(\pi_f,\BC)$ is one-dimensional, with generator  not vanishing on the subspace $\pi_f^K\subset\pi_f$.
\end{remark}

\begin{remark}
A parallel question is to investigate the structure of the $\sH_K$-submodule in $H^{2(n-2)}(M_K(\wtHG),\BC)$ generated by the cohomology class $\cl_{n-1}(z_{K})$. However, since every automorphic representation contributing to the cohomology in off-middle degree must be non-tempered, the answer to such a question must involve the non-tempered version of the Gan--Gross--Prasad conjecture. We hope to return to this in the future.
\end{remark}

\begin{remark}\label{ex verygood}
As remarked above, the height pairing is defined unconditionally if $M_K(\wtHG)$ has everywhere good reduction. To construct such cases, let us assume now that $K=K_{\wt\HG}=K_{Z^\BQ} \times K_H \times K_G$, where $K_{Z^\BQ}$ is the usual maximal compact subgroup \eqref{K_Z^BQ}, $K_G$ is the stabilizer of a lattice $\Lambda$ in $W$, and $K_H$ is the stabilizer of a lattice $\Lambda^\flat$ in $W^\flat$. We make the following assumptions on the field extensions $F/F_0/\BQ$:
\begin{altitemize}
\item Each finite place $v$ of $F_0$ which is ramified over $\BQ$ or of residue characteristic $2$ is split in $F$.
\item Each finite place $v$ of $F_0$ which ramifies in $F$ is of degree one over $\BQ$.
\end{altitemize}
We also make the following assumptions on the hermitian spaces $W$ and $W^\flat$. We distinguish the case when $n$ is odd from the case when $n$ is even, cf.\ Remark \ref{latticecond herm}.

When $n=2m+1$ is odd, we impose that 
\begin{altitemize}
\item $W$ is split at all finite places of $F_0$ which are inert in $F$; and
\item  $W^\flat$ is split at all finite places, which forces $m$ and $[F_0 : \BQ]$ to be odd, cf.\ Remark \ref{latticecond herm}. 
\end{altitemize}
Then we choose $\Lambda_v$ to be self-dual when $v$ is split or inert in $F$, and almost $\pi_v$-modular when $v$ is ramified in $F$. Furthermore, we choose $\Lambda_v^\flat$ to be self-dual when $v$ is inert in $F$ and $\pi_v$-modular when $v$ is ramified in $F$. Such lattices  exist, even when we impose that $\Lambda=\Lambda^\flat\oplus O_F u$ with $(u,u) \in O_{F_0}^\times$. With these definitions,  $M_K(\wtHG)$ has everywhere good reduction. 

When $n=2m$ is even, we impose that 
\begin{altitemize}
\item $W$ is split at all finite places of $F_0$, which again forces $m$ and $[F_0 : \BQ]$ to be odd, cf.\ Remark \ref{latticecond herm}; and
\item  $W^\flat$ is split at all finite places of $F_0$ which are inert in $F$.
\end{altitemize}
Now we choose $\Lambda_v$ to be self-dual when $v$ is split or inert in $F$, and $\pi_v$-modular when $v$ is ramified in $F$. Furthermore, we choose $\Lambda_v^\flat$ to be self-dual when $v$ is inert in $F$ and almost $\pi_v$-modular when $v$ is ramified in $F$. Such lattices  exist, even when we impose that there exists $u\in W$ with $(u, u) \in O_{F_0}^\times$ such that, for all inert finite places  $\Lambda_v=\Lambda_v^\flat\oplus O_{F,v} u$,  and such that, for all ramified places,  $\Lambda_v^\flat$ and $\Lambda_v$  are related by a chain of inclusions
$$
   \pi_{v} (\Lambda_{v}^\flat \oplus O_{F,v} u)^* \subset^1 \Lambda_{v} \subset^1 \Lambda_{v}^\flat \oplus O_{F,v} u,
$$
 cf.\ \eqref{type 2 lattice relation}.  With these definitions, $M_K(\wtHG)$ has everywhere good reduction. 

One may ask in this connection whether Conjecture \ref{conj AGGP 1} is non-empty for  $M_K(\wtHG)$ (with everywhere good reduction). By our expectations in Remark \ref{rem coh triv}(\ref{rct i}), this comes down to asking whether there are representations $\pi\in\Pi_{\rm disc}(\wt\HG)$ with $\pi_f^K\neq 0$ which contribute to the cohomology  $H^{2n-3}(\Sh_K(\wtHG), \BC)$. Chenevier \cite{Chen} has indicated to us a method of producing such $\pi$ for low values of $n$, when $F/F_0$ is everywhere unramified. The method should also apply when $F/F_0$ is ramified once one has a better understanding of the local Langlands correspondence for unitary groups in ramified cases. 
\end{remark}

\begin{remark}
Assume $n=2$. In this case, the Beilinson--Bloch pairing is defined unconditionally and coincides with the N\'eron--Tate height. Conjecture \ref{conj AGGP 2} is closely related to the Gross--Zagier formula in \cite{YZZ1}. It would be interesting to clarify this relation.
\end{remark}

\section{$L$-functions and the relative trace formula}\label{s:Lfcts}

In this section, we recall certain distributions on the group $G'= \Res_{F/F_0}(\GL_{n-1} \times \GL_n)$ that appear in the context of the relative trace formula. For test functions with some local hypotheses, we follow \cite[\S3.1]{Z12} and \cite[\S2.1--\S2.4]{Z14}. In general, our definition relies on the truncation of relative trace formulas of Zydor \cite{Zy}.  

On one hand, these distributions are related to $L$-functions via the Rankin--Selberg theory (for $\GL_{n-1}\times\GL_n$) of Jacquet, Piatetski-Shapiro, and Shalika \cite{JPSS}. On the other hand, they will serve as the analytic side in our conjectures on arithmetic intersection numbers formulated in the next section.

\subsection{The $L$-function}\label{L-fcn}
Let $\Pi=\Pi_1\boxtimes \Pi_2$ be a cuspidal automorphic representation of $G'(\BA_{F_0})$, where $\Pi_1, \Pi_2$ are automorphic representations of $
\GL_{n-1}(\BA_F)$ and   $
\GL_{n}(\BA_F)$ respectively.  Let $L\left(s,\Pi_1\boxtimes\Pi_2\right)$ be the Rankin--Selberg convolution $L$-function. This is an entire function in $s\in\BC$ and it satisfies a functional equation of the form
$$
L\left(s,\Pi_1\boxtimes\Pi_2\right)=\epsilon(s,\Pi_1\boxtimes\Pi_2)L\left(1-s,\Pi_1^\vee\boxtimes\Pi_2^\vee\right),
$$
cf.\ \cite{JPSS}. Here $\Pi_i^\vee$ denotes the contragredient representation of $\Pi_i$.

The  $L$-function $L\left(s,\Pi_1\boxtimes\Pi_2\right)$ is represented by an integral. Let $\varphi=\otimes_v\varphi_v\in \Pi=\otimes_v\Pi_v$ be a decomposable vector. Consider the integral
\begin{align}\lambda(\varphi,s):=
\int_{H_1'(F_0)\bs H_1'(\BA_{F_0})}\varphi(h)\lv\det(h)\rv_F^{s}\,dh,\quad s\in\BC.
\end{align}
Then by \cite{JPSS}, with the Haar measures specified in \cite[\S2]{Z14b}, we have a decomposition (cf.\ \cite[\S3.3]{Z14b})
$$
   \lambda(\varphi,s)
   = L\biggl(s+\frac{1}{2},\Pi_1\boxtimes\Pi_2\biggr)\prod_{v} \lambda_v(\varphi_v,s).
$$
Here the left-hand side is an entire function in $s\in \BC$, and the local factors $\lambda(\varphi_v,s)$ have the following properties. 
\begin{altenumerate}
\renewcommand{\theenumi}{\arabic{enumi}}
\item For every $\varphi_v\in\Pi_v$, the function $s\mapsto \lambda_v(\varphi_v,s)$ is entire, and there exists $\varphi_v^\circ$ such that $ \lambda_v(\varphi_v^\circ,s) \equiv 1$.
\item For any decomposable $\varphi=\otimes_v\varphi_v\in \Pi=\otimes_v\Pi_v$, we have  $\lambda_v(\varphi_v,s)\equiv 1$ for all but finitely many $v$.
\end{altenumerate}

It follows that if $L\left(\frac{1}{2},\Pi_1\boxtimes\Pi_2\right)=0$ (e.g., $\Pi_1$ and $\Pi_2$ are  self-dual, and  $\epsilon(1/2,\Pi_1\boxtimes\Pi_2)=-1$), then we have 
$$
   \frac{d}{ds}\Big|_{s=0}\int_{H_1'(F_0)\bs H_1'(\BA_{F_0})}\varphi(h)\lv\det(h)\rv_F^{s}\,dh=L'\biggl(\frac{1}{2},\Pi_1\boxtimes\Pi_2\biggr)\prod_{v} \lambda(\varphi_v,0).
$$

We note that, if $\Pi_1$ and $\Pi_2$ correspond to the Arthur parameters $\psi^{(n-1)}$ and $\psi^{(n)}$ in \eqref{A para}, then we may write the $L$-function in \eqref{RS} as 
$$
L(s,\pi,R)=L\left(s,\Pi_1\boxtimes\Pi_2\right).
$$

\subsection{The global distribution on $G'$}
We briefly recall the global distribution on $G'$ from \cite[\S3.1]{Z12}, \cite[\S2]{Z14} (the notation in loc.~cit.\ is slightly different). We denote by $A_{G'}$ and  $A_{H_2'}$ the maximal $F_0$-split tori in the centers of $G'$ and $H_2'$, respectively. We have (via the embedding $H_2'\incl G'$) an equality $A_{G'}=A_{H_2'}$ and both are isomorphic to $\BG_{m,F_0}\times \BG_{m,F_0}$.

Let $f'=\otimes_v f'_v\in \sH(G'(\BA_{F_0}))=C_c^\infty (G'(\BA_{F_0}))$ be a pure tensor. We consider the associated automorphic kernel function 
\begin{equation}\label{kerfct}
   K_{f'}(x,y) := \int_{A_{G'}(F_0)\backslash A_{G'}(\BA_{F_0})}\sum_{\gamma\in
G'(F_0)}f'(x^{-1}z\gamma y)\,dz,\quad x,y\in G'(\BA_{F_0}) .
\end{equation}
We define
\begin{equation}\label{globdist}
   J(f',s) := 
	   \int_{H'_1(F_0)\backslash H_1'(\BA_{F_0}) }\int_{A_{H'_2}(\BA_{F_0})H_2'(F_0)\backslash H_2'(\BA_{F_0})}   
		K_{f'}(h_1,h_2)\eta(h_2)\lv\det(h_1)\rv_F^{s}\,dh_1\,dh_2,
\end{equation}
where the quadratic character $\eta\colon H_2'(\BA_{F_0})\to\{\pm1\}$ is defined as follows: for $h_2=(x_{n-1},x_{n})\in H_2'(\BA_{F_0})$, with $ x_i\in \GL_i(\BA_{F_0})$,
$$
   \eta(h_2)=\eta_{F/F_0}\bigl(\det(x_{n-1})^{n}\det(x_{n})^{n-1})\bigr).
$$
Here $\eta_{F/F_0}\colon \BA_{F_0}^\times\to \{\pm1\}$ is the idele class character associated to the extension of global fields $F/F_0$.

The kernel function \eqref{kerfct} has a spectral decomposition. The contribution of a cuspidal automorphic representation $\Pi$ to the kernel function is given by
$$K_{\Pi,f'}(x,y)=\sum_{\varphi\in \mathrm{OB}(\Pi)}\bigl(\Pi(f')\varphi \bigr)(x)\,\ov{ \varphi(y)},
$$
where the sum runs over an orthonormal basis $\mathrm{OB}(\Pi)$ of $\Pi$. 
Correspondingly, a cuspidal automorphic representation $\Pi$ contributes to the global distribution $J(f',s)$ 
$$
   J_{\Pi}(f',s) := \sum_{\varphi\in \mathrm{OB}(\Pi)}\lambda\bigl(\Pi(f')\varphi ,s\bigr)\,\beta(\ov\varphi),
$$
where $\beta$ is the Flicker--Rallis period integral with respect to the subgroup $H_{2}'$, cf.\ \cite[\S3.2]{Z14b}. It follows from the endoscopic classification for unitary groups that the period $\beta$ does not vanish identically if and only if $\Pi$ is in the image of base change from (quasi-split) unitary groups associated to the quadratic extension $F/F_0$ (cf.\  \cite[Th.\ 2.5.4, Rem.\ 2.5.5]{Mok}, and note that the corresponding Asai $L$-function has a pole if and only if  the period $\beta$ does not vanish identically, cf.\ \cite[Rem.\ 8, p.\ 976]{Z14} and references therein). 

 We again use the Haar measures specified in \cite[\S2]{Z14b}.
\begin{proposition}\label{p:J Pi}
Let $\Pi=\Pi_1\boxtimes\Pi_2$ be cuspidal, and assume that it is the base change of an automorphic representation $\pi$ on (quasi-split) unitary groups. If $L(1/2,\Pi_1\boxtimes\Pi_2)=0$, then
$$
\frac{d}{ds}\Big|_{s=0} J_{\Pi}(f',s)=L(1,\eta)^2\frac{L'(1/2,\Pi_1\boxtimes\Pi_2)}{L(1,\pi,{\rm Ad})}\prod_{v} J_{\Pi_v}(f'_v),
$$
where $ J_{\Pi_v}(f'_v)$ is the local distribution defined in \cite[\S3, (3.31)]{Z14b}, and $L(1,\pi,{\rm Ad})$ is the adjoint $L$-function (cf.\ \cite[\S3.4]{Z14b},  \cite{Mok}, \cite{KMSW}).
\end{proposition}
\begin{proof}
It follows from \cite[Prop.\ 3.6]{Z14b} that
\begin{align}\label{J Pi s}
 J_{\Pi}(f',s)=L(1,\eta)^2\frac{L(s+1/2,\Pi_1\boxtimes\Pi_2)}{L(1,\pi,{\rm Ad})}\prod_{v} J_{\Pi_v}(f'_v,s).
\end{align}
Here the local distribution $J_{\Pi_v}(f'_v,s)$ is defined in an analogous way to that of its value at $s=0$ given by \cite[\S3, (3.31)]{Z14b},
\begin{align}\label{loc J Pi}
   J_{\Pi_v}(f'_v,s) = 
   \sum_{\{\varphi_v\}} \frac{\lambda_v^\nat\bigl(\Pi_v(f_v')\varphi_v,s\bigr)\ov{\beta^\nat}(\varphi_v)} 
       {\theta^\nat(\varphi_v,\varphi_v)},
\end{align}
where the linear functional $\varphi_v\mapsto\lambda_v^\nat(\varphi_v,s)$ on $\Pi_v$  is defined by \cite[\S3, (3.24)]{Z14b}, using the Whittaker model.  We refer to \cite[\S3]{Z14b} for the precise normalization of  measures and the linear functionals $\beta^\nat$ and $\theta^\nat$.
In particular, if $L(1/2,\Pi_1\boxtimes\Pi_2)=0$, then we have
$$
\frac{d}{ds}\Big|_{s=0} J_{\Pi}(f',s)=L(1,\eta)^2\frac{L'(1/2,\Pi_1\boxtimes\Pi_2)}{L(1,\pi,{\rm Ad})}\prod_{v} J_{\Pi_v}(f'_v,0).
$$Since $J_{\Pi_v}(f'_v,0)=J_{\Pi_v}(f'_v)$ by definition,  the proof is complete.
\end{proof}

\begin{proposition}\label{p:spec}
Let $f'=\otimes_v f'_v\in \sH(G'(\BA_{F_0}))=C_c^\infty (G'(\BA_{F_0}))$ be a pure tensor. Suppose that for a split place $v$ the function $f_v'$ has the property that, for every character $\chi_v$ of the center $Z_{G'}(F_{0,v})$ of $G'(F_{0,v})$, the function $f'_{v,\chi_v}:g\mapsto \int_{Z_{G'}(F_{0,v}) }f'_v(zg)\chi^{-1}_v(z)\,dz$ is the sum of matrix coefficients of supercuspidal representations.  Then the integral \eqref{globdist} converges absolutely and it decomposes as
\begin{equation*}
   J(f',s)=\sum_{\Pi} J_{\Pi}(f',s)\notag
      =\sum_{\Pi}L(1,\eta)^2\frac{L(s+1/2,\Pi_1\boxtimes\Pi_2)}{L(1,\pi,{\rm Ad})}\prod_{v} J_{\Pi_v}(f'_v,s) ,
\end{equation*}
where the sum runs through the set of cuspidal automorphic representations $\Pi=\Pi_1\boxtimes\Pi_2$ of $G'(\BA_{F_0})$ coming by base change from  automorphic representations $\pi$ of (quasi-split) unitary groups. Here the distribution $J_{\Pi_v}(f'_v,s) $ is defined by \eqref{loc J Pi}.
 \end{proposition}
\begin{proof}The spectral decomposition, i.e., the first equality, follows from the simple version of the relative trace formula in \cite[Th.\ 2.3]{Z14}. Note that in loc.\ cit., the test function $f'$ is required to be ``nice'' with respect to a character of the center $\chi$ of $G'(\BA_{F_0})$. However, the spectral decomposition only requires the existence of a place $v$ where the function $f'_{v,\chi_v}$ is a matrix coefficient of a super-cuspidal representation. Hence the result holds by the linearity of $J$ under the current assumption. Moreover, though only the case $s=0$ is stated in loc.\ cit., the same proof works for all $s\in\BC$. 
The second  equality follows from Proposition \ref{p:J Pi}.
\end{proof}

\begin{remark}
There are many test functions $f_v'$ with the property in the above statement. It suffices to construct such functions for $\GL_m(F)$ where $F$ is a $p$-adic local field. Let $\pi$ be a supercuspidal representation of  $\GL_m(F)$ and $\wt f$ a matrix coefficient. Let $ {\bf 1}_{[{\rm val}=0]}$ be the characteristic function of the set consisting of $g\in \GL_m(F)$ such that $\det(g)\in O_F^\times$. Then we claim that $f=\wt f\cdot {\bf 1}_{[{\rm val}=0]}$ has the desired property. To see this, let $\chi_0$ be the central character of $\pi$, and $\nu$ an unramified character of $F^\times$ of order (exactly) $m$. Then $f_{\chi_0}= \frac{1}{m}\sum_{i=0}^{m-1} \wt f\cdot  (\nu^i\circ\det)$ (we normalize the measure such that $\vol(O_F^\times)=1$). Now note that $f(zg)=\chi_{0}(z)f(g)$ for $z\in O_F^\times$. Hence $f_{\chi}=0$ unless $\chi\chi_0^{-1}$ is unramified. If $\chi\chi_0^{-1}$ is unramified, we may assume that $\chi\chi_0^{-1}=\xi^m$ for some unramified character $\xi$.   Then $f_{\chi}=\frac{1}{m}\sum_{i=0}^{m-1}\wt f\cdot ((\xi \nu^i)\circ\det)$, hence a sum of matrix coefficients of $\pi\otimes\xi \nu^i$, $0\leq i\leq m-1$.
\end{remark}

It follows that for $f'$ as in Proposition \ref{p:spec}, we have an expansion for the first derivative
\begin{equation}\label{delJ spec}
\begin{aligned}
   \frac{d}{ds}\Big|_{s=0}J(f', s)=&\sum_{\substack{\Pi\\ \epsilon(\Pi)=-1}}L(1,\eta)^2\frac{L'(1/2,\Pi_1\boxtimes\Pi_2)}{L(1,\pi,{\rm Ad})}\prod_{v} J_{\Pi_v}(f'_v,0)\\
   &\quad+\sum_{\substack{\Pi\\ \epsilon(\Pi)=1}}L(1,\eta)^2\frac{L(1/2,\Pi_1\boxtimes\Pi_2)}{L(1,\pi,{\rm Ad})}\cdot \frac{d}{ds}\Bigl(\prod_{v} J_{\Pi_v}(f'_v,s)\Bigr)\bigg|_{s=0}.
\end{aligned}
\end{equation}
Here $\epsilon(\Pi)=\epsilon(1/2,\Pi_1\boxtimes\Pi_2)$ is the global root number for the Rankin--Selberg convolution.

\begin{remark}We note that the contribution to the spectral decomposition  from non-cuspidal automorphic representations is more complicated and we will not touch on this topic in this paper. A full spectral decomposition to remove the local restriction on $f'_v$  in the last proposition is the work in progress by Chaudouard and Zydor, and a coarse spectral decomposition has been obtained by Zydor in \cite{Zy}. 
\end{remark}

\begin{definition}\label{regsuppG'} Let $\lambda$ be a place of $F_0$. A function $f'_\lambda\in C_c^\infty (G'(F_{0, \lambda}))$ \emph{has regular support} if $\supp(f_\lambda')\subset G'(F_{0,\lambda})_\rs$. A pure tensor $f'=\otimes_v f'_v\in \sH(G'(\BA_{F_0}))$ \emph{has regular support at} $\lambda$ if $f'_\lambda$ has regular support. 
\end{definition}
Let us assume that $f'$ has regular support at $\lambda$. Later we will assume that $\lambda$ is non-archimedean.
Then, by \cite[Lem.\ 3.2]{Z12} the integral \eqref{globdist} is absolutely convergent for all $s\in\BC$, and admits a decomposition into a finite sum for a given $f'$ (cf.\ the proof of loc.\ cit.),
\begin{align}
\label{eqn J dec}
J(f',s)=\sum_{\gamma\in G'(F_0)_\rs/H_{1,2}'(F_0)} \Orb(\gamma,f',s),
\end{align}
where each term is a product of local orbital integrals (\cite[(3.2)]{Z12}),
\begin{equation}
\Orb(\gamma,f',s)=\prod_{v}\Orb(\gamma,f'_v,s),
\end{equation}
where in turn
\begin{equation*}
   \Orb(\gamma, f_v', s) := 
	   \int_{H_{1,2}'(F_{0,v})} f_v'(h_1^{-1}\gamma h_2) \lv\det h_1\rv^s \eta(h_2)\, dh_1\, dh_2.
\end{equation*}
We set 
\begin{equation}\label{J(0)}
J(f'):=J(f', 0) .
\end{equation}
Then the decomposition \eqref{eqn J dec} specializes to
\[
   J(f')=\sum_{\gamma\in G'(F_0)_\rs/H_{1,2}'(F_0)} \Orb(\gamma,f'),
\]
where
\[
   \Orb(\gamma,f') := \Orb(\gamma,f',0).
\]
We introduce
\[
   J_v(f',s) := \sum_{\gamma\in G'(F_0)_\rs/H_{1,2}'(F_0)} \Orb(\gamma,f'_v,s)\cdot \prod_{u\neq v} \Orb(\gamma,f'_u).
\]
We set 
\begin{equation}\label{delJ}
\begin{aligned}
   \delJ(f')&:=\frac{d}{ds}\Big|_{s=0}J(f',s),\\
   \delJ_v(f')&:= \frac{d}{ds}\Big|_{s=0}J_v(f',s),\\
   \del(\gamma,f'_v) &:= \frac{d}{ds}\Big|_{s=0} \Orb(\gamma,f_v',s).
  	\end{aligned}
\end{equation} 
Note that
\begin{equation}\label{eqn J'v dec}
   \delJ_v(f')= \sum_{\gamma\in G'(F_0)_\rs/H_{1,2}'(F_0)} \del(\gamma,f'_v)\cdot \prod_{u\neq v} \Orb(\gamma,f'_u). 
\end{equation}
Then we may decompose 
\begin{align}\label{eqn J' dec}
   \delJ(f')=\sum_{v} \delJ_v(f').
\end{align}

Without the regularity assumption on $f'$, the integral \eqref{globdist} may diverge in general. For all $f'\in \sH(G'(\BA_{F_0}))$, the truncation process of Zydor \cite{Zy} allows us to define a meromorphic distribution $J(\cdot,s)$ (\cite[Th.\ 0.1]{Zy}) which is holomorphic away from $s=\pm 1$. This allows us to define \eqref{J(0)} and \eqref{delJ}. We will use these distributions to formulate our conjectures in the next section. Zydor also obtains a coarse decomposition of $J(\cdot,s)$ into a sum of global orbital integrals, although for a non-regular-semisimple orbit there is no natural decomposition into a product of local orbital integrals.

\subsection{Smooth transfer}
The notion of smooth transfer between functions on unitary groups and on linear groups or their symmetric spaces is based on the concept of matching, cf.\ \cite[\S2]{RSZ2} and \cite{Z12,Z14}.  Using the results of Section \ref{orbit matching ss}, we can transpose this concept to our situation at hand.

Our definitions below depend on the choice of the transfer factor $\omega$ and on the choices of Haar measures. For definiteness, we will always take the transfer factor from \cite[\S 2.4]{RSZ2} (this is a slight variant of \cite[\S 2.4]{Z14}), which works for all places $v$,
 \begin{equation}\label{tran fac} 
 \omega(\gamma)=\prod\nolimits_v\omega_v(\gamma_v),\quad \gamma=(\gamma_v)\in G'(\BA_{F_0}).
 \end{equation}
The transfer factor has the properties 
\begin{enumerate}
\item ($\eta$-invariance) For $h_1\in H_1'(\BA_{F_0})$ and $h_2\in  H_2'(\BA_{F_0})$, we have $\omega(h_1^{-1}\gamma h_2)=\eta(h_2)\omega(\gamma)$.
\item (product formula) For $\gamma\in G'(F_0)$ we have
\begin{equation}\label{prod formula} 
\prod\nolimits_v\omega_v(\gamma)=1 .
\end{equation}
\end{enumerate}

Now let $p$ be a rational prime. We fix Haar measures on $Z^\BQ(\BQ_p)$, $H(F_{0,v})$ and $G_W(F_{0,v})$ for $v \mid p$.  We choose the Haar  measures on $\wt H(\BQ_p)$ and $\wtHG(\BQ_p)$ compatible with the product decompositions $\wtHG(\BQ_p)=Z^{\BQ}(\BQ_p)\times \prod_{v \mid p}H(F_{0,v})$ and $
\wtHG(\BQ_p)=Z^{\BQ}(\BQ_p)\times \prod_{v \mid p}G_W(F_{0,v})$, induced by \eqref{proddec}. We choose the quotient measure on $(\wt H(\BQ_p)\times\wt H(\BQ_p)) /\Delta(Z^\BQ(\BQ_p))$. For the archimedean places, we also choose Haar measures satisfying similar compatibilities.

We define the orbital integral for $f_p \in C_c^\infty(\wtHG(\BQ_p))$ and $g\in  \wt H(\BQ_p) \bs \wtHG(\BQ_p)_\rs / \wt H(\BQ_p)$,
\begin{equation}\label{eq Orb wtH}
   \Orb(g,f_p) :=\int_{(\wt H(\BQ_p)\times\wt H(\BQ_p)) /\Delta(Z^\BQ(\BQ_p)) } f_p(h_1^{-1}g h_2)\, dh_1\, dh_2.
\end{equation}

\begin{definition}\label{transfer def}
A function $f_p \in C_c^\infty(\wtHG(\BQ_p))$ and a collection $(f_v') \in \prod_{v\mid p} C_c^\infty(G'(F_{0,v}))$  of functions are \emph{transfers} of each other if for any element $\gamma = (\gamma_v) \in \prod_{v\mid p} G'(F_{0, v})_\rs$, the following identity holds:
\[
   \omega(\gamma) \prod_{v\mid p} \Orb(\gamma_v,f_v') = 
	   \begin{cases}
			\Orb(g,f_p), &  \text{whenever $g$ matches $\gamma$};\\
			0,  &  \text{no $g\in \wtHG(\BQ_p)$ matches $\gamma$}.
		\end{cases}
\]
We make the same definition for a function $f_\infty \in C_c^\infty (\wtHG(\BR))$ and a collection of functions $(f_v') \in \prod_{v\mid \infty} C_c^\infty(G'(F_{0,v}))$. 
\end{definition}

Let us explain the relation to smooth transfer between functions in $C^\infty_c(G'(F_{0, v}))$ and functions in $C^\infty_c(G_W(F_{0, v}))$, as $W$ varies through the isomorphism classes of hermitian spaces of dimension $n$ over $F_v$. This definition is based on the concept of matching between elements of $G'(F_{0, v})_\rs$ and elements of $G_W(F_{0, v})_\rs$ (see\ \cite[\S2]{RSZ1} for non-archimedean places $v$ of $F_0$ that are non-split in $F$; the definition extends in an obvious way to the archimedean places and to the split non-archimedean places).

\begin{definition}\label{defcompdecG}
A function $f_p\in  C_c^\infty(\wtHG(\BQ_p))$ is \emph{completely decomposed} if it is of the form 
\begin{equation}\label{f_p dec}
   f_p = \phi_{p}\otimes\bigotimes_{v\mid p}f_v,
\end{equation}
where $\phi_p\in C_c^\infty(Z^{\BQ}(\BQ_p))$ and $f_v\in C_c^\infty(G_W(F_{0,v}))$. A pure tensor $f=\bigotimes_p f_p\in \sH(\wtHG(\BA_f))$ is \emph{completely decomposed} if $f_p$ is completely decomposed for all $p$. 
\end{definition}

Note that an arbitrary element in $ \sH(\wtHG(\BA_f))$ is a linear combination of completely decomposed pure tensors.
 
\begin{remark}\label{rem match}
Let $ f_p = \phi_{p}\otimes\bigotimes_{v\mid p}f_v$ be completely decomposed. Set $$
c(\phi_p):=\int_{Z^{\BQ}(\BQ_p)}\phi_p(z) \, dz.
$$
By Lemma \ref{matchintilde}, we have, for $g\in \wtHG(\BQ_p)$ corresponding to the collection $g_v\in (H\times G)(F_{0,v})$, 
$$ 
   \Orb(g,f_p) =c(\phi_p)\prod_{v\mid p}  \Orb(g_v,f_v),
   $$
  where 
  $$
   \Orb(g_v,f_v)=\int_{H(F_{0,v})\times H(F_{0,v})} f_v(h_1^{-1}g_v h_2)\,dh_1\,dh_2
  $$
  is the orbital integral in \cite[\S2]{RSZ2}.
If the orbital integrals of $f_p$ do not vanish identically, then $f_p$ and $(f_v')_{v\mid p}$ are transfers of each other in the sense of Definition \ref{transfer def} if and only if for some non-zero  constants $c_v$ such that $c(\phi_p)=\prod_{v\mid p}c_v$, the functions $f_v$ and $c_v f_v'$ are transfers of each other for each $v$ in the sense of \cite[\S2]{RSZ2}. 
\end{remark}

\begin{definition} Let $v$ be an archimedean place of $F_0$. A  function $f'_v\in C^\infty_c(G'(F_{0, v}))$ is a \emph{Gaussian test function} if it transfers to the constant function $\mathbf{1}$ on $G_{W_0}(F_{0, v})$, where $W_0$ denotes the negative-definite hermitian space, and transfers to the zero function on $G_W(F_{0, v})$ for any other hermitian space $W$ (i.e., in the terminology\footnote{But note that here $G_W$ is the product of two unitary groups.}  of \cite[Def.\ 3.5]{Z12}, $f'_v$ is pure of type $W_0$ and a transfer of  $\mathbf{1}_{G_{W_0}(F_{0, v})}$). 

A pure tensor $f'=\otimes_v f'_v\in \sH(G'(\BA_{F_0}))$ is a \emph{Gaussian test function} if the archimedean components $f'_v$ for $v\mid\infty$ are all (up to scalar factor) Gaussian test functions. 
\end{definition}

We have fixed a Haar measure on each (compact) $H_{W_0}(F_{0, v})$. We will assume that the volume of $H_{W_0}(F_{0, v})$ is one. Then $f_v'$ is a Gaussian test function if and only if  for all $\gamma\in G'(F_{0,v})_\rs$,
\begin{align}\label{eq Gaussian}
\omega_v(\gamma)\Orb(\gamma, f'_v)=\begin{cases}1,& \text{there exists $g \in G_{W_0}(F_{0,v})$ matching $\gamma$} ;\\
0,& \text{no $g\in G_{W_0}(F_{0,v})$ matches $\gamma$}.
\end{cases}
\end{align}

The existence of Gaussian test functions is still conjectural. A Gaussian test function does not have regular support, in the sense of Definition \ref{regsuppG'}. 

\begin{definition}
A pure tensor $f=\otimes_p f_p\in \sH(\wtHG(\BA_f))$ and a pure tensor $f'=\otimes_v f'_v\in \sH(G'(\BA_{F_0, f}))$ are \emph{smooth transfers} of each other if they are expressible in a way that $f_p$ and $(f_v')_{v\mid p}$ are transfers of each other for each prime $p$. Here we choose the product measures on $Z^\BQ(\BA_f), \wt \HG(\BA_f)$, and $G'(\BA_{F_0, f})$ (implicitly we require that our local choices are made such that the product is convergent); also, the adelic transfer factor \eqref{tran fac}  is simply the product of the local transfer factors.
\end{definition}

\begin{remark}
The existence of local smooth transfer is known for non-archimedean places \cite{Z14}; hence for any $f\in \sH(\wtHG(\BA_f))$ as above, there exists a smooth transfer $f'\in \sH(G'(\BA_{F_0, f}))$ as above, and conversely. 
\end{remark}

\begin{lemma}\label{analyticsplit}
Let $f'=\otimes_v f'_v\in \sH(G'(\BA_{F_0}))$ be a  Gaussian test function. Assume that $f'$ has regular support at some place $\lambda$ of $F_0$. Then for any place $v_0$ of $F_0$ split in $F$, 
\begin{align*}
 \delJ_{v_0}(f')=0.
\end{align*}
\end{lemma}
\begin{proof}
This is \cite[Prop.\ 3.6(ii)]{Z12}. Note that implicitly our test function $f'$ is pure of an incoherent type \cite[\S3.1]{Z12}, so that for every regular semisimple $\gamma$, the local orbital integral $\Orb(\gamma,f'_v,s)$ vanishes at $s=0$ for at least one non-split place $v$.
\end{proof}

\section{The conjectures for the arithmetic intersection pairing}\label{s:conjaip}

In this section we formulate  a conjectural formula for the Gillet--Soul\'e arithmetic intersection pairing for cycles on the integral models of $M_K(\wtHG)$ we introduced earlier. This formula uses the distributions introduced in Section \ref{s:Lfcts}.

Throughout this section we assume that the extension $F/F_0$ and the hermitian space $W$ are such that all places $v \in \CV_\AT^W$ (cf.\ \eqref{V_AT^W}) are of degree one over $\BQ$.

\subsection{The global conjecture, trivial level structure}
Let $\Lambda^\flat\subset W^\flat$ and $\Lambda\subset W$ be a pair of $O_F$-lattices related as in Section \ref{global mod prob trivial level}, and recall the models $\CM_{K_{\wt H}^\circ}(\wt H)$, $\CM_{K_{\wt G}^\circ}(\wt G)$, and $\CM_{K_{\wtHG}^\circ}(\wtHG)$ over $\Spec O_E$ defined in loc.\ cit.\ for the Shimura varieties of Section \ref{Shimura varieties}.   In the case $F_0=\BQ$  and $\CM_{{K_\wtHG^\circ}}(\wtHG)$ is non-compact, we are implicitly replacing $\CM_{{K_\wtHG^\circ}}(\wtHG)$ by its toroidal compactification. Then the model $\CM_{K_\wtHG^\circ}(\wtHG)$ is proper and flat over $\Spec O_E$.  Furthermore, it is regular provided that there are no places $\nu$ of $E$ for which $E_\nu$ is ramified over $\BQ_p$ and $(v_0,\Lambda_{v_0})$ is of AT type \eqref{almost self-dual type}, where $v_0$ denotes the place of $F_0$ induced by $\nu$ via $\varphi_0$. Throughout this section, we assume that there are no places $v_0$ for which $(v_0,\Lambda_{v_0})$ is of AT type  \eqref{n=2 self-dual type} (the justification for this assumption is given by Remark \ref{excl type4}). 

The compact open subgroup $K^\circ_{\wt H} \subset \wt H(\BA_f)$ contains ${K_\wtHG^\circ}\cap \wt H(\BA_f)$, with equality when there are no places $v$ of $F_0$ for which $(v,\Lambda_v)$ is of AT type \eqref{pi-mod type}.  In this case, there is a closed embedding
$$
   \CM_{K^\circ_{\wt H}}(\wt H) \inj \CM_{{K_\wtHG^\circ}}(\wtHG).
$$
Like in  \eqref{eqn zK}, we obtain a cycle (with $\BQ$-coefficients) $z_{K_{\wtHG}^\circ}=\vol(K^\circ_{\wt H})[\CM_{K^\circ_{\wt H}}(\wt H)]$ on $\CM_{K_\wtHG^\circ}(\wtHG)$. We denote by  the same symbol its class in the rational Chow group, 
\begin{align}
\label{eqn zK int}
   z_{K_{\wtHG}^\circ}\in  \Ch^{n-1}\bigl(\CM_{K_\wtHG^\circ}(\wtHG)\bigr)_\BQ .
\end{align}
In general, let $v_1,\dotsc,v_m \in \CV_\AT^W$ be the places for which $(v_i,\Lambda_{v_i})$ is of 
AT type \eqref{pi-mod type}.  We use the closed embedding  \eqref{global corresp diag HG} to define the cycle
$
   z_{K_{\wtHG}^\circ} = \vol(K_{\wt H}^{\circ \prime})[\CM_{K_{\wt H}^{\circ \prime}}(\wt H)] 
$ and its class in the Chow group.

\begin{remark}
The definitions above of the cycle class $  z_{K_{\wtHG}^\circ}$  use a Haar measure on $\wt H(\BA_f)$. We will always choose the product of the  measures  used to define the local orbital integrals \eqref{eq Orb wtH}.
\end{remark}

We denote by $\wh\Ch{}^{n-1}(\CM_{K_\wtHG^\circ}(\wtHG))$ the arithmetic Chow group. Elements are represented by pairs $(Z, g_Z)$, where $Z$ is a cycle and $g_Z$ is a Green's current (cf.\ \cite[\S3.3]{GS}). We are going to use the Gillet--Soul\'e  arithmetic intersection pairing, cf.\  \cite{GS}, 
\begin{equation*}\label{eqn GS Q}
   \sform_{\GS}\colon \wh\Ch{}^{n-1}\bigl(\CM_{K_\wtHG^\circ}(\wtHG)\bigr) \times \wh\Ch{}^{n-1}\bigl(\CM_{K_\wtHG^\circ}(\wtHG)\bigr)\to \BR .
\end{equation*}

We extend this from a symmetric pairing to a hermitian pairing on the corresponding $\BC$-vector space ($\BC$-linear combinations of $(Z, g_Z)$), 
\begin{equation}\label{eqn GS}
   \sform_{\GS}\colon \wh\Ch{}^{n-1}\bigl(\CM_{K_\wtHG^\circ}(\wtHG)\bigr)_{\BC}\times \wh\Ch{}^{n-1}\bigl(\CM_{K_\wtHG^\circ}(\wtHG)\bigr)_{\BC}\to \BC .
\end{equation}

We choose a Green's current $g_{z_{K_{\wtHG}^\circ} }$ of the cycle (with $\BQ$-coefficients) $z_{K_{\wtHG}^\circ}$  to get an element in the rational arithmetic Chow group,
\begin{equation}
   \wh{z}_{K_{\wtHG}^\circ}=
   (z_{K_{\wtHG}^\circ},g_{z_{K_{\wtHG}^\circ} })\in \wh\Ch{}^{n-1}\bigl(\CM_{K_\wtHG^\circ}(\wtHG)\bigr)_\BQ .
\end{equation}
The Green's current is not unique. We shall work  in the following with an arbitrary but fixed choice.

Let
\begin{equation}
\label{Hk trivial}
   \sH_{K_\wtHG^\circ}^{\spl,\Phi} \subset \sH_{K_\wtHG^\circ}= \sH\bigl(\wtHG(\BA_f),K_\wtHG^\circ\bigr)
\end{equation}
be the partial Hecke algebra spanned by completely decomposed pure tensors of the form $f=\otimes_p f_p\in \sH_{K_\wtHG^\circ}$, where $f_p=\phi_p\otimes\bigotimes_{v|p}f_v$, as in Definition \ref{defcompdecG}, with $\phi_p=\mathbf{1}_{K_{Z^\BQ, p}}$ for all $p$ and where $f_v=\mathbf{1}_{{K^\circ_{H\times G,v}}}$ unless $v\in \Sigma^{\spl, \Phi}$. Here $\Sigma^{\spl, \Phi}$ is as in \eqref{Sigma spl Phi}. We have 
$$
 \sH_{K_\wtHG^\circ}^{\spl,\Phi}\simeq  \bigotimes_{v\in \Sigma^{\spl, \Phi}}\sH_{\HG,K^\circ_{\HG,v}} .
$$
By Lemma \ref{CM matching lem}(\ref{deg 1 match}), we have $\Sigma^{\spl, \Phi}\supset \Sigma^{\deg=1}:=\{v\in\Sigma^{\spl} \mid\deg_\BQ v=1\}$ and hence the  algebra $\sH_{K_\wtHG^\circ}^{\spl,\Phi}$ is not too small.

For $f\in \sH_{K_\wtHG^\circ}^{\spl,\Phi} $, we introduce via \eqref{globalDrHecke} a Hecke correspondence, hence an induced endomorphism $\wh R(f)$ on the arithmetic Chow group $\wh\Ch{}^{n-1}(\CM_{K_\wtHG^\circ}(\wtHG))_\BC$, cf.\ \cite[5.2.1]{GS}. Using the arithmetic intersection pairing \eqref{eqn GS}, we define
\begin{equation}\label{defofInt}
\begin{aligned}
   \Int^\natural(f) &:= \bigl(\wh R(f)\wh z_{K_\wtHG^\circ}, \wh z_{K_\wtHG^\circ}\bigr)_{ \GS} ,\\
     \Int(f) &:= \frac{1}{\tau(Z^\BQ)[E:F]}\cdot\Int^\natural(f) .
\end{aligned}
\end{equation}
Here 
$$
   \tau(Z^\BQ) := \vol\bigl(Z^\BQ(\BA_f)/Z^\BQ(\BQ) \bigr) .
$$

\begin{conjecture}[Global conjecture, trivial level structure]\label{conj integraltrivlev}
Let $f=\otimes_p f_p\in  \sH_{K_\wtHG^\circ}^{\spl,\Phi}$, and let  $f'=\otimes_{v}f'_v\in \sH(G'(\BA_{F_0}))$ be a Gaussian test function such that  $\otimes_{v<\infty}f'_v$ is a smooth transfer of $f$.  Then
$$
\Int(f)=-\delJ(f')-J(f'_{\corr}),
$$
where $f'_{\corr}\in C_c^\infty(G'(\BA_{F_0}))$ is a correction function. Furthermore, we may choose $f'$ such that $f'_{\corr}=0.$
\end{conjecture}

\begin{remark}The notion of smooth transfer at each individual place $v$ depends on the choice of transfer factors, and of Haar measures on various groups. However, the validity of the conjecture does not depend on these choices (use   the product formula \eqref{prod formula}).
\end{remark}

This conjecture has the following drawback. Since we cannot impose any regular support assumptions  on functions in $ \sH_{K_\wtHG^\circ}^{\spl,\Phi}$, the left-hand side of the asserted equality may involve self-intersection numbers, and these are difficult to calculate explicitly.  Analogously, on the right-hand side, the  terms in  Zydor's truncation that are not regular-semisimple orbital integrals are more delicate. Nevertheless, assuming a spectral decomposition of $J(f',s)$ that generalizes the case of special test functions in Proposition \ref{p:spec} and \eqref{delJ spec}, Conjecture \ref{conj integraltrivlev} relates the intersection number $\Int(f)$ to the first derivative of $L$-functions in the Arithmetic Gan--Gross--Prasad conjecture \ref{conj AGGP 1} and \ref{conj AGGP 2}.

\subsection{The global conjecture, non-trivial level structure}\label{section global conj non-triv level}
In this subsection, we use the integral models of the Shimura varieties with deeper level structures depending on the choice of a function $\bm$ as in \eqref{bm}. Note that the models $\CM_{K_\wtHG^\bm}(\wtHG)$ are not regular in the fibers over places lying above the support of $\bm$ (in effect, we are taking here the product of two copies of the Drinfeld moduli scheme). Therefore, the Gillet--Soul\'e pairing \eqref{eqn GS} is not defined for them. However, under certain hypotheses that assure that our physical cycle $z_K$ and its physical transform under a Hecke correspondence do not intersect in the generic fiber, we can define a naive intersection number for them, by the usual derived tensor product formula.

Similarly to the case with trivial level structure, we obtain a cycle (with $\BQ$-coefficients) $z_{K_\wtHG^\bm} =\vol(K_{\wt H}^{\bm \prime} )[\CM_{K_{\wt H}^{\bm \prime}}(\wt H)]$ on $\CM_{K_\wtHG^{\bm}}(\wtHG)$, cf.\ \eqref{eqn zK int} and, again, we denote by  the same symbol its class in the Chow group. We choose a Green's current $g_{z_{K_{\wtHG}^\bm} }$ of the cycle (with $\BQ$-coefficients) $z_{K_{\wtHG}^\bm}$  to get an element in the arithmetic Chow group,
\begin{equation}
   \wh{z}_{K_{\wtHG}^\bm} = (z_{K_{\wtHG}^\bm},g_{z_{K_{\wtHG}^\bm}})\in \wh\Ch{}^{n-1}\bigl(\CM_{K_\wtHG^\bm}(\wtHG)\bigr)_\BQ .
\end{equation}

Let
\begin{equation}
\label{partial Hk}
   \sH_{{K_\wtHG^\bm}}^{\spl,\Phi} := \sH\bigl(\wtHG(\BA_f),{{K_\wtHG^\bm}}\bigr)^{\spl,\Phi} 
\end{equation}
be the partial Hecke algebra spanned by completely decomposed pure tensors of the form $f=\otimes_p f_p\in \sH_{K_\wtHG^\bm}$, where $f_p=\phi_p\otimes\bigotimes_{v|p}f_v$, as in Definition \ref{defcompdecG}, with $\phi_p=\mathbf{1}_{K_{Z^\BQ, p}}$ for all $p$, and  where $f_v=\mathbf{1}_{K^\circ_{H\times G,v}}$  unless $v\in \Sigma^{\spl, \Phi}$. We have 
$$
 \sH_{K_\wtHG^\bm}^{\spl,\Phi}\simeq  \bigotimes_{ v\in \Sigma^{\spl, \Phi}}\sH_{\HG,K^\bm_{\HG,v}} .
$$
By Lemma \ref{CM matching lem}(\ref{deg 1 match}), we have $\Sigma^{\spl, \Phi}\supset \Sigma^{S,\deg=1}:=\{v\in\Sigma^{\spl} \mid v\nmid S\ \text{and}\ \deg_\BQ v=1\}$,  
for any finite set $S$ of places of $\BQ$ such that the places of $F_0$ above $S$ contain the support of $\bm$. 

 To any $f \in \sH_{{K_\wtHG^\bm}}^{\spl,\Phi}$
we associate via \eqref{globalDrHecke} a Hecke correspondence on $\CM_{K_\wtHG^\bm}(\wtHG)$.

\begin{definition}\label{regularsupport} Let $\lambda$ be a non-archimedean place of $F_0$, of residue characteristic $\ell$. Let $f_\ell\in C_c^\infty(\wtHG(\BQ_\ell))$ be completely decomposed, i.e., $f_\ell = \phi_{\ell}\otimes\bigotimes_{v\mid \ell}f_v$, cf.\ Definition \ref{defcompdecG}. Then $f_\ell$ is said to \emph{have regular  support at  $\lambda$} if $\supp f_\lambda\subset G_W(F_{0,\lambda})_\rs$. If $f=\otimes_p f_p\in \sH_{{K_\wtHG^\bm}}^{\spl,\Phi}$ is a completely decomposed pure tensor, then $f$ \emph{has regular  support at   $\lambda$}  if $f_\ell$ has regular support at $\lambda$.
\end{definition}

\begin{theorem}\label{int generic}
Let $f=\otimes_p f_p\in \sH_{{K_\wtHG^\bm}}^{\spl,\Phi}$ be a completely decomposed pure tensor. Assume that  $f$ has  regular support at some place $\lambda$ of $F_0$. Then the following statements on the  support of the intersection of the  cycles $z_{K_\wtHG^\bm}$ and $R(f) z_{K_\wtHG^\bm}$ of $\CM_{K_\wtHG^\bm}(\wtHG)$ hold. 
\begin{altenumerate}
\item\label{int generic i} The support does not meet the generic fiber. 
\item\label{int generic ii} Let $\nu$ be a place of $E$ lying over a place of $F_0$ which splits in $F$. Then the  support does not meet the special fiber 
   $\CM_{K_\wtHG^\bm}(\wtHG)\otimes_{O_E} \kappa_{\nu}$. 
\item\label{int generic iii} Let $\nu$ be a place of $E$ lying over a place of $F_0$ which does not split in $F$. Then the support meets the special fiber 
   $\CM_{K_\wtHG^\bm}(\wtHG)\otimes_{O_E} \kappa_{\nu}$ only in its basic locus.\footnote{In this special case, the basic locus is characterized as follows. Let $\bigl( A_0,\iota_0,\lambda_0,A^\flat,\iota^\flat,\lambda^\flat,\ov\eta^\flat, \varphi^\flat, A,\iota,\lambda,\ov\eta, \varphi\bigr)$  correspond  to a point of $\CM_{K_\wtHG^\bm}(\wtHG)$ with values in  an algebraically closed extension of $\kappa(\nu)$. Consider the decomposition $A[p^\infty]=\prod_{w|p}A[w^\infty]$, resp.\ $A^\flat[p^\infty]=\prod_{w|p}A^\flat[w^\infty]$, of the  $p$-divisible group  of $A$, resp.\ $A^\flat$,  under the action of $O_F\otimes \BZ_p$; then all factors $A[w^\infty]$, resp.\ $A^\flat[w^\infty]$,  are isoclinic. }
\end{altenumerate}
\end{theorem}

\begin{remark}\label{rknoncompact}
Let $F_0 = \BQ$ and assume that $\CM_{K_\wtHG^\bm}(\wtHG)$ is non-compact.  Then  the   closure in the toroidal compactification \cite[\S 2]{Ho-kr2} of $\CM_{K_\wtHG^\bm}(\wtHG)$ of the support of the intersection of the  cycles $z_{K_\wtHG^\bm}$ and $R(f) z_{K_\wtHG^\bm}$  does not meet  the boundary. Indeed, this follows from Theorem \ref{int generic}(\ref{int generic iii}) because the basic locus of $\CM_{K_\wtHG^\bm}(\wtHG)\otimes_{O_E} \kappa_{\nu}$  does not meet the boundary. 
\end{remark}

The proof of Theorem \ref{int generic} will be based on the following lemma.

\begin{lemma}\label{no isog}
Let $k$ be an algebraically closed field which is an $O_E$-algebra. Let $(A_0, \iota_0)$ be an abelian variety with $O_F$-action with Kottwitz condition of signature $((0, 1)_{\varphi\in\Phi})$, cf.\ \eqref{KottwitzA0}, and let $(A, \iota)$ be  an abelian variety with $O_F$-action with Kottwitz condition of type $r$ as in Remark \ref{M remarks}(\ref{Kottwitzr}).  Assume there exists an $F$-linear isogeny
$$
A_0^n\to A .
$$
Then $k$ is of positive characteristic $p$. Let $\nu$ be the corresponding place of $E$. The place $v_0$ of $F_0$ induced by $\nu$ is non-split in $F$, and the isogeny classes of $A_0$ and $A$ only depend on the CM type $\Phi$.  If $v_0$ is the only place of $F_0$ above $p$, then $A_0$ and $A$ are supersingular. 
\end{lemma}
\begin{proof} We prove the first statement by contradiction. Assume that $k$ is of characteristic zero. Then an isogeny as above induces an $F\otimes k$-linear isomorphism
$$
\Lie (A_0)^n\isoarrow \Lie A .
$$
Such an isomorphism cannot exist due to the different Kottwitz conditions on $A_0^n$ and $A$. 

Now suppose that the place $v_0$ of $F_0$ induced by $\nu$ is split in $F$. Then there is a splitting  of the $p$-divisible group of $A_0^n$, resp.\ $A$, according to the two places of $F$ above $F_0$,
$$
A_0^n[p^\infty]=X_0^{(1)}\times X_0^{(2)}, \quad A[p^\infty]=X^{(1)}\times X^{(2)},
$$
An $F$-linear isogeny as above induces isogenies of $p$-divisible groups,
$$
X_0^{(1)}\to X^{(1)}, \quad  X_0^{(2)} \to X^{(2)} .
$$
However, by the Kottwitz conditions, the dimension of $X_0^{(1)}$, resp.\  $X_0^{(2)}$, is divisible by $n$, whereas the dimension of $X^{(1)}$, resp.\  $X^{(2)}$, is $\equiv \pm 1\bmod n$. Hence such an isogeny cannot exist, hence  $v_0$ is non-split in $F$.

  The rational Dieudonn\'e module $M_0$ of  $A_0[p^\infty]$ is a free $F\otimes W(k)$-module of rank one. Consider its decomposition according to the places $w$ of $F$ above $p$,
$$
M_0=\bigoplus_w M_{0, w} .
$$
Then each summand is isoclinic. More precisely, if $w$ lies above a non-split place of $F_0$ (such as the place $w_0$ induced by $\nu$), then $M_{0, w}$ is isoclinic of slope $1/2$; and if $\ov w\neq w$, then the slope of $M_{0, w}$ is equal to $a_w/d_w$, where
$$
a_w := \#\{\,\varphi\in \Phi \mid w_\varphi=w \,\}, \quad d_w := [F_w:\BQ_p] . 
$$
Here we chose an embedding of $\ov\BQ$ into $\ov\BQ_p$ so that any $\varphi\in\Phi$ induces a place $w_\varphi$ of $F$. 
\end{proof}

\begin{proof}[Proof of Theorem \ref{int generic}] (\ref{int generic i}): Suppose that
\[
   ( A_0,\iota_0,\lambda_0,A^\flat,\iota^\flat,\lambda^\flat,\ov\eta^\flat, \varphi^\flat, A,\iota,\lambda,\ov\eta,\varphi)\in \CM_{K^{\bf m}_{\wtHG}}(\wtHG)(k)
\]
is a point in the support, where $k$ is an algebraically closed field of characteristic zero. Then $(A, \iota,\lambda)=(A^\flat \times  A_0,\iota^\flat \times \iota_0,\lambda^\flat \times \lambda_0(u))$, and
there exists $g\in (H\times G)(\BA_{F_0,f})_\rs$ such that there is  an isogeny $$
 \phi\colon A= A^\flat\times A_0\to A
 $$
 which makes the diagram
   \begin{equation}\label{comm hecke}
	\begin{gathered}
   \xymatrix{
	   \wh V(A_0,  A)   \ar[d]_-{\phi} 	   \ar[r]^-{\eta} &   -W\otimes_F\BA_{F, f} \ar[d]^-g \\
	   \wh V(A_0, A)    \ar@{->}[r] ^-{\eta} &   -W\otimes_F\BA_{F, f}
	}
	\end{gathered}
\end{equation}
 commute. From the splitting $A= A^\flat\times A_0$, we also have  
$$
u\colon A_0\to A,
$$
which makes the diagram
  \begin{equation}
	\begin{gathered}
   \xymatrix{
	      \wh V(A_0,  A_0)   \ar[d]_-u 	   \ar[r]^-{\eta_0} &   \BA_{F, f} \ar[d]^-u \\
	 \wh V(A_0, A)    \ar@{->}[r] ^-{\eta} &   -W\otimes_F\BA_{F, f}
	}
	\end{gathered}
\end{equation}
commute, where $\eta_0$ is defined in the the obvious way. 
We consider the homomorphism   
$$
(\phi^i u)_{0\leq i\leq n-1}: A_0^n\to A
$$
whose $i$th component is $\phi^i u: A_0\to A$. We claim that this is an isogeny. It suffices to show that its induced map on any rational Tate module is an isomorphism.   This follows by the commutativity of the diagram \eqref{comm hecke} from the  regular semi-simplicity of $g$. This conclusion contradicts Lemma \ref{no isog}. 

(\ref{int generic ii}), (\ref{int generic iii}): Now let $(A_0,\iota_0,\lambda_0,A^\flat,\iota^\flat,\lambda^\flat,\ov\eta^\flat, \varphi^\flat, A,\iota,\lambda,\ov\eta, \varphi)\in \CM_{K^{\bf m}_{\wtHG}}(\wtHG)(k)$ be a point in the support, where $k$ is an algebraically closed field of positive characteristic $p$.  When there exists $g\in (H\times G)(\BA_{F_0,f}^p)_\rs$ such that there exists a commutative diagram \eqref{comm hecke} (with an upper index $^p$ added everywhere), the argument is as before, by reduction to Lemma \ref{no isog}. 

Now assume that the  place $v_0$ of $F_0$ induced by $\nu$ is split in $F$ and  that there exists $g\in G(F_{0, v_0})_\rs$ such that  there is  an isogeny $$
 \phi\colon A= A^\flat\times A_0\to A 
 $$
 which makes a  diagram analogous to \eqref{comm hecke} commute. To explain this diagram, we use a compatible system of Drinfeld level structures,  
$$
\wt \varphi\colon \Lambda_{w_0}[\pi_{w_0}^{-1}]/\Lambda_{w_0}\to T_{w_0}(A_0, A) ,
$$
which induces  for every $m$ a Drinfeld level  $ {w^m_0}$-structure on the $w_0^m$-torsion subgroup,
$$\varphi\colon \pi^{-m}_{w_0}\Lambda_{w_0}/\Lambda_{w_0}\to T_{w_0}(A_0, A) [w_0^m] .
$$ 
Then the analog of \eqref{comm hecke} is the commutative diagram for sufficiently large $m'$, 
 \begin{equation}\label{comm hecke-p}
	\begin{gathered}
   \xymatrix{
	      \Lambda_{w_0}[\pi_{w_0}^{-1}]/\Lambda_{w_0}\ar[d]_-{\pi^{m'}_{w_0}g}  	   \ar[r]^-{\wt\varphi} &  T_{w_0}(A_0, A)\ar[d]^-{\pi^{m'}_{w_0}\phi} \\
	  \Lambda_{w_0}[\pi_{w_0}^{-1}]/\Lambda_{w_0} \ar[r]^-{\wt\varphi} &   T_{w_0}(A_0, A)  .
	}
	\end{gathered}
\end{equation}
From the splitting $A= A^\flat\times A_0$, we also have  $$
u\colon A_0\to A,
$$
which makes the diagram
  \begin{equation}
	\begin{gathered}
   \xymatrix{
	 O_{F,w_0}[\pi_{w_0}^{-1}]/O_{F,w_0}   \ar[d]_-{u} 	   \ar[r]^-{\wt\varphi_0} &     T_{w_0}(A_0, A_0)  \ar[d]^-u \\
	 \Lambda_{w_0}[\pi_{w_0}^{-1}] /  \Lambda_{w_0}    \ar@{->}[r]^-{\wt\varphi} &  T_{w_0}(A_0, A)
	}
	\end{gathered}
\end{equation}
commute, where $\wt\varphi_0$ is the limit over $m$ of the homomorphisms \eqref{varphi0}. 
We consider the homomorphism   
$$
(\phi^i u)_{0\leq i\leq n-1}: A_0^n\to A
$$
whose $i$th component is $\phi^i u: A_0\to A$. We again claim that this is an isogeny, which would contradict Lemma \ref{no isog}. It suffices to show that its induced map on the $p$-divisible group is an isogeny.   This follows by the commutativity of diagram \eqref{comm hecke-p} from the  regular semi-simplicity of $g$. 
\end{proof}

Let us assume that $f\in \sH_{{K_\wtHG^\bm}}^{\spl,\Phi}$ is a completely decomposed pure tensor which has regular  support at some place $\lambda$ of $F_0$. Then by  Theorem \ref{int generic}\eqref{int generic i}, the generic fibers of the cycles $ z_{K_\wtHG^\bm}$ and $ R(f) z_{K_\wtHG^\bm}$ do not intersect, and we may  define
\begin{equation}\label{int for globwith}
 \begin{aligned}
  \Int_\nu^\natural (f)&:=\bigl\la\wh R(f)\wh z_{K_\wtHG^\bm}, \wh z_{K_\wtHG^\bm}\bigr\ra_{{\nu}} \log q_\nu ,\\
     \Int(f)&:=\frac{1}{\tau(Z^\BQ) [E:F]}\sum_\nu \Int_\nu^\natural (f).
   \end{aligned}
\end{equation}
Here the first quantity is defined for a non-archimedean place $\nu$ through the Euler--Poincar\'e characteristic of a derived tensor product on $\CM_{K_\wtHG^\bm}(\wtHG)\otimes_{O_E} O_{E, (\nu)}$, comp.~\cite[4.3.8(iv)]{GS}.  Note that the intersection numbers  are indeed  defined for a non-archimedean place because the intersection of the cycles $z_{K_\wtHG^\bm}$ and $R(f)z_{K_\wtHG^\bm}$ avoids all fibers of $\CM_{K_\wtHG^\bm}(\wtHG)$ over places $\nu$ lying over the support of $\bm$, as follows from Theorem \ref{int generic}(\ref{int generic ii}). Therefore the intersection of these cycles takes place in the regular locus of $\CM_{K_\wtHG^\bm}(\wtHG)$, cf.\ Theorem \ref{int generic},  and hence the Euler--Poincar\'e characteristic is finite. For an archimedean place $\nu$, the last quantity  is defined by the archimedean component of the arithmetic intersection theory and we have set  $\log q_\nu:=[E_\nu:\BR]=2$, cf.\ \cite{GS}.

\begin{conjecture}[Global conjecture, nontrivial level structure]\label{conj integralnontrivlev}
Let $f=\otimes_p f_p\in  \sH_{{K_\wtHG^\bm}}^{\spl,\Phi}$ be a completely decomposed pure tensor and let  $f'=\otimes_{v}f'_v\in \sH(G'(\BA))$ be a Gaussian test function such that  $\otimes_{v<\infty}f'_v$ is a smooth transfer of $f$. Assume that  $f$ has  regular support at some place $\lambda$ of $F_0$.  
Then
$$
\Int(f)=-\delJ(f')-J(f'_{\corr}),
$$where $f'_{\corr}\in C_c^\infty(G'(\BA))$ is a correction function. Furthermore,  $f'=\otimes_{v}f'_v$ may be chosen such that $f'$ has regular support at $\lambda$ and that $f'_{\corr}=0.$
\end{conjecture}

\begin{remark}
Part of the conjecture asserts that a change of the Green's current is compensated by a  change of the correction function $f'_{\corr}$.
\end{remark}

\begin{remark}
Similar to the case  of trivial level structure, when there exists a split place $v$ such that $f_v'$ satisfies the assumption in Proposition \ref{p:spec}, by \eqref{delJ spec},  Conjecture \ref{conj integralnontrivlev} relates the intersection number $\Int(f)$ to the first derivative of $L$-functions in the Arithmetic Gan--Gross--Prasad conjecture \ref{conj AGGP 1} and \ref{conj AGGP 2}. The hypothesis on the existence of such a split place $v$ could  be dropped once a full spectral decomposition of $J(f',s)$ for all test functions is available. 
\end{remark}

Note that here the right-hand side is well-defined (cf.\ Section \ref{L-fcn}). We also point out that if $f\neq 0$ has regular  support at some place $\lambda$ of $F_0$, then  $\lambda$ must be in $\supp \bm$ (in particular $\lambda$ is split in $F$). Since we will not need this statement, we omit the proof.

\begin{lemma}
\begin{altenumerate}
\item 
Let $u \in \supp \bm$ be a place above $p$ (in particular $u$ is split in $F$).  There exists a non-zero function $f_p \in \sH_{{K_\wtHG^\bm,p}}$ that has regular  support at  $u$.
\item
For  any $f_p\in \sH_{{K_\wtHG^\bm,p}}$ with regular  support at the place $u$ above $p$, there exists a transfer $(f'_v)_{v\mid p}$ such that $f'_{u}$ has regular support at $u$.
\end{altenumerate}\label{lem transfer split} 
\end{lemma}

\begin{proof}Let $m$ be a positive integer. At a place $u$ split in $F$, by choosing a basis of a self-dual lattice in $W^\flat\otimes F_{0,u}$, resp.\  $W\otimes F_{0,u}$,  we may identify $H(F_{0,u})=\GL_{n-1}(F_{0,u})$ and $G(F_{0,u})=\GL_{n}(F_{0,u})$, and $K^m_{H,u}$, resp.\  $K^m_{G,u}$,  with the principal congruence subgroups of level $m$. We may choose the basis of the lattice in $W\otimes F_{0,u}$ by adding the special vector $u$ to the basis of the lattice in $W^\flat\otimes F_{0,u}$. Therefore we may further assume that the embedding $H(F_{0,u})\incl G(F_{0,u})$ has the property that $K^m_{H,u}=K^m_{G,u}\cap H(F_{0,u})$. Moreover, the stabilizers of the lattices are identified with $\GL_{n-1}(O_{F_0,u})$ and  $\GL_{n}(O_{F_0,u})$ respectively, and $\GL_{n-1}\incl \GL_{n}$ sends $h$ to $\diag (h,1)$.

Now let $f_p =  \phi_{p}\otimes\bigotimes_{v\mid p}f_v$ be completely decomposed. It suffices to show that there exists a non-zero $f_u=f_{n-1,u}\otimes f_{n, u}\in \sH((H\times G)(F_{0, u}), K^m_{H, u}\times K^m_{G, u})$ with regular support. We construct such a function by first setting $f_{n-1,u}=\mathbf{1}_{K^m_{H,u}}$. Note that the pair of functions $( f_{n-1,u}, f_{n, u})$ defines  the function $\wt f_u\in  \sH(G(F_{0,u}), K^m_{G,u})$  by ``contraction" under the map 
\begin{equation}\label{contraction}
\begin{gathered}
   \xymatrix@R=0ex{
      H\times G \ar[r]  &  G\\
      (h,g) \ar@{|->}[r]  &  h^{-1}g,   
   }
\end{gathered}
\end{equation}
namely,
$$
\wt f_u(g)=\int_{H(F_{0,u})} f_{n-1,u}(h) f_{n,u}(hg)\,dh.
$$
We then have $\wt f_u=\vol(K^m_{H,u}) f_{n,u}.$ Then the function $f_u=f_{n-1,u}\otimes f_{n, u}$ has regular support (with respect to the $H\times H$-action) if (and only if) $\wt f_u$ or, equivalently, $f_{n,u}$ has regular support (with respect to the conjugation action of $H$). This holds because the inverse image of $G(F_{0,u})_{\rs}$ under the contraction map \eqref{contraction} is exactly $(H\times G)(F_{0,u})_{\rs}$. 

We now choose $f_{n,u}$ supported in $\GL_{n}(O_{F_0,u})$. Recall that $G(F_{0,u})_{\rs}$ is defined by the equation $\Delta\neq 0$ where $\Delta$ is a polynomial in the entries of $\GL_n$ with coefficients in $\BZ$, and it is easy to exhibit an element $g\in \GL_{n}(O_{F_0,u})$ such that $\Delta(g)\in\BZ^\times=\{\pm 1\}$\footnote{For example, $g=\left( \begin{array}{ccccc}
0&1& \cdots& 0& 0\\
0&0&\cdots& 1&0\\
0&0 &\cdots& 0&1\\
1&1 &\cdots&1& 0
 \end{array}\right)\in \GL_{n}(\BZ) $ satisfies $\Delta(g)\in\{\pm 1\}$.}.  Then the function $f_{n,u}=\mathbf{1}_{K^m_{G,u}g K^m_{G,u}}$ has regular support. Indeed, consider the reduction map $\GL_{n}(O_{F_0,u})\to \GL_{n}(O_{F_0,u}/\varpi_u^m)$ where $\varpi_u$ is a uniformizer. It is easy to see that for $k,k'\in 1+\varpi_u^m M_{n}(O_{F_0,u})$ we have $\Delta(kgk')\equiv \Delta(g)\equiv \pm 1\mod \varpi_u^m$. In particular we have $ K^m_{G,u}g K^m_{G,u}\subset \GL_{n}(O_{F_0,u})_\rs$. This completes the proof of the first part.

To show the second part, by a reduction process similar to the first part (cf.\ \cite[Prop.\ 2.5]{Z14}), it suffices to work with the inhomogeneous version, i.e., to show there exists  a smooth transfer $f'_u$ with support in $S_n(F_{0,u})_\rs$. We may identify $S_n(F_{0,u})=\GL_n(F_{0,u})$ and then the notions of regular semi-simplicity on $S_n(F_{0, u})$ and on $G(F_{0,u})$ coincide. This completes the proof.
\end{proof}

The left-hand side of \eqref{int for globwith} can be localized, i.e., we can write it as a sum over all non-archimedean places, 
\begin{equation}\label{locInt}
   \Int(f) = \sum_v \Int_v(f) ,
\end{equation}
where  
\begin{equation}\label{intloc}
   \Int_v(f) :=\frac{1}{\tau(Z^\BQ)\cdot  [E:F]} \sum_{\nu\mid v} \Int^\natural_\nu(f).
\end{equation}

By Lemma \ref{lem transfer split},  the smooth transfer  $f'$ of $f$ can be chosen such that $f'$ has regular support at $\lambda$, which we assume from now on. Then also the right-hand side of Conjecture \ref{conj integralnontrivlev} can be written as a sum of local contributions of each place of $F_0$, cf.\  \eqref{eqn J' dec} for $\delJ(f')$ and \eqref{J(0)} for $J(f'_{\corr})$.
\begin{proposition}\label{p:split I=J}
In the situation of Conjecture \ref{conj integralnontrivlev},
let $v_0$ be a place of $F_0$ that is split in $F$. Then
$$
\Int_{v_0}(f)=\delJ_{v_0}(f')=0.
$$
\end{proposition}
\begin{proof}
In Lemma \ref{analyticsplit} we have proved $\delJ_{v_0}(f')=0$. Now  $
\Int_{v_0}(f)=0$ follows from Theorem \ref{int generic}\eqref{int generic ii}.
\end{proof}

 In the next subsection, we are going to formulate a semi-global conjecture for each non-split place $v$ (including archimedean ones) which refines Conjecture \ref{conj integralnontrivlev}.

\subsection{The semi-global conjecture}
Let  $v_0$ be a place of $F_0$ above the place $p\leq \infty$ of $\BQ$.  By Proposition \ref{p:split I=J}, from now on we may and do assume that $v_0$ is non-split in $F$. 

Now assume that $v_0$ is non-archimedean. We assume that $v_0$ is either of hyperspecial level type or of AT parahoric level type, in the sense of Section \ref{section semi-global}. We also take up the notation of Sections \ref{subsec hyper} and \ref{subsec AT} and denote,  for any place $\nu$ of $E$ lying over $v_0$, by $\CM_{K_{\wt G}}(\wt G)_{(\nu)}$, $\CM_{K_{\wt H}}(\wt H)_{(\nu)}$, and $\CM_{K_{\wtHG}}(\wtHG)_{(\nu)}$ the corresponding semi-global moduli stacks over $\Spec O_{E,(\nu)}$. 
Let $K_\wtHG = K_\wtHG^p \times K_{\wtHG,p} \subset \wtHG(\BA_f)$ be as  in Section \ref{subsec hyper}, resp.\  Section \ref{subsec AT}. Let
\begin{equation}
\label{partial Hk p}
   \sH_{K_\wtHG}^p \subset   \sH_{{K_\wtHG}} = \sH\bigl(\wtHG(\BA_f),{{K_\wtHG}}\bigr)
\end{equation}
be the partial Hecke algebra spanned  by completely decomposed pure tensors of the form $f=\otimes_\ell f_\ell\in \sH_{K_\wtHG}$, where $f_\ell=\phi_\ell\otimes\bigotimes_{v|\ell}f_v$, as in Definition \ref{defcompdecG}, with $\phi_\ell=\mathbf{1}_{K_{Z^\BQ, \ell}}$ for all $\ell$, and where
$$
f_p=\mathbf{1}_{K_{\wtHG, p}}.$$
We note that this defines a bigger Hecke algebra than \eqref{partial Hk} when ${K_\wtHG}={K_\wtHG^\bm}$, 
\begin{equation}\label{HPhi Hp}
  \sH_{{K_\wtHG}}^p \supset \sH_{{K_\wtHG^\bm}}^{\spl,\Phi} .
 \end{equation}

Let $f=\otimes_{\ell} f_\ell\in \sH_{{K_\wtHG}}^p$ be completely decomposed with  regular support  at some place $\lambda$. We define as before in \eqref{intloc}
\begin{equation}\label{int for semi}
 \begin{aligned}
    \Int_\nu^\natural (f)&:=\bigl\la\wh R(f)\wh z_{K_\wtHG}, \wh z_{K_\wtHG}\bigr\ra_{{\nu}} \log q_\nu  ,\\
  \Int_{v_0}(f)&:=\frac{1}{\tau(Z^\BQ) [E:F]}\sum_{\nu|v_0} \Int_\nu^\natural (f) ,
   \end{aligned}
\end{equation}
where again the contribution of the place $\nu$  is defined through the Euler--Poincar\'e characteristic of a derived tensor product on $\CM_{K_\wtHG}(\wtHG)_{(\nu)}$. This extends the definition \eqref{intloc} to the bigger Hecke algebra $\sH_{{K_\wtHG}}^p$.

We proceed similarly for an archimedean place $v_0\in \Hom(F_0,\BR)$. Denote,  for any place $\nu$ of $E$ lying over $v_0$, by   $\CM_{K_{\wt G}}(\wt G)_{(\nu)}$, $\CM_{K_{\wt H}}(\wt H)_{(\nu)}$, and $\CM_{K_{\wtHG}}(\wtHG)_{(\nu)}$ the corresponding complex analytic spaces (in fact, orbifolds).   Note that the Green's current $g_{z_{K_\wtHG }}$ is the multi-set $g_{z_{K_\wtHG },\nu}$ indexed by $\nu\in \Hom(E,\BC)$. We define

\begin{equation}\label{intloc semiglob}
 \begin{aligned}
   \Int_\nu^\natural (f)&:=\bigl\la\wh R(f)\wh z_{K_\wtHG}, \wh z_{K_\wtHG}\bigr\ra_{{\nu}} \log q_\nu ,\\
    \Int_{v_0}(f)&:=\frac{1}{\tau(Z^\BQ) [E:F]}\sum_{\nu|v_0} \Int_\nu^\natural (f),
   \end{aligned}
\end{equation}
 where the first quantity is defined before Conjecture \ref{conj integralnontrivlev}.

A refinement of Conjecture \ref{conj integralnontrivlev} is now given by the following statement.

\begin{conjecture}[Semi-global conjecture]\label{conj semiglob} 
Fix a place $v_0$ of $F_0$ above a place $p\leq \infty$ of $\BQ$. Let $f=\otimes_{\ell} f_\ell\in \sH_{{K_\wtHG}}^p$ ($\sH_{{K_\wtHG}}$ if $p$ is archimedean) be completely decomposed, and let $f'=\otimes_{v}f'_v\in \sH(G'(\BA_{F_0}))$ be a Gaussian test function such that  $\otimes_{v < \infty}f'_v$ is a smooth transfer of $f$.
Assume that for some $\ell$ prime to $v_0$  and some place $\lambda$ above $\ell$, the function $f$ has regular support at $\lambda$ in the sense of Definition \ref{regularsupport} and that $f'$ has regular support at $\lambda$ in the sense of Definition \ref{regsuppG'}. 
\begin{altenumerate}
\item\label{conj hypersp} Assume that $v_0$ is non-archimedean of hyperspecial type, cf.\ Section \ref{subsec hyper}, and that $f_{v_0}'=\mathbf{1}_{G'(O_{F_0,v_0})}$. Then
$$
\Int_{v_0}(f)=-\delJ_{v_0}(f') . 
$$
\item\label{conj ATtype}  Assume that $v_0$ is archimedean, or non-archimedean of AT  type, cf.\ Section \ref{subsec AT}.  Then 
$$
\Int_{v_0}(f)=-\delJ_{v_0}(f')-J(f'_{\corr}[v_0]),
$$where $f'_{\corr}[v_0]=\otimes_{v}f'_{\corr, v}$, with $f'_{\corr, v}=f'_v$ for $v\neq v_0$,  is a correction function. Furthermore,   $f'$ may be  chosen such that $f'_{\corr}[v_0]$ is zero.
\end{altenumerate}
\end{conjecture}

\begin{theorem}
The semi-global conjecture Conjecture \ref{conj semiglob} for all  places $v_0$  implies the global conjecture Conjecture \ref{conj integralnontrivlev}.
\end{theorem}
\begin{proof}

By \eqref{eqn J' dec}, and \eqref{locInt},  it follows from Proposition \ref{p:split I=J} for split places $v_0$ of $F_0$ and the semi-global conjecture Conjecture \ref{conj semiglob} for non-split places $v_0$ that 
$$
\Int(f)= -  \delJ(f')-\sum_{\text{$v_0$ bad}}J(f'_{\corr}[v_0]) ,
$$
where the sum runs over a finite set of places $v_0$ in Conjecture \ref{conj semiglob}\eqref{conj ATtype}.  Here we note that by \eqref{HPhi Hp} we may apply the semi-global conjecture Conjecture \ref{conj semiglob} for the given test function $f$ in the global conjecture Conjecture \ref{conj integralnontrivlev}.
This proves Conjecture \ref{conj integralnontrivlev} by taking 
\[
   f'_\corr=\sum_{\text{$v_0$ bad}}f'_{\corr}[v_0]. \qedhere
\]
\end{proof}

In the direction towards the semi-global conjecture, there are the following results.

\begin{theorem} 
\label{th: hypersp}
Conjecture \ref{conj semiglob}(\ref{conj hypersp}) holds when $n\leq 3$. 
\end{theorem}

\begin{proof}
Let $p$ denote the residue characteristic of $v_0$, and as previously in the paper, let $\CV_p$ denote the set of places of $F_0$ lying above $p$.  It suffices to show that the result holds if we assume the Arithmetic Fundamental Lemma (AFL) conjecture, which is known under these circumstances; see \cite[Th.\ 5.5]{Z12}, comp.~also \cite{M-AFL}.

We imitate the proof of \cite[Th.\ 3.11]{Z12}. More precisely, we consider the non-archimedean uniformization along the basic locus,
\begin{equation}\label{eq unif1}
   \bigl(\CM_{(\nu)} \otimes_{O_{E,(\nu)}} O_{\breve E_\nu}\bigr)\sphat \,
	   = \wtHG\vphantom{G}'(\BQ) \Big\bs \Bigl[ \CN'\times \wtHG (\BA_f^p) /K_\wtHG^p \Bigr].
\end{equation}
Here the hat on the left-hand side denotes the completion along the basic locus in the special fiber of $\CM_{(\nu)}:=\CM_{K_\wtHG}(\wtHG)_{(\nu)}$. The group $\wtHG\vphantom{G}'$ is an inner twist of $\wtHG$. More precisely, the group $\wtHG\vphantom{G}'$ is associated to the pair of hermitian spaces $(W^{\prime\flat},W')$, where  $W^{\prime\flat}$ and  $W'$ are negative definite at all archimedean places, and  isomorphic to $W^\flat$, resp.\ $W$, locally at all non-archimedean places except at $v_0$. Furthermore, $\CN'$ is the  RZ space  relevant in this situation. Using Lemma \ref{notwistbanal} below,  can write
\begin{equation}\label{e:N' vs N}
   \CN'\simeq \bigl(Z^\BQ(\BQ_p)/K_{Z^\BQ, p}\bigr)\times\CN_{O_{\breve E_\nu}} \times \prod_{v\in \CV_p\ssm\{v_0\}}(H\times G)(F_{0,v})/(K_{H,v}\times K_{G,v}) . 
\end{equation}
Here $\CN_{O_{\breve E_\nu}}=\CN \mathbin{\wh\otimes}_{O_{\breve F_{w_0}}}O_{\breve E_\nu} $, where  $\CN = \CN_{n-1} \times_{O_{\breve F_{w_0}}} \CN_n$ is the \emph{relative} RZ space of \cite{Z12}.   More precisely, the formal scheme  in the uniformization theorem is  the RZ space of polarized $p$-divisible groups  with action by $O_{F,w_0}$ satisfying the Kottwitz condition  \eqref{kottwitzF} of signature $((1, n-1)_{\varphi_0}, (0, n)_{\varphi\in\Phi_{v_0}\ssm\{\varphi_0\}})$. That this  coincides with the relative RZ space of \cite[\S2.2]{Z12} follows from   \cite[Th.\ 4.1]{M-Th}, where in the notation of loc.~cit.\ we take $E_0=F_{0,v_0}$, $E=F_{w_0}$, $(r, s)=(1, n-1)$, and $K=\BQ_p$.

Therefore we may rewrite \eqref{eq unif1} as 
\begin{equation}\label{eq unif2}
   \bigl(\CM_{(\nu)} \otimes_{O_{E,(\nu)}} O_{\breve E_\nu}\bigr)\sphat \,
	   = \wtHG\vphantom{G}'(\BQ) \Big\bs \Bigl[  \CN_{O_{\breve E_\nu}} \times \wtHG (\BA_f^{v_0}) /K_\wtHG^{v_0} \Bigr].
\end{equation}
Here we denote for simplicity (even though $\wtHG$ is not an algebraic group over $F_0$)
$$
 \wtHG (\BA_f^{v_0})/K_\wtHG^{v_0} = 
    \wtHG (\BA_f^{p})/K_\wtHG^{p}\times \bigl(Z^\BQ(\BQ_p)/K_{Z^\BQ, p}\bigr)\times \prod_{v\in \CV_p\ssm\{v_0\}}(H\times G)(F_{0,v})/(K_{H,v}\times K_{G,v}).
$$
There is also a similar uniformization of the basic locus of $\CM_{K_{\wt H}}(\wt H)_{(\nu)}$ involving the twisted form $\wt H'$ of $H$.

By Theorem \ref{int generic}\eqref{int generic iii}, the intersection of the cycles is supported in the basic locus, and hence we can imitate the proof of \cite[Th.\ 3.9]{Z12}. The difference is that here we have an extra central subgroup $Z^\BQ$. By the same procedure as in  loc.\ cit., we see that \eqref{intloc semiglob}  can be written as a sum 
\begin{equation}\label{sum}
	\Int_{v_0}(f) = 
      \frac 1{[E:F]} \sum_{g \in \BO(\wtHG(\BQ))_\rs}  \Orb(g, f^{p})\cdot 
         \biggl(c(\phi_p) \prod_{v\in \CV_p\ssm\{v_0\}} \Orb(g,f_v)\biggr) \cdot  \Int_{v_0}^\natural (g),
\end{equation}
cf.\  Remark \ref{rem match}. Here the volume factor $\tau(Z^\BQ)=\vol(Z^\BQ(\BA_f)/Z^\BQ(\BQ))$ is canceled with the one in the definition of $\Int_{v_0}(f)$, and 
$$
 \Orb(g, f^{p})=\prod_{\ell \nmid p}  \Orb(g, f_\ell).
$$
Also,  we have set 
$$ 
   \BO\bigl(\wtHG(\BQ)\bigr)_\rs := \wt H'(\BQ) \bs \wtHG\vphantom{G}'(\BQ)_\rs / \wt H'(\BQ) , 
$$ 
and, for $g\in (H\times G)(F_{0, v_0})$, in analogy with \eqref{int for semi}, 
\begin{equation}
\begin{aligned}
   \Int^\natural_{v_0}(g)&:=\sum_{\nu|v_0}\Int^\natural_\nu(g),\\
	   \Int^\natural_\nu(g)&:=\bigl\la \Delta(\CN_{n-1,O_{\breve E_\nu}}), g\Delta(\CN_{n-1,O_{\breve E_\nu}}) \bigr\ra_{\CN_{O_{\breve E_\nu}}} \log q_\nu .
\end{aligned}
\end{equation}
Now note that
\begin{align}   \label{eq:int nat}
\sum_{\nu \mid v_0} \Int^\natural_\nu(g)&= \sum_{\nu \mid v_0} \bigl\la \Delta(\CN_{n-1,O_{\breve E_\nu}}), g\Delta(\CN_{n-1,O_{\breve E_\nu}}) \bigr\ra_{\CN_{O_{\breve E_\nu}}} \log q_\nu \notag\\
		&=[E:F]\bigl\la \Delta(\CN_{n-1}), g\Delta(\CN_{n-1}) \bigr\ra_\CN \log q_{w_0}.
\end{align}
Here we use the equality $\sum_{\nu\mid w_0} e_{\nu/w_0} f_{\nu/w_0}=\sum _{\nu\mid w_0} d_{\nu/w_0}=[E:F]$. By the AFL identity (cf.\ \cite[\S4]{RSZ2}), we have
\begin{align}\label{eq AFL hom}
 2\bigl\la \Delta(\CN_{n-1}), g\Delta(\CN_{n-1}) \bigr\ra_\CN \log q_{v_0}=-\omega_{v_0}(\gamma)\del(\gamma,f_{v_0}'),
\end{align}
for any $\gamma\in G'(F_{0,v_0})_\rs$ matching $g$. Since $v_0$ is inert in $F$, we have $\log q_{w_0}=2 \log q_{v_0}$, and hence 
\begin{align}\label{eq AFL v0}
  \Int^\natural_{v_0}(g)=- [E:F]\omega_{v_0}(\gamma)\del(\gamma,f_{v_0}').
\end{align}
Since $f$ and $f'$ are smooth transfers of each other, we have for $\gamma\in G'(F_{0})$ matching $g\in  \BO(\wtHG(\BQ))_\rs$,
\begin{align}\label{eq p infty}
\Orb(g, f^{p})=\prod_{v<\infty,\, v\nmid p }\omega_v(\gamma)\Orb(\gamma,f_v').
\end{align}

By Remark \ref{rem match}, since the orbital integrals of $f_{v_0}$ and $f'_{v_0}$ do not vanish identically, one of the following holds.
\begin{enumerate}
\item $c(\phi_p)=0$ or one of $f_v$, for $v\in \CV_p\ssm\{v_0\}$, has identically vanishing orbital integrals.
\item the orbital integrals of $f_p$ do not vanish identically. 
\end{enumerate}
In the first case,  one of $f'_v$, for $v\in \CV_p\ssm\{v_0\}$,  has identically vanishing orbital integrals, and we have $$
\Int_{v_0}(f)=0=\delJ_{v_0}(f'),
$$
where the second equality follows from \eqref{eqn J'v dec}.

Therefore we only need to consider the second case. By Remark \ref{rem match}, there are non-zero constants $c_v$ such that $\prod_{v \in \CV_p}c_v=c(\phi_p)$, and for every $v \in \CV_p$, $c_v f_v'$ is a transfer of $f_v$. By computing orbital integrals at some special $g$ and $\gamma$ (e.g., \cite{Y}) we  conclude that $c_{v_0}=1$ (this follows in any case if the FL conjecture holds, which is known when the residue characteristic is big enough). It follows that 
\begin{align}\label{eq p-v0}
c(\phi_p) \prod_{v\in \CV_p\ssm\{v_0\}}\Orb(g,f_v)=\prod_{v\in \CV_p\ssm\{v_0\}}\omega_v(\gamma)\Orb(\gamma,f'_v).
\end{align}

Note that for  $v$ archimedean, $f_v'$ is a Gaussian test function so that we have for regular semisimple $\gamma\in G'(F_{0,v})_\rs$ (cf.\ \eqref{eq Gaussian})
$$
\omega_v(\gamma)\Orb(\gamma, f'_v)=\begin{cases}1,&  \text{there exists $g\in (H\times G)(F_{0,v})$ matching $\gamma$};\\
0,& \text{no $g\in (H\times G)(F_{0,v})$ matches $\gamma$}.
\end{cases}
$$
By the last equality for archimedean $v$, by \eqref{eq p infty} for non-archimedean $v$ with $v\nmid p$, by \eqref{eq p-v0} for $v\in \CV_p\ssm \{v_0\}$, and by \eqref{eq AFL v0} for $v=v_0$, we have (cf.\ \eqref{sum})
 $$
 \Int_{v_0}(f)=-\sum_{\gamma\in G'(F_0)_\rs/H_{1,2}'(F_0)} \omega_{v_0}(\gamma)\del(\gamma,f'_{v_0})\cdot  \prod_{v\neq v_0}  \omega_{v}(\gamma)\Orb(\gamma,f'_v).
 $$
 Here the sum runs  a priori only over those $\gamma$ which match some $g\in G_W(F_0)$. However, those $\gamma$ which do not match any $g\in G_W(F_0)$ have vanishing orbital integrals away from $v_0$, cf.\ \cite[Prop.\ 3.6, eq.\ (3.4)]{Z12}. 

By the product formula for transfer factors \eqref{prod formula}, we have $\prod_{v}\omega_{v}(\gamma) =1$, and hence
 $$
 \Int_{v_0}(f)=-\sum_{\gamma\in G'(F_0)_\rs/H_{1,2}'(F_0)} \del(\gamma,f'_{v_0})\cdot  \prod_{v\neq v_0}\Orb(\gamma,f'_v).
 $$
By \eqref{eqn J'v dec},
the right-hand side equals 
 $-\delJ_{v_0}(f'),$ and this completes the proof.
\end{proof}

In the preceding proof, we used the following lemma.

\begin{lemma}\label{notwistbanal}
Let $L_0/\BQ_p$ be a finite extension and $L/L_0$ an \'etale $L_0$-algebra of rank $2$. Assume $p\neq 2$ if $L$ is a field.  Let $\Phi_L \subset \Hom_{\BQ_p}( L, \ov\BQ_p)$ be a local CM type for $L/L_0$, and let $r_L\colon \Hom_{\BQ_p}(L, \ov\BQ_p)\to \{0, n\}$ be a \emph{banal} generalized CM type of rank $n$. Let $E'$ be the join of the reflex fields of $\Phi_L$ and $r_L$. Let $\ov k$ be an algebraic closure of the residue field of $O_{E'}$. Let $(\BX_0, \iota_{\BX_0}, \lambda_{\BX_0})$ be a local CM triple of type $\ov\Phi_L$ over $\ov k$, where  $\lambda_{\BX_0}$ is principal.  Let $(\BX, \iota_{\BX}, \lambda_{\BX})$ be a local CM triple of type $r$ over $\ov k$ which satisfies the Eisenstein condition, cf.\ \eqref{eis cond lm}. Assume that $\ker \lambda_{\BX}\subset \BX[\pi]$, where $\pi$ denotes a uniformizer\footnote{ If $L=L_0\oplus L_0$, this means, as elsewhere in the paper,  an ordered pair of uniformizers in the usual sense.} of $L$. Let $\CN_{\ov\Phi_L, r_L}$ be the formal scheme  over $\Spf O_{\breve E'}$ that represents the functor which associates to $S\in {\rm Nilp}_{O_{\breve E'}}$ the set of isomorphism classes of tuples $(X_0, \iota_0, \lambda_0, X, \iota, \lambda, \varrho_0, \varrho)$, where $(X_0, \iota_0, \lambda_0)$ is a local CM triple of type $\ov\Phi_L$, and $(X, \iota, \lambda)$,  is a local CM triple of type $r_L$,  over $S$, and where 
$$
\rho_0\colon (X_0, \iota_0, \lambda_0)\times_S \ov S\to \BX_0\times_{\Spec \ov k} \ov S, \quad \rho\colon (X, \iota, \lambda)\times_S \ov S\to \BX\times_{\Spec \ov k} \ov S
$$
are quasi-isogenies respecting the $O_L$-actions and the  polarizations.  Here $(X, \iota, \lambda)$ is supposed to satisfy the Eisenstein condition. Then 
$$
\CN_{\ov\Phi_L, r_L}\simeq G(L_0)/K ,
$$
where $G$ is the unitary group of an  $L/L_0$-hermitian vector space  of invariant $\inv^{r_L}(\BX_0, \BX)$, and where $K$ is the stabilizer of a vertex lattice of type $t := \log_q \lv\ker \lambda_\BX\rv$, $q := \# O_L/\pi O_L$.

\textup{Here the invariant  $\inv^{r_L}(\BX_0, \BX)$ is defined in Remark \ref{rm:pdivgps}.}
\end{lemma}

\begin{proof} By Lemmas \ref{banal LM GL_n triv} and \ref{banal LM GU_n triv}, $\CN_{\ov\Phi_L, r_L}$ is formally \'etale over $\Spf O_{\breve E'}$. So it remains to determine the point set $\CN_{\ov\Phi_L, r_L}(\ov k)$ or, equivalently, the point set $\CN_{\ov\Phi_L, r_L}^{\rm rig}(\breve E')$, where $\CN_{\ov\Phi_L, r_L}^{\rm rig}$ denotes the generic fiber. Consider the crystalline period map $\pi\colon \CN_{\ov\Phi_L, r_L}^{\rm rig}\to { \CF}\otimes_{E'}\breve E'$, cf.\ \cite{RZ}. However, by the banality of $r_L$, the Grassmannian $\CF$ consists of a single point. Furthermore, the fiber over this point is identified with $G(L_0)/K$, since $\CN_{\ov\Phi_L, r_L}^{\rm rig}$ corresponds in the RZ tower to the level $K$.
\end{proof}

\begin{remark} We have considered local CM triples of type $\ov\Phi_L$ in the lemma because these are what naturally arise from the Kottwitz condition \eqref{KottwitzA0} in the moduli problem for $\CM_0$.  Of course, one could just as well consider the moduli space $\CN_{\Phi_L,r_L}$; then one replaces $\Phi_L$ with $\ov\Phi_L$ in the definition of $\inv^{r_L}(\BX_0,\BX)$, cf.\ Remark \ref{tired}.
\end{remark}

\begin{remark}\label{JandG}
Consider the $J$-group in the sense of Kottwitz  \cite{RZ} associated to the situation of Lemma \ref{notwistbanal}, when  $L$ is a field. Let $M(\BX_0)$, resp.\ $M(\BX)$,  be the rational Dieudonn\'e module of $\BX_0$, resp.\ $\BX$. Let $M=\Hom_L(M(\BX_0), M(\BX))$. Then $M$ is a rational Dieudonn\'e module free of rank $n$ over $L\otimes_{\BZ_p} W(\ov k)$ which, by our assumption, is isoclinic of slope $0$. It is equipped with a hermitian form $h \colon M\times M\to L\otimes_{\BZ_p} W(\ov k)$. Let $C$ be the space of Frobenius invariants in $M$. Then the restriction of $h$ makes  $C$ into an $L/L_0$-hermitian vector space of dimension $n$. The $J$-group is the unitary group $J=\U(C)$. We claim that $J\simeq G$. Indeed, the difference between $\inv^{r_L}(\BX_0, \BX)$ and $\inv(C)$ is equal to $\sgn(r_L)$, where 
$$
\sgn(r_L)=(-1)^{\sum_{\varphi \in \Phi_L} (r_L)_\varphi} , 
$$
cf.\ \eqref{sgn r}, \eqref{locinveq}.  Since $r_L$ is banal, the exponent is a multiple of $n$. Therefore the assertion follows in the case when $n$ is even. When $n$ is odd, any two unitary groups of $L/L_0$-hermitian vector spaces of dimension $n$ are isomorphic. 
\end{remark}

\begin{theorem}
Conjecture \ref{conj semiglob}(\ref{conj ATtype}) holds when $n\leq 3$ and $v_0$ is non-archimedean. 
\end{theorem}

\begin{proof} This follows in the same way as in the proof of the  previous Theorem \ref{th: hypersp}, from the AT conjecture proved for $n\leq 3$ in  \cite{RSZ1} and \cite{RSZ2}. 

We indicate where we need to modify the proof.  To simplify, we only prove the ``Furthermore" part of Conjecture \ref{conj semiglob}(\ref{conj ATtype}).
We first consider the ramified case in Section \ref{subsec AT}. In \eqref{e:N' vs N}, we have
$\CN_{O_{\breve E_\nu}}=\CN \mathbin{\wh\otimes}_{O_{\breve F_{w_0}}}O_{\breve E_\nu} $, where  $\CN = \CN_{n-1} \times_{O_{\breve F_{w_0}}} \CN_n$ is the relative RZ space in \cite{RSZ1,RSZ2}. Then we replace the AFL identity \eqref{eq AFL hom} by the AT identity (cf.\ \cite[\S 5]{RSZ1} when $n$ is odd, and \cite[\S 12]{RSZ2} when $n$ is even), 
\begin{align}\label{eq ATC hom}
 \bigl\la \Delta(\CN_{n-1}), g\Delta(\CN_{n-1}) \bigr\ra_\CN \log q_{v_0}=-\omega_{v_0}(\gamma)\del(\gamma,f_{v_0}'),
\end{align}
where $f_{v_0}'$ is a function in the AT conjecture.
Now since $v_0$ is ramified in $F$, we have $\log q_{w_0}=\log q_{v_0}$ and hence the equation  \eqref{eq AFL v0} remains true by \eqref{eq:int nat}.

In the case of unramified AT type, we replace the corresponding space $\CN$ after \eqref{e:N' vs N} by the RZ space in \cite{RSZ2}.

In both cases, the rest of the proof is the same.
\end{proof}

\begin{remark}The proof in the ramified case explains the discrepancy of  
the factor $2$ in the ramified ATC \cite{RSZ1,RSZ2} and the AFL. In these identities, it would be most natural to normalize the intersection number by the factor $\log q_{w_0}$.
\end{remark}

\appendix
\section{Sign invariants}\label{sign invariants}
In this appendix we adapt the sign invariants of \cite{KR-new} to the setting of the moduli problems introduced in Sections \ref{section semi-global} and \ref{section global}; see also \cite{KRZ}.  We continue with the notation in the main body of the paper, with $F/\BQ$ a CM field, $F_0$ its maximal totally real subfield of index $2$, and $\Phi$ a CM type for $F$.  However, we will allow for more general functions $r$ than in \eqref{r main body}. Set $d := [F_0 : \BQ]$, and let $v$ be a finite place of $F_0$ which is non-split in $F$.  (In the case $v$ splits in $F$, the analog of the theory in this appendix is trivial, insofar as the value group $F_{0,v}^\times/ \Nm F_v^\times$ is trivial.)  Let $k$ be an arbitrary field.  We are first going to define an invariant 
\begin{equation}
   \inv_v(A_0,\iota_0,\lambda_0,A,\iota,\lambda)^\nat \in F_{0,v}^\times / \Nm F_v^\times
\end{equation}
attached to the following objects over $k$:
\begin{altitemize}
\item abelian varieties $A_0$ and $A$ over $k$ of  respective dimensions $d$ and $nd$;
\item rational actions $\iota_0\colon F \to \End^\circ(A_0)$ and $\iota \colon F \to \End^\circ(A)$; and
\item quasi-polarizations $\lambda_0 \in \Hom^\circ(A_0,A_0^\vee)$ and $\lambda \in \Hom^\circ(A,A^\vee)$ whose corresponding Rosati involutions induce the non-trivial automorphism on $F/F_0$.
\end{altitemize}
We will then write
\begin{equation}
   \inv_v(A_0,\iota_0,\lambda_0,A,\iota,\lambda) \in \{\pm 1\}
\end{equation}
for the image of $\inv_v(A_0,\iota_0,\lambda_0,A,\iota,\lambda)^\nat$ under the identification $F_{0,v}^\times / \Nm F_v^\times \cong \{\pm 1\}$.  We note in advance that this invariant will depend on $A_0$ and $A$ only up to isogeny.  We give the definition separately in the cases that $\reschar v \neq \charac k$ and $\reschar v = \charac k$, where $\reschar v$ denotes the residue characteristic of $v$.

\begin{altenumerate}
\renewcommand{\theenumi}{\alph{enumi}}
\item \emph{$\reschar v \neq \charac k$.  Let $\ell := \reschar v$}, and let $V_\ell(A_0)$ and $V_\ell(A)$ denote the respective rational $\ell$-adic Tate modules of $A_0$ and $A$. Set
\[
   V_\ell(A_0, A) := \Hom_F\bigl(V_\ell(A_0), V_\ell(A)\bigr).
\]
Then $V_\ell(A_0, A)$ is a free $F \otimes_\BQ \BQ_\ell$-module of rank $n$. The polarizations $\lambda_0$ and $\lambda$ endow $V_\ell(A_0, A)$ with a non-degenerate hermitian form $h$, defined by
\[
   h(\phi_1,\phi_2) := \lambda_0\i \circ \phi_2^\vee \circ \lambda \circ \phi_1 \in \End_{F \otimes \BQ_\ell}\bigl( V_\ell(A_0, A) \bigr) \cong F \otimes \BQ_\ell.
\]
The decomposition $F_0 \otimes \BQ_\ell =  \prod_{v'\mid \ell} F_{0,v'}$, where $v'$ runs through the places of $F_0$ over $\ell$, induces a decomposition
\[
   V_\ell(A_0, A) = \bigoplus_{v'\mid \ell} V_{v'}(A_0, A),
\]
and each $V_{v'}(A_0, A)$ is a free $F_{v'}$-module of rank $n$.  By assumption $F_v$ is a field, and we define the invariant at $v$ as for any $n$-dimensional $F_v/F_{0,v}$-hermitian space in the main body of the paper \eqref{Hasse def},
\[
   \inv_v(A_0,\iota_0,\lambda_0,A,\iota,\lambda)^\nat := (-1)^{n(n-1)/2} \det V_v(A_0, A) \in F_{0,v}^\times / \Nm F_v^\times,
\]
where $\det V_v(A_0, A)$ is the class mod $\Nm F_v^\times$ of any hermitian matrix representing the component $h_v$.  We note that $\inv_v(A_0,\iota_0,\lambda_0,A,\iota,\lambda)^\nat$ is unchanged after any base change $k \to k'$.
\item \emph{$\reschar v = \charac k = p$}.  Let $\ov k$ be an algebraic closure of $k$.  Let $W = W(\ov k)$ denote the ring of Witt vectors of $\ov k$, and let $\sigma$ denote the Frobenius operator on $W$.  The decomposition $O_{F_0} \otimes \BZ_{p} =  \prod_{v'\mid  p} O_{F_0,v'}$ induces decompositions of $p$-divisible groups $A_0[p^\infty] = \prod_{v'\mid  p}A_0[v^{\prime \infty}]$ and $A[p^\infty] = \prod_{v'\mid  p}A[v^{\prime \infty}]$, and we denote by $M(A_0)_{\BQ,v}$ and $M(A)_{\BQ,v}$ the respective rational Dieudonn\'e modules of $A_0[v^\infty]$ and $A[v^\infty]$ over $W_\BQ$. Let $\uF_0$ and $\uF$ denote the respective Frobenius operators of $M(A_0)_{\BQ,v}$ and $M(A)_{\BQ,v}$.  Then $M(A_0)_{\BQ,v}$ and $M(A)_{\BQ,v}$ carry actions of $F_v$ which commute with the Frobenius operators and make $M(A_0)_{\BQ,v}$ and $M(A)_{\BQ,v}$ into free $F_v \otimes_{\BQ_p} W_\BQ$-modules of respective ranks $1$ and $n$. Furthermore $M(A_0)_{\BQ,v}$ is isoclinic of slope $1/2$.  We consider the internal Hom in the category of isocrystals with $F_v$-action,
\[
   M_{\BQ,v} := \Hom_{F_v \otimes W_\BQ}\bigl(M(A_0)_{\BQ,v}, M(A)_{\BQ,v}\bigr),
\]
which is a free $F_v \otimes_{\BQ_p}W_\BQ$-module of rank $n$.
Here, as for any internal Hom object, the Frobenius operator $\uF_{M_{\BQ,v}}$ on $M_{\BQ,v}$ sends $\phi \mapsto \uF \circ \phi \circ \uF_0\i$.  The polarizations $\lambda_0$ and $\lambda$ endow $M_{\BQ,v}$ with an $F_v \otimes_{\BQ_p} W_\BQ/F_{0,v} \otimes_{\BQ_p} W_\BQ$-hermitian form $h$, defined by
\[
   h(\phi_1,\phi_2) := \lambda_0\i \circ \phi_2^\vee \circ \lambda \circ \phi_1 \in \End_{F_v \otimes W_\BQ}\bigl(M(A_0)_{\BQ,v}\bigr) \cong F_v \otimes_{\BQ_p} W_\BQ.
\]
Let
\[
   N_{\BQ,v} := \bigwedge\nolimits_{F_v \otimes W_\BQ}^n M_{\BQ,v}.
\]
Then $N_{\BQ,v}$ is isoclinic of slope zero, and $h$ induces a hermitian form $\sform$ on it satisfying
\[
   (\uF_{N_{\BQ,v}} x, \uF_{N_{\BQ,v}} y) = (x,y)^\sigma, \quad x,y \in N_{\BQ,v}.
\]
For any element $x_0 \in N_{\BQ,v}$ fixed by $\uF_{N_{\BQ,v}}$, the class $(x_0,x_0) \in F_{0,v}^\times/\Nm F_v^\times$ is independent of the choice of $x_0$, and we define
\[
   \inv_v(A_0,\iota_0,\lambda_0,A,\iota,\lambda)^\nat := (-1)^{n(n-1)/2}(x_0,x_0) \in F_{0,v}^\times/\Nm F_v^\times
\]
for such $x_0$.
\end{altenumerate}

Now let $r\colon \Hom_\BQ(F,\ov\BQ) \to \BZ_{\geq 0}$ be a generalized CM type for $F$ of rank $n$, i.e., a function $\varphi \mapsto r_\varphi$ satisfying $r_\varphi + r_{\ov\varphi} = n$ for all $\varphi \in \Hom_\BQ(F,\ov\BQ)$, cf.\ \cite[Def.\ 2.1]{KR-new}.  Also, let $r_0$ be the \emph{opposite} of the canonical generalized CM type for $F$ of rank one attached to the CM type $\Phi$,
\begin{equation}\label{r0}
   r_{0,\varphi} = 
	\begin{cases}
		0,  &  \varphi\in \Phi;\\
		1,  &  \varphi\notin\Phi,
	\end{cases}
	\quad \varphi \in \Hom_\BQ(F,\ov\BQ).
\end{equation}
Let $E$ be the subfield of $\ov\BQ$ characterized by 
\begin{equation}\label{genE}
   \Gal(\ov\BQ/E) = \bigl\{\, \sigma\in\Gal(\ov\BQ/\BQ) \bigm| \sigma \circ\Phi=\Phi \text{ and } r_{\sigma\varphi}= r_\varphi \text{ for all } \varphi\in\Phi \,\bigr\}.
\end{equation}
Thus $E$ is the join of the reflex fields of $r$ and of $r_0$, in the sense of \cite[\S2]{KR-new}. Note that, in the situation of the main body of the paper, this definition of $E$ agrees with \eqref{defE}; but in constrast to the main body, in general $F$ need not admit an embedding into $E$. 

Recall from loc.\ cit.\ that a \emph{triple of CM type $r$} over an $O_E$-scheme $S$ is a triple $(A,\iota,\lambda)$ consisting of an abelian scheme $A$ over $S$, an action $\iota\colon O_F \to \End_S(A)$ satisfying the Kottwitz condition of type $r$, and a polarization $\lambda\colon A \to A^\vee$ such that $\Ros_\lambda$ induces on $O_F$, via $\iota$, the nontrivial Galois automorphism of $F/F_0$.
We denote by $\CM_{r_0,r}$ the stack over $\Spec O_E$ of tuples $(A_0,\iota_0,\lambda_0,A,\iota,\lambda)$, where $(A_0,\iota_0,\lambda_0)$ is a CM triple of type $r_0$ and $(A,\iota,\lambda)$ is a CM triple of type $r$.

Let $k$ be a field which is an $O_E$-algebra, and let $(A_0,\iota_0,\lambda_0,A,\iota,\lambda) \in \CM_{r_0,r}(k)$.  Again let $v$ be a finite place of $F_0$ which is non-split in $F$.  We are going to define the \emph{$r$-adjusted invariant} $\inv_v^r(A_0,\iota_0,\lambda_0,A,\iota,\lambda)$ (it depends on both $r_0$ and $r$). If the residue characteristic of $v$ is different from the characteristic of $k$, then we set
\begin{equation}\label{deftrmodsign}
   \inv_v^r(A_0,\iota_0,\lambda_0,A,\iota,\lambda) := \inv_v(A_0,\iota_0,\lambda_0,A,\iota,\lambda).
\end{equation}

Now suppose that the residue characteristic of $v$ is equal to the characteristic $p$ of $k$.
Let $\nu$ be the place of $E$ determined by the structure map $O_E \to k$, and let $\wt\nu\colon \ov\BQ \to \ov\BQ_p$ be an embedding which induces $\nu$. Let
\[
   \Phi_{\nu,v} := \bigl\{ \, \varphi \in \Phi \bigm| \wt\nu \circ \varphi|_{F_0} \text{ induces } v\, \bigr\}.
\]
Then the set $\Phi_{\nu,v}$ is independent of the choice of $\wt\nu$ inducing $\nu$, and using $\wt\nu$ we may identify
\begin{equation}\label{F_v embed}
   \Hom_{\BQ_p}(F_v,\ov\BQ_p) \simeq \Phi_{\nu,v} \sqcup \ov\Phi_{\nu,v}.
\end{equation}
Let
\[
   r_{\nu,v} := r|_{\Phi_{\nu,v} \sqcup \ov\Phi_{\nu,v}}.
\]
Then we define
\begin{equation}\label{sgn r}
   \sgn(r_{\nu,v}) := (-1)^{\sum_{\varphi\in \Phi_{\nu,v}} r_\varphi}
\end{equation}
and
\begin{equation}\label{defmodsign}
   \inv_v^r(A_0,\iota_0,\lambda_0,A,\iota,\lambda) := \sgn(r_{\nu,v}) \cdot \inv_v(A_0,\iota_0,\lambda_0,A,\iota,\lambda).
\end{equation}

The analog of \cite[App.]{KRZ} (which corrects  \cite[Prop.\ 3.2]{KR-new}) is now as follows.

\begin{proposition}\label{inv_v^r const}
Let $(A_0,\iota_0,\lambda_0,A,\iota,\lambda) \in \CM_{r_0,r}(S)$, where $S$ is a connected scheme over $\Spec O_E$. Then for every non-archimedean place $v$ of $F_0$ which is non-split in $F$, the function
\[
   s \mapsto \inv_v^r(A_{0,s},\iota_{0,s},\lambda_{0,s},A_s,\iota_s,\lambda_s)
\]
is constant on $S$.
\end{proposition}

\begin{proof} The proof is easy when the residue characteristic $ \ell$ of $v$ is invertible in $\CO_S$, in which case $\wedge^{\rm max} V_\ell(A_0, A)$ is a lisse \'etale sheaf on $S$, comp. the proof of  \cite[Prop.\ 3.2]{KR-new}.  A similar argument proves the assertion  when $S$ is a scheme over $\BF_p$ and $\ell = p$, cf.~\cite[Lem.\ 8.2.2]{KRZ}.\footnote{The corresponding passage in  the proof of  \cite[Prop.\ 3.2]{KR-new} is an over-simplification.}  
The remaining cases are reduced to the case $\ell = p$ and $S = \Spec O_L$, where $L$ is the completion of a subfield of $\ov\BQ_p$ which contains $E$ and such that its ring of integers $O_L$ is a discrete valuation ring with residue field $k = \ov\BF_p$. To compare the invariants at the generic and closed points of $S$, as in loc.\ cit.\ we are going to use $p$-adic Hodge theory. Let $A_L$ and $A_{0, L}$ denote the respective generic fibers of $A$ and $A_0$, and let $A_k$ and $A_{0,k}$ denote the respective special fibers.  Let $\nu$ denote the induced place of $E$.

We decompose the homomorphism module of the rational $p$-adic Tate modules of $ A_{0,L}$ and $ A_L$, resp.\ the homomorphism module of the rational Dieudonn\'e modules of $A_{0, k}$ and $A_k$, with respect to the actions of $F \otimes_\BQ \BQ_p \cong \prod_{w \mid p} F_w$,
\[
\begin{aligned}
   V_p &:= \Hom_{F \otimes \BQ_p}\bigl(V_p(A_{0, L}),V_p(A_L)\bigr)= \bigoplus\nolimits_{w\mid p} V_{w} ,\\
	\quad
	M_\BQ &:= \Hom_{F \otimes \breve \BQ_p}\bigl(M(A_{0,k})_\BQ,M(A_k)_\BQ\bigr) = \bigoplus\nolimits_{w\mid p} M_{\BQ,w}.
\end{aligned}	
\]
Here $w$ runs through the places of $F$, and we recall the notation $\breve \BQ_p = W(k)_\BQ$. Furthermore, for each place $w \mid p$, the summand $V_{w}$ is a free $F_{w}$-vector space of rank $n$, and $M_{\BQ,w}$ is a free $F_{w} \otimes_{\BQ_p} \breve \BQ_p$-module of rank $n$. At our distinguished place $v$, we set
\[
   S_v := \bigwedge\nolimits_{F_v}^n V_v
	\quad\text{and}\quad
	N_{\BQ,v} := \bigwedge\nolimits_{F_v \otimes \breve \BQ_p}^n M_{\BQ,v}.
\]
Then $S_v$ and $N_{\BQ,v}^{\text{Frob}=1}$ are one-dimensional $F_v$-vector spaces (the latter because $N_{\BQ,v}$ is isoclinic of slope zero) equipped with natural $F_v/F_{0,v}$-hermitian forms.  Our problem is to compare these hermitian spaces, which we will do via \cite[Prop.\ 1.20]{RZ}.

To explain the group-theoretic setup in our application of loc.\ cit., let $T$ be the torus over $\BQ_p$ which is the kernel in the exact sequence
\[
   1 \to T \to \Res_{F_v /\BQ_p} \BG_{m,F_v} \xra{\Nm} \Res_{F_{0,v} /\BQ_p} \BG_{m,F_{0,v}} \to 1.
\]
Then $H^1(\BQ_p,T) = F_{0,v}^\times/\Nm F_v^\times$. The spaces $S_v$ and $N_{\BQ,v}^{\text{Frob}=1}$ are natural $\BQ_p$-rational representations of $T$, and we may regard the isomorphisms of hermitian vector spaces $\Isom(N_{\BQ,v}^{\text{Frob}=1},S_v)$ as an \'etale sheaf on $\Spec \BQ_p$.
This is a $T$-torsor. To calculate its class $\cl(N_{\BQ,v}^{\text{Frob}=1},S_v)$ via loc.\ cit., we seek to express the filtered isocrystal $N_{\BQ,v}$ in the form $\CI(N_{\BQ,v}^{\text{Frob}=1})$ for an admissible pair $(\mu,b)$ in $T$, in the notation of \cite[1.17]{RZ}.

Since $N_{\BQ,v}$ is isoclinic of slope zero, we take $b \in T(\breve \BQ_p)$ to be the identity. To determine the cocharacter $\mu$, we need to identify the filtration on $N_{\BQ,v} \otimes_{\breve\BQ_p} \ov{\breve\BQ}_p$.  Choose any embedding $\ov\BQ_p \to \ov{\breve \BQ}_p$, and identify $\Hom_{\BQ_p}(F_v,\ov\BQ_p) \simeq \Phi_{\nu,v} \sqcup \ov\Phi_{\nu,v}$ as in (\ref{F_v embed}). By the Kottwitz condition, the filtration on $M(A_k)_{\BQ,v} \otimes_{\breve\BQ_p} \ov{\breve\BQ}_p \cong \bigoplus_{\varphi \in \Phi_{\nu, v} \sqcup \ov\Phi_{\nu,v}} M(A_k)_{\BQ,\varphi}$ is given by, for each $\varphi$,
\[
   M(A_k)_{\BQ, \varphi} \supset^{r_\varphi} \Fil_\varphi^1 \supset 0,
\]
where the displayed terms are in respective degrees $0$, $1$, and $2$, and the upper index on the first containment means that the cokernel is of dimension $r_\varphi$.  The filtration on $M(A_{0,k})_{\BQ,v} \otimes_{\breve\BQ_p} \ov{\breve\BQ}_p \cong \bigoplus_{\varphi \in \Phi_{\nu, v} \sqcup \ov\Phi_{\nu, v}} M(A_{0,k})_{\BQ, \varphi}$ is similarly given by, for each $\varphi$,
\[
   M(A_{0,k})_{\BQ,\varphi} \supset^{r_{0,\varphi}} \Fil_{0,\varphi}^1 \supset 0,
\]
where $r_{0,\varphi}$ is given in \eqref{r0}. The unique jump in this filtration occurs in degree $1-r_{0, \varphi}$. The filtration on the dual space $M(A_{0,k})_{\BQ,\varphi}^\vee$ therefore has unique jump in degree $r_{0, \varphi}-1$.  Therefore in the filtration on the one-dimensional space
\[
   N_{\BQ,\varphi} = \bigwedge\nolimits_{\ov{\breve\BQ}_p}^n M_{\BQ,\varphi} 
	   \cong \bigl(M(A_{0,k})_{\BQ,\varphi}^\vee\bigr)^{\otimes n} \otimes \bigwedge\nolimits_{\ov{\breve\BQ}_p}^n M(A_k)_{\BQ,\varphi},
\]
the unique jump occurs in degree $n - r_\varphi + n(r_{0,\varphi} -1) = nr_{0,\varphi} - r_\varphi$.

Now consider the natural identification
\[
   X_*(T) \cong \ker\bigl[\Ind_{F_{0,v}}^{F_v} \bigl(\Ind_{\BQ_p}^{F_{0,v}}\BZ\bigr) \rightarrow \Ind_{\BQ_p}^{F_{0,v}}\BZ \bigr] .
\]
The corresponding filtration on $N_{\BQ,\varphi}$ is then given by the cocharacter $\mu \in X_*(T)$ with
\[
   \mu_\varphi = n r_{0,\varphi} - r_\varphi, \quad \varphi \in \Phi_{\nu,v} \sqcup \ov\Phi_{\nu,v}.
\]
Then $N_{\BQ,\varphi} = \CI(N_{\BQ,v}^{\text{Frob}=1})$ for the above choices of $\mu$ and $b$.

Now, by $p$-adic Hodge theory, in the case of the abelian scheme $A$, there is a canonical isomorphism
\begin{equation}\label{comparison}
   V_p(A_L) \otimes_{\BQ_p} B_\crys \cong M(A_k)_{\BQ} \otimes_{\breve \BQ_p} B_\crys
\end{equation}
compatible with all structures on both sides (e.g.\ the Frobenii, the $F$-actions, and the polarization forms), cf.\ \cite{Fal,T}. Here $B_\crys$ is the crystalline period ring of Fontaine \cite{Fon}.  Moreover, after extension of scalars under the inclusion $B_\crys \subset B_\text{dR}$, this isomorphism is compatible with the filtrations on both sides.  Furthermore, there is an analogous isomorphism with $A_0$ in place of $A$.  Taking homomorphism modules on both sides between the $v$-components of (\ref{comparison}) and its analog for $A_0$, and then passing to top exterior powers, we obtain an isomorphism between free $F_v \otimes_{\BQ_p} B_\crys$-modules of rank one,
\[
   S_v \otimes_{\BQ_p} B_\crys \cong N_{\BQ,v} \otimes_{\breve \BQ_p} B_\crys,
\]
again compatible with all structures on both sides, and in particular with the hermitian forms.  Taking Frobenius-fixed elements in the zeroth filtration modules, we obtain an isometry of $F_v/F_{0,v}$-hermitian spaces,
\[
   S_v \cong \CF\bigl( \CI(N_{\BQ,v}^{\text{Frob}=1}) \bigr),
\]
where $\CF$ denotes Fontaine's functor from admissible filtered isocrystals to Galois representations.

We conclude that the class $\cl(N_{\BQ,v}^{\text{Frob}=1},S_v)$ is computed by the formula $\kappa(b) - \mu^\sharp$ in \cite[Prop.\ 1.20]{RZ}.  Here $\mu^\sharp$ denotes the image of $\mu$ in the coinvariants $X_*(T)_\Gamma$, where $\Gamma:= \Gal(\ov\BQ_p/\BQ_p)$.  Since $b$ is trivial, under the identification $X_*(T)_\Gamma \cong H^1(\BQ_p,T) \cong \BZ/2\BZ$, we obtain
\begin{equation}\label{musharp}
   \cl(N_{\BQ,v}^{\text{Frob}=1},S_v) = \mu^\sharp = \sum_{\varphi \in \Xi_v} \mu_\varphi ,
\end{equation}
where $\Xi_v$ is any half-system, i.e., $ \Xi_v \sqcup \ov\Xi_v \cong \Hom_{\BQ_p}(F_v,\ov\BQ_{p})$. In the formula for $\sgn (r_{\nu,v})$, we took $\Xi_v=\Phi_{\nu,v}$. 
\end{proof}

\begin{remark}\label{tired}
In the definition of the $r$-adjusted invariant above, we took the function $r_0$ in (\ref{r0}) to be the opposite of the canonical rank one function for $\Phi$ because this is what occurs in the moduli problem for $\CM_0^\fka$
 in the main body of the paper, cf.\ (\ref{KottwitzA0}).  Of course, we could have instead worked with respect to the canonical function (sending $\varphi \mapsto 1$ for $\varphi \in \Phi$ and $\varphi \mapsto 0$ for $\varphi \notin \Phi$), which is tantamount to replacing $\Phi$ by $\ov\Phi$.  In this case one defines the $r$-adjusted invariant at a place $v$ dividing $\charac k$ to be $(-1)^{\sum_{\varphi\in \ov\Phi_v} r_\varphi}\cdot\inv_v(A_0,\iota_0,\lambda_0,A,\iota,\lambda)$, and the statement and proof of Proposition \ref{inv_v^r const} for this version go through virtually without change.
\end{remark}

\begin{remark}\label{rm:pdivgps}
There is an obvious variant of the sign invariant for $p$-divisible groups. More precisely, let $L_0$ be a finite extension of $\BQ_p$, and let $L/L_0$ be a quadratic extension. Fix a local CM type $\Phi_L\subset \Hom_{\BQ_p}(L, \ov\BQ_p)$ for $L/L_0$, and let $r_L\colon \Hom_{\BQ_p}(L, \ov\BQ_p)\to \BZ_{\geq 0}$ be a local generalized CM type of rank $n$ for $L/L_0$, cf.\ \cite[\S5]{KR-new}. Let $k$ be a field of characteristic $p$ which is an $O_{E_{\Phi_L,r_L}}$-algebra, where $E_{\Phi_L,r_L} \subset \ov\BQ_p$ denotes the join of the reflex fields for $\Phi_L$ and for $r_L$. Let $(X_0, \iota_0, \lambda_0)$ and  $(X, \iota, \lambda)$ be  $p$-divisible groups over $k$ with actions by $O_L$ and quasi-polarizations whose associated Rosati involutions induce on $F$ the Galois conjugation over $L_0$. Assume that  $(X_0, \iota_0)$ is of CM type $\ov\Phi_L$ and that $(X, \iota)$ is of generalized CM type $r_L$, cf.\ \cite[\S5]{KR-new}. 
Then there is a sign invariant
\begin{equation}\label{locinveq}
   \inv^{r_L}(X_0,\iota_0,\lambda_0,X,\iota,\lambda) := \sgn(r_L) \cdot \inv(X_0,\iota_0,\lambda_0,X,\iota,\lambda) ,
\end{equation}
with  properties analogous to the case of abelian varieties. In fact, returning to all of our notation from the global setting, suppose that $k$ is an $O_E$-algebra, and let $(A_0,\iota_0,\lambda_0,A,\iota,\lambda) \in \CM_{r_0,r}(k)$.
Consider the decomposition of the corresponding $p$-divisible groups  induced by the decomposition $O_{F_0}\otimes\BZ_p=\prod_{v\mid p} O_{F_{0}, v}$, 
\[
   A_0[p^\infty]=\prod_{v\mid p}A_0[v^\infty] 
   \quad\text{and}\quad
   A[p^\infty]=\prod_{v\mid p}A[v^\infty] .
\]
Let $\nu$ denote the place of $E$ determined by $k$, and for a place $v \mid p$ which is non-split in $F$, take $L := F_v$, $\Phi_L := \Phi_{\nu,v}$, and $r_L := r_{\nu,v}$ (where we implicitly choose an identification $\Hom_{\BQ_p}(F_v,\ov\BQ_p) \simeq \Phi_{\nu,v} \sqcup \ov\Phi_{\nu,v}$ as above).
Then
$$
 \inv_v^r(A_0,\iota_0,\lambda_0,A,\iota,\lambda) = \inv^{r_{\nu,v}}\bigl(A_0[v^\infty],\iota_0[v^\infty],\lambda_0[v^\infty],A[v^\infty],\iota[v^\infty],\lambda[v^\infty]\bigr).
$$
\end{remark}

\section{Local models in the case of banal signature}\label{appendix}

In this appendix, we prove that local models attached to Weil restrictions of $\GL_n$ and $\GU_n$, defined using an analog of the Eisenstein condition of \cite{RZ14}, are trivial in the case of \emph{banal signature}.  Let $L/K$ be a finite separable extension of discretely valued henselian fields, with respective valuation rings $O_L$ and $O_K$.  Let $\pi$ be a uniformizer for $L$, and fix an algebraic closure $\ov K$ of $K$.  
The material in this appendix applies to the main body of the paper in the case $K = \BQ_p$ and $L = F_w$ for $w$ a $p$-adic place of the number field $F$.

\subsection{The $\GL_n$ case}
Let $n$ be a positive integer, and fix a function
\begin{equation}\label{r}
	\begin{gathered}
   r\colon
	\xymatrix@R=0ex{
	   \Hom_K(L,\ov K)\ar[r]  &  \{0,n\}\\
		\varphi \ar@{|->}[r]  &  r_\varphi.
	}
	\end{gathered}
\end{equation}
The \emph{reflex field} attached to $r$ is the fixed field $E_r \subset \ov K$ of the subgroup of the Galois group,
\[
   \bigl\{\, \sigma \in \Gal(\ov K/K) \bigm| r_{\sigma \circ \varphi} = r_\varphi \text{ for all } \varphi \in \Hom_K(L,\ov K)\,\bigr\}.
\]
Then $E_r$ is a finite extension of $K$, contained in the normal closure of $L$ in $\ov K$. Note that in contrast to the analogous global situation considered in the main body of the paper (with the particular choice of $r$ in \eqref{r main body}), $L$ need not admit an embedding into $E_r$. Let $\CL$ be a periodic $O_L$-lattice chain in $L^n$.

The \emph{local model} attached to the group $\Res_{L/K} \GL_n$, the function $r$, and the lattice chain $\CL$ is the scheme $M = M_{\Res_{L/K} \GL_n,r,\CL}$ over $\Spec O_{E_r}$ representing the following functor.  To each $O_{E_r}$-scheme $S$, the functor associates the set of isomorphism classes of families $(\Lambda \otimes_{O_K} \CO_S \surj \CP_\Lambda)_{\Lambda \in \CL}$ such that
\begin{altitemize}
\item for each $\Lambda$, $\CP_\Lambda$ is an $O_L \otimes_{O_K} \CO_S$-linear quotient of $\Lambda \otimes_{O_K} \CO_S$, locally free on $S$ as an $\CO_S$-module;
\item for each inclusion $\Lambda \subset \Lambda'$ in $\CL$, the arrow $\Lambda \otimes_{O_K} \CO_S \to \Lambda' \otimes_{O_K} \CO_S$ induces an arrow $\CP_\Lambda \to \CP_{\Lambda'}$;
\item for each $\Lambda$, the isomorphism $\Lambda \otimes_{O_K} \CO_S \xra[\undertilde]{\pi \otimes 1} (\pi \Lambda) \otimes_{O_K} \CO_S$ identifies $\CP_\Lambda \isoarrow \CP_{\pi\Lambda}$; and
\item for each $\Lambda$, $\CP_\Lambda$ satisfies the \emph{Kottwitz condition}
\begin{equation}\label{kottwitz cond lm}
   \charac_{\CO_S} (a \otimes 1 \mid \CP_\Lambda) = \prod_{\varphi \in \Hom_K(L,\ov K)} \bigr(T - \varphi(a)\bigr)^{r_\varphi}
	\quad\text{for all}\quad
	a \in O_L.
\end{equation}
\end{altitemize}
We further require that the family $(\Lambda \otimes_{O_K} \CO_S \surj \CP_\Lambda)_{\Lambda \in \CL}$ satisfies the (analog of the) \emph{Eisenstein condition} of \cite[(8.2)]{RZ14}, which in our case takes the following form.  Let $L^t$ denote the maximal unramified extension of $K$ in $L$.  We first formulate the condition when $S$ is an $O_{E_r \wt L^t}$-scheme, where $E_r \wt L^t$ is the compositum in $\ov K$ of $E_r$ and the normal closure $\wt L^t$ of $L^t$; by Lemma \ref{banal LM GL_n triv} below, the condition will descend over $O_{E_r}$ (and yield $M \isoarrow \Spec O_{E_r}$).  For each $\psi \in \Hom_K(L^t,\ov K)$, set
\begin{align*}
   A_\psi &:= \bigl\{\, \varphi \in \Hom_K(L,\ov K)\bigm| \varphi|_{F^t} = \psi \text{ and } r_\varphi = n \,\bigr\}, \\
	B_\psi &:= \bigl\{\, \varphi \in \Hom_K(L,\ov K)\bigm| \varphi|_{F^t} = \psi \text{ and } r_\varphi = 0 \,\bigr\}.
\end{align*}
Further set
\[
   Q_{A_\psi}(T) := \prod_{\varphi \in A_\psi} \bigl(T - \varphi(\pi)\bigr)
	\quad\text{and}\quad
	Q_{B_\psi}(T) := \prod_{\varphi \in B_\psi} \bigl(T - \varphi(\pi)\bigr).
\]
Then $Q_{A_\psi}$ and $Q_{B_\psi}$ are polynomials with coefficients in $O_{E_r\wt L^t}$.  Since we assume that $S$ is an $O_{E_r \wt L^t}$-scheme, there is a natural isomorphism
\begin{equation}\label{unram decomp}
   O_{L^t} \otimes_{O_K} \CO_S \isoarrow \prod_{\psi \in \Hom_K(L^t,\ov K)} \CO_S,
\end{equation}
whose $\psi$-component is $\psi \otimes \id$.
This induces a decomposition, for each $\Lambda$,
\begin{equation}\label{unram decompP}
   \CP_\Lambda \isoarrow \bigoplus_{\psi \in \Hom_K(L^t,\ov K)} (\CP_\Lambda)_\psi.
\end{equation}
The Eisenstein condition is that, for each $\Lambda$,
\begin{equation}\label{eis cond lm}
   Q_{A_\psi}(\pi \otimes 1)|_{(\CP_\Lambda)_\psi} = 0
	\quad\text{for all}\quad
	\psi \in \Hom_K(L^t,\ov K).
\end{equation}

To complete the definition of the moduli problem, an isomorphism from $(\Lambda \otimes_{O_K} \CO_S \surj \CP_\Lambda)_{\Lambda \in \CL}$ to $(\Lambda \otimes_{O_K} \CO_S \surj \CP'_\Lambda)_{\Lambda \in \CL}$ consists of an isomorphism $\CP_\Lambda \isoarrow \CP_\Lambda'$ for each $\Lambda$, compatible with the given epimorphisms in the obvious way.  Note that such an isomorphism is unique if it exists.

The main result is that the moduli scheme $M$ we have defined is trivial, in the following sense.

\begin{lemma}\label{banal LM GL_n triv}
Let $S$ be an $O_{E_r \wt L^t}$-scheme.  Then $M(S)$ consists of a single point.
\end{lemma}

\begin{proof}
It suffices to consider the case that $\CL$ consists of the homothety class of a single lattice $\Lambda$; the general case then follows immediately.  Let $\Lambda \otimes_{O_K} \CO_S \cong \bigoplus_{\psi \in \Hom_K(L^t,\ov K)} (\Lambda \otimes_{O_K} \CO_S)_\psi$ denote the decomposition induced by \eqref{unram decomp}.  Then the Eisenstein condition forces
\begin{equation}\label{P_psi}
   (\CP_\Lambda)_\psi = (\Lambda \otimes_{O_K} \CO_S)_\psi \big/ Q_{A_\psi}(\pi \otimes 1) \cdot (\Lambda \otimes_{O_K} \CO_S)_\psi,
\end{equation}
which completes the proof.
\end{proof}

As we have already noted, it follows by descent that $M \cong \Spec O_{E_r}$.  This also shows that the Eisenstein condition is independent of the choice of uniformizer $\pi$.

\begin{remark}\label{n=1}
In the special case $n = 1$, we may take $\Lambda = O_L$ in the proof of Lemma \ref{banal LM GL_n triv}, and then the Kottwitz condition already implies (\ref{P_psi}).  Thus the Eisenstein condition is redundant in this case.
\end{remark}

We also note that the Eisenstein condition is redundant in the unramified case, comp.\ \cite[Prop.~2.2]{RZ14}.

\begin{lemma}\label{kott=>eis}
Suppose that $L/K$ is unramified.  Then the Kottwitz condition \eqref{kottwitz cond lm} on $\CP_\Lambda$ implies the Eisenstein condition \eqref{eis cond lm}.
\end{lemma}

\begin{proof}
When $L=L^t$ is unramified over $K$, then  all sets $A_\psi$  have at most one element.  If $A_\psi$ is empty,  then  $(\CP_\Lambda)_\psi=0$ and the condition \eqref{eis cond lm} is empty. If $A_\psi$ is non-empty,  then   the condition \eqref{eis cond lm} is equivalent to the definition of the $\psi$-eigenspace in the decomposition \eqref{unram decompP}.
\end{proof}

\subsection{The unitary case}
In this subsection we assume that the residue characteristic $p$ of $K$ is not $2$.  We retain the setup of the previous subsection, and we assume in addition that $L$ is a quadratic extension of a field $L_0/K$.  Let $a \mapsto \ov a$ denote the nontrivial automorphism of $L/L_0$, and for each $\varphi \in \Hom_K(L,\ov K)$, define $\ov\varphi(a) := \varphi(\ov a)$.  We assume that the function $r$ in \eqref{r} satisfies $r_\varphi + r_{\ov\varphi} = n$ for all $\varphi \in \Hom_K(L,\ov K)$.  Furthermore, we endow $L^n$ with a non-degenerate $L/L_0$-hermitian form $h$, and we assume that the lattice chain $\CL$ is self-dual for $h$.  We define the alternating $K$-bilinear form $\aform \colon L^n \times L^n \to K$ as follows.  Let $\vartheta_{L_0/K}\i$ be a generator of the inverse different $\fkd_{L_0/K}\i$.   If $L/L_0$ is unramified, then choose an element $\zeta \in O_L^\times$ such that $\ov \zeta = -\zeta$ (since $p \neq 2$, such a $\zeta$ always exists), and set
\[
   \la x,y \ra := \tr_{L/K}\bigl(\vartheta_{L_0/K}\i\zeta h(x,y)\bigr), \quad x,y \in L^n.
\]
If $L/L_0$ is ramified, then choose the uniformizer $\pi$ to satisfy $\pi^2 \in L_0$ (since $p \neq 2$, such a $\pi$ always exists), and set
\[
   \la x,y \ra := \tr_{L/K}\bigl(\vartheta_{L_0/K}\i\pi\i h(x,y)\bigr), \quad x,y \in L^n.
\]
Then in both cases, the dual $\Lambda^*$ of an $O_L$-lattice $\Lambda$ in $L^n$ is the same with respect to $h$ as it is with respect to \aform.

The \emph{local model} attached to the group $\Res_{L_0/K} \GU(h)$, the function $r$, and the lattice chain $\CL$ is by definition the closed subscheme $M_{\Res_{L_0/K} \GU(h),r,\CL}$ of $M_{\Res_{L/K}\GL_n,r,\CL}$ defined by the additional condition
\begin{altitemize}
\item for each $\Lambda$, the perfect pairing
\[
   (\Lambda \otimes_{O_K} \CO_S) \times (\Lambda^* \otimes_{O_K} \CO_S) \xra{\aform\otimes \CO_S} \CO_S
\]
identifies $\ker[\Lambda \otimes_{O_K} \CO_S \surj \CP_\Lambda]^\perp$ with $\ker[\Lambda^* \otimes_{O_K} \CO_S \surj \CP_{\Lambda^*}]$.
\end{altitemize}

It is a trivial consequence of Lemma \ref{banal LM GL_n triv} that this additional condition is redundant, and that we again have the following.

\begin{lemma}\label{banal LM GU_n triv}
$M_{\Res_{L_0/K} \GU(h),r,\CL} \cong \Spec O_{E_r}$. \qed
\end{lemma}

\end{document}